\pdfoutput=1
\RequirePackage{ifpdf}
\ifpdf % We are running pdfTeX in pdf mode
\documentclass[pdftex]{sigma}
\else
\documentclass{sigma}
\fi

\newtheorem{thm}{Theorem}[section]
\newtheorem{cor}[thm]{Corollary}
\newtheorem{lem}[thm]{Lemma}
\newtheorem{prop}[thm]{Proposition}
 { \theoremstyle{definition}
\newtheorem{defn}[thm]{Definition}

\newtheorem{rem}[thm]{Remark}
\newtheorem{rems}[thm]{Remark}
 }

\numberwithin{equation}{section}

\newcommand{\Q}{\mathbb Q}
\newcommand{\C}{\mathbb C}

\newcommand{\Z}{\mathbb Z}
\newcommand{\N}{\mathbb N}
\DeclareMathOperator\ord{ord}
\DeclareMathOperator\lc{lc}
\DeclareMathOperator\Res{Res}
\DeclareMathOperator\Aut{Aut}
\DeclareMathOperator\pf{pf}
\DeclareMathOperator\GL{GL}
\DeclareMathOperator\cC{{\cal C}}
\DeclareMathOperator\cR{{\cal R}}
\DeclareMathOperator\cK{{\cal K}}
\DeclareMathOperator\cL{{\cal L}}
\DeclareMathOperator\tcR{\tilde{\cal R}}
\newcommand{\blambda}{{\boldsymbol\lambda}}

\newcommand{\bmu}{{\boldsymbol\mu}}

\newcommand{\binomE}{\genfrac(){0pt}{}}
\newcommand{\la}{\langle}
\newcommand{\ra}{\rangle}
\newcommand{\obinomE}{\genfrac\la\ra{0pt}{}}
\def\Gamm#1{\Gamma_{\!\!#1}}
\def\Gampq{\Gamma_{\!p,q}}
\def\Gamppqq{\Gamma_{\!p^2,q^2}}
\def\Gampqq{\Gamma_{\!p,q^2}}
\def\Gamppq{\Gamma_{\!p^2,q}}
\def\Gamphq{\Gamma_{\!p^{1/2},q}}
\def\Gamphqh{\Gamma_{\!p^{1/2},q^{1/2}}}
\def\Gampqh{\Gamma_{\!p,q^{1/2}}}
\def\Gampqt{\Gamma^+_{\!p,q,t}}

\begin{document}

\allowdisplaybreaks

\newcommand{\arXivNumber}{1408.0305}

\renewcommand{\thefootnote}{}

\renewcommand{\PaperNumber}{019}

\FirstPageHeading

\ShortArticleName{Multivariate Quadratic Transformations and the Interpolation Kernel}

\ArticleName{Multivariate Quadratic Transformations\\ and the Interpolation Kernel\footnote{This paper is a~contribution to the Special Issue on Elliptic Hypergeometric Functions and Their Applications. The full collection is available at \href{https://www.emis.de/journals/SIGMA/EHF2017.html}{https://www.emis.de/journals/SIGMA/EHF2017.html}}}

\Author{Eric M.~RAINS}
\AuthorNameForHeading{E.M.~Rains}

\Address{Department of Mathematics, California Institute of Technology, USA}
\Email{\href{mailto:rains@caltech.edu}{rains@caltech.edu}}

\ArticleDates{Received September 12, 2017, in final form February 27, 2018; Published online March 08, 2018}

\Abstract{We prove a number of quadratic transformations of elliptic Selberg integrals (conjectured in an earlier paper of the author), as well as studying in depth the ``interpolation kernel'', an analytic continuation of the author's elliptic interpolation functions which plays a major role in the proof as well as acting as the kernel for a Fourier transform on certain elliptic double affine Hecke algebras (discussed in a~later paper). In the process, we give a~number of examples of a~new approach to proving elliptic hypergeometric integral identities, by reduction to a~Zariski dense subset of a formal neighborhood of the trigonometric limit.}

\Classification{33D67; 33E05}

{\small \tableofcontents}

\renewcommand{\thefootnote}{\arabic{footnote}}
\setcounter{footnote}{0}

\section{Introduction}

In \cite{littlewood}, the author conjectured a number of multivariate integral identities, which could be viewed either as elliptic analogues of the vanishing results of \cite{vanish} (which in turn were Macdonald polynomial analogues of, e.g., the classical fact that the integral of a~Schur function over the orthogonal group is $0$ unless the corresponding partition has no odd parts), or as multivariate elliptic hypergeometric quadratic transformations. (In the latter interpretation, the conjectures were new even in the ``trigonometric'' ($q$-hypergeometric) limit \dots) The first purpose of the present note is to prove these conjectures, as well as give some partial explanation to some of the commonalities of structure shared by the seven main conjectures.

These conjectures were originally formulated by first guessing one
conjecture (and verifying it in various special cases), then using certain
symmetries to transform the given conjecture into qualitatively quite
different forms. Unfortunately, for the most part, those symmetries did
not actually apply to the specific contour integrals being considered, but
rather only to certain degenerations of those integrals to finite sums. As
a result, even though the main conjectures form a single orbit under the
symmetries, it would seem that we need to prove each conjecture
independently. This impression is especially reinforced by the fact that
in many cases, the conjectures actually involved two qualitatively
different finite degenerations (corresponding to the fact that, unlike most
elliptic hypergeometric integral identities, these quadratic
transformations typically break the symmetry between $p$ and $q$). As a
result, the formulation of the conjectures involved several steps in which
the author had to essentially guess half of the identity. On the other
hand, two special cases of the conjectures were proved in~\cite{vandeBultFJ:2011}; since both of those were separated by at least one
such guess from the original conjecture, this suggested that there might in
fact be more structure to the symmetries than was apparent.

Note that although the conjectures had multivariate quadratic
transformations as special cases, the full version also incorporated
``representation theory''-ish information. For the analogues of results of
\cite{vanish}, this involved introducing one of the biorthogonal functions
of \cite{bctheta, xforms} into the integrand; similarly, the full quadratic
transformations used the corresponding ``interpolation functions''. One
difficulty in working with the symmetries is that the biorthogonal and
interpolation functions are not in general elliptic; rather, they are
products of two functions, each of which lives on a different elliptic
curve. Thus even if we knew one of the conjectures in the elliptic case,
this would still not suffice to prove the conjecture in general. It turns
out, however, that there is a way to finesse this issue. In addition to
five continuous parameters and $n$ variables, the interpolation functions
also depend on a pair of partitions (one for each elliptic factor). It
turns out that this discrete family of functions can instead be obtained as
a discrete family of specializations of a single function (the
``interpolation kernel'') depending on four continuous parameters and two
sets of $n$ variables; for each pair of partitions, there is a
corresponding 1-parameter family of specializations for the second set of
variables giving the corresponding interpolation function.

Thus, rather than attempt to prove the various quadratic transformations in
their original versions, we could instead try to prove the analogous
identities involving this kernel. Although at first glance, this does not
appear to be any easier (though it does, at least, make the expressions of
the identities somewhat simpler), a surprising thing happens that allows us
to reduce to simpler cases. It turns out that under fairly weak
conditions, an identity involving the interpolation kernel is actually {\em
 equivalent} to the specialization not just to interpolation functions,
but to {\em elliptic} interpolation functions. In fact, in many cases, it
is even equivalent to an identity between sums of elliptic functions.

This is a special case of a general principle, which is likely to be useful
in proving a wide variety of elliptic hypergeometric integral identities.
(Indeed, an important secondary objective of the present note is to
introduce this technique, and give some examples of various ways it can be
applied.) Suppose we are given an elliptic hypergeometric integral
that becomes an evaluation under some suitable limit as $p\to 0$. (E.g.,
most of the quadratic transformations we wish to prove have limits in which
both sides are integrals over the Koornwinder density.) In particular,
since the integral is holomorphic at $p=0$, we may expand it in a power
series (or, more likely, a Puiseux series) in $p$ around that point. In
many cases, we can then show that the coefficients of that Puiseux series
are rational functions in the remaining parameters, possibly after dividing
by the constant term. We can then prove equality of two such integrals by
showing that those rational functions agree. In particular, it suffices to
prove the identity for a {\em Zariski dense} set of points (i.e., such that
any nonzero polynomial is nonzero somewhere on the set), rather than
needing equality on a set with a limit point (or something analogous if we
have specialized multiple parameters). For instance, most elliptic
hypergeometric integrals that have been considered in the literature have
degenerations coming from residue calculus in which they become finite
sums. This typically involves a degree of freedom being specialized to a
power of $q$; since the parameters generally live on $\C^*$, the powers of
$q$ do not even have a limit point, but are certainly Zariski dense. (The
specializations that turn the interpolation kernel into elliptic
interpolation functions are similarly dense.) This reduction to finite
identities is, of course, particularly powerful when the identity of
interest was formulated as the integral analogue of a sum.

\looseness=-1 As an example of this, consider the elliptic beta integral
\cite{SpiridonovVP:2001}.\footnote{Note that (as the referee pointed out)
 Spiridonov's original proof of the elliptic beta integral evaluation also
 involved showing an identity of Taylor series coefficients, but needed a
 more complicated set of special cases that had a limit point.} We would
like to be able to deduce the elliptic beta integral evaluation from the
corresponding identity for elliptic hypergeometric sums (the Frenkel--Turaev
sum \cite{FrenkelIB/TuraevVG:1997}), but the corresponding subset of
parameter space (where two of the parameters multiply to $q^{-m}$) is
nowhere dense (and, indeed, we could add any holomorphic function to the
right-hand side without affecting the validity of the Frenkel--Turaev
limits). However, if we rescale two of the parameters by $p^{1/2}$, then
the evaluation identity becomes the normalization for the Askey--Wilson
density in the limit $p\to 0$, and thus we may divide each side by its
respective limit without affecting the validity of the result. If we take
the Puiseux series expansion of each side around $p=0$, then the
coefficients on each side are rational functions in the parameters.
Indeed, for the right-hand side, this is easy (simply show the
corresponding statement for the logarithm). We may interpret the left-hand
side as the integral against the {\em normalized} Askey--Wilson density
(i.e., normalized to have integral 1) of the ratio of integrands. The
ratio of integrands is itself a product, so can be seen to have rational
function coefficients, and those rational functions are moreover symmetric
Laurent polynomials in the integration variable. We may thus integrate
term-by-term, and reduce to showing that the integral of a symmetric
Laurent polynomial against the normalized Askey--Wilson density is a
rational function of the parameters. But this is easy, as we can compute
such an integral as the constant term of the expansion into Askey--Wilson
polynomials, and those certainly have rational function coefficients. If
two rational functions agree whenever $t_0t_1=q^{-m}$ for all $m$, then
they agree identically, and thus the Frenkel--Turaev sum does, in the end,
imply the elliptic beta integral.

Once we begin thinking about the Puiseux series associated to an integral,
we find that there are other possible ways to establish equalities between
such series. Here the additional observation we make is that although the
Puiseux series we obtain are, of course, convergent, we never use that fact
in establishing the identities. Once we think of the problem as
establishing an identity between {\em formal} Puiseux series, we see that
there is, for instance, no need to restrict our attention to convergent
formulas for those series. For instance, one of the key ideas in
implementing one of the symmetries we need for the quadratic
transformations is the observation that as long as we are willing to forego
convergence, we can replace the dimension of the integrals by a continuous
parameter (in such a way that the finite-dimensional case is Zariski
dense). This analytic continuation then has additional symmetries not
preserving the finite-dimensional specialization, allowing us to prove some
of the trickier quadratic transformations. Other approaches used below
involve writing the Puiseux series as a formally convergent infinite sum
(sidestepping the otherwise formidable obstacles to making sense of
nonterminating elliptic hypergeometric sums), or showing that the series
satisfies a family of difference equations having a unique formal solution.

Although we initially introduced the interpolation kernel as merely a tool
for using the above formal method to analytically continue results
involving interpolation functions, we quickly found that it had a number of
remarkable properties justifying studying it in its own right, bringing us
to the main purpose of the present note beyond proving the quadratic
transformations. One key property we will make only limited use of at
present (but use in a variant form in~\cite{elldaha}) is the fact that the
interpolation kernel can be viewed as the kernel of an integral
transformation, a sort of multivariate elliptic generalization of the
Fourier--Laplace transform. Just as the usual Fourier transformation acts
on the algebra of differential operators, this transformation acts on a
certain algebra of elliptic difference operators; among other things, it
preserves (up to an additive scalar and an action on the parameters) the
van Diejen Hamiltonian \cite{vDiejenJF:1994} (as well as the filtered
algebra of operators commuting with this integrable Hamiltonian). Another
application of the kernel is that we can use it in ``Bailey lemma''-type
arguments; thus, for instance, we will be able to directly derive the
$W(E_7)$-symmetry of the order 1 elliptic Selberg integral
\cite[Theorem~9.7]{xforms} from the kernel analogue (the ``braid relation'')
of the much simpler Theorem 9.2 of \cite{xforms}. With this in mind, we
will devote a fair amount of space in the present note to developing a
theory of this kernel.

We will take as our primary definition of the interpolation kernel a
certain sum of interpolation functions, generalizing the Cauchy identity of
\cite{littlewood}. A priori, this sum is nonconvergent (indeed, we cannot
even avoid poles of terms without excluding a dense set!), but this is
avoided if instead we take a Puiseux series in $p$ and only insist on
formal convergence, which turns out to hold as long as the remaining
parameters behave as suitable powers of $p$. We then find without much
difficulty that the coefficients of the resulting series are rational
functions in the variables and parameters, with well-controlled poles. If
we specialize one of the sets of variables to make the sum terminate (a
Zariski dense set of specializations), we find that the result can be
evaluated as an interpolation function. (This function now lives on a
formal elliptic curve, a.k.a.~a ``Tate curve'' \cite{TateJ:1995}.) Thus
various results about interpolation functions have immediate consequences
for this formal kernel. In particular, we obtain an identity to the effect
that a certain integral operator built from this kernel has a very nice
action on interpolation functions.

\looseness=-1 As it stands, the formal kernel would be of only limited usefulness. There
are, however, two important ways in which one can extend this kernel. The
first involves the fact that one can express the $n$-dimensional formal
kernel as an integral involving the $(n-1)$-dimensional formal kernel
(integrating term-by-term), and thus in general as an
$n(n-1)/2$-dimensional integral. This expression turns out to be {\em
 analytically} convergent on a large open set of parameters, and by
general considerations of \cite[Section~10]{xforms} extends to a meromorphic
function on parameter space, giving us the true interpolation kernel. In
other words, although the infinite sum defining the formal kernel is not
analytically convergent, the limiting formal power series often converges.
And, of course, the analytic interpolation kernel inherits the identities
of the formal kernel! It also has additional symmetries of its own; in
addition to the symmetry between $p$ and $q$ that those familiar with known
elliptic hypergeometric integrals might expect (but which is meaningless
for the formal kernel), there is also a symmetry between $t$ and $pq/t$.
We also find that there are values of the parameter outside the domain of
formal convergence in which we recover known operators: not only the
integral operators of \cite{xforms}, but the difference operators as well.
These difference operators in turn extend to a family of formal difference
operators; in this case, not formal in their coefficients (which are
functions on a general elliptic curve) but in that they are formal series
in the set of possible shifts. (This means that the operators cannot be
applied to actual functions, but as we will see does not cause any
particular difficulty in multiplying (and dividing!) them.) These formal
operators will not play a direct role in the present work, but in~\cite{elldaha} play an important role in understanding a certain algebra of
(actual) difference operators, by giving an alternate approach to the view
of the integral operators as generalized Fourier transformations that
avoids having to deal with convergence and boundary effects.

The other important way in which we can extend the formal kernel is that we
can analy\-ti\-cal\-ly continue in the dimension. More precisely, the
coefficients of the $n$-dimensional formal kernel are symmetric Laurent
polynomials in the two sets of variables, and in each case can be expressed
as a suitable specialization of a certain symmetric function depending only
mildly on~$n$. Again, this gives rise to additional symmetries, the two
main ones being an action of the Macdonald involution and a certain
plethystic symmetry. The key ingredient in this construction is a~corresponding symmetric function version of the interpolation functions,
which (once we start thinking in terms of formal series) is a~straightforward consequence of identities of \cite{littlewood}. Although
it would be surprising if either symmetric function analogue converged
analytically, we can still use it to prove identities: simply check the
identity on a Zariski dense set of points, then specialize to a case where
the series converges.

As we mentioned above, the expression of the formal kernel as a sum can be
viewed as a~deformation of a nonterminating version of the elliptic Cauchy
identity of \cite{littlewood}. That paper also established an elliptic
analogue of the Littlewood identity, which also immediately extends to a
nonterminating formal identity. It turns out that the resulting sum also
has a 1-parameter deformation which sums to the Puiseux series of a certain
meromorphic function. This also arises in a straightforward way from the
kernel analogue of Conjecture L1 of \cite{littlewood} (not one of the main
sequence of quadratic transformations). This transformation (which is
relatively straightforward to prove given the machinery we develop for the
interpolation kernel) has a special case in which the symmetry becomes a
continuous one; the resulting function with a one-parameter family of
integral representations also has a formal expression as a deformed
Littlewood sum. A~``Bailey lemma''-type argument turns this family of
integral representations into a transformation, only one side of which
involves this ``Littlewood kernel''. The form of the right-hand side turns
out to appear in two other conjectures from \cite{littlewood}, which can
thus be interpreted as describing certain degenerations of this function.
Moreover, the Littlewood kernel has a special case (i.e., when the deformed
Littlewood sum is not actually deformed) that can be expressed as a
product, and thus gives rise to its own set of identities along the lines
of \cite{littlewood}. In fact, this identity, together with the $t\mapsto
pq/t$ symmetry as well as the various other symmetries used in
\cite{littlewood}, is what will let us prove the main quadratic
transformations.

The plan of the paper is as follows. Apart from a discussion of notation
and reminders of the interpolation functions in the remainder of this
introduction, we will begin in Section~\ref{section2} by defining the formal kernel and
proving some initial properties, culminating in the main integral
representation. In Section~\ref{section3}, we will then use this to define the full
analytic version of the kernel, and establish a number of its properties
and special cases. Section~\ref{section4} (which will not be used elsewhere in the
present work) uses the kernel to construct a certain family of formal
difference operators and again consider their main properties. The main
result (Theorem~\ref{thm:formal_diff_En}) of this section is that these
operators can be used to construct twisted representations of a certain
sequence of Coxeter groups; the one nontrivial braid relation in this
interpretation turns out to be the difference operator of the main identity
satisfied by the kernel (Proposition \ref{prop:kern_braid}), which we thus
refer to as ``the'' braid relation. Using these operators, we also make
precise a weak version (Theorem \ref{thm:fourier} below) of the fact that
the kernel is the kernel of a generalized Fourier transformation; see~\cite{elldaha} for a stronger version of this fact.

Section~\ref{section5} constructs the symmetric function version of the formal kernel,
as well as reminding the reader of some of the properties (especially
duality) of the corresponding analogue of the Koornwinder integral that
will be needed in order to apply this kernel. Of particular note here is
Lemma \ref{lem:koorn_On_disc}, which shows how a certain $0/0$ issue in
degenerating these integrals results in their expansion as a sum of two
finite-dimensional integrals, in particular explaining why Conjecture~Q5 of
\cite{littlewood} involved such sums.

Section~\ref{section6} proves the kernel version of Conjecture L1 of \cite{littlewood},
and as mentioned above uses this to construct the Littlewood kernel and
understand a number of identities satisfied by the kernel. (One notable
special case is an elliptic analogue (Theorem \ref{thm:funny_pfaffian}) of
Conjecture 1 of \cite{BeteaD/WheelerM/Zinn-JustinP:2015}, which expressed a
certain deformation of the usual Littlewood identity for Macdonald
polynomials as a pfaffian related to the 6-vertex model.) Section~\ref{section7} proves
Conjectures~L2 and~L3 of~\cite{littlewood}, and studies the corresponding
analogues of the Littlewood kernel (the dual Littlewood and Kawanaka
kernels, respectively). In Section~\ref{section8}, we use the machinery developed for
these various kernels to finally prove the remaining conjectures of
\cite{littlewood}, the promised multivariate quadratic transformations, and~consider a few new transformations that arise by viewing these as
statements about degenerations of the Littlewood and other kernels. We
finish with an appendix of sorts that uses properties of the interpolation
polynomials of \cite{OkounkovA:1998,bcpoly} to establish that certain
difference and integral equations have unique polynomial solutions. (This
will then imply that certain equations with formal coefficients have unique
formal solutions, which will in turn be used in proving Theorems~\ref{thm:funny_pfaffian} and~\ref{thm:Q1} below.)

{\bf Notation.} As in \cite{littlewood}, we will be using the notation of
\cite{bctheta} and \cite{xforms}. In particular, bold-face greek letters
refer to pairs of partitions; if only one of the partitions is nonzero, we
will either give the partition pair explicitly, or rewrite using the
notation of \cite{bctheta}, explicitly breaking the symmetry between~$p$
and~$q$. Thus the interpolation functions are denoted by
\begin{gather*}
\cR^{*(n)}_{\blambda}(z_1,\dots,z_n;a,b;t;p,q),
\end{gather*}
which factors as
\begin{gather*}
\cR^{*(n)}_{\lambda,\mu}(z_1,\dots,z_n;a,b;t;p,q)
=
R^{*(n)}_{\lambda}(z_1,\dots,z_n;a,b;p,t;q)
R^{*(n)}_{\mu}(z_1,\dots,z_n;a,b;q,t;p),
\end{gather*}
with the first factor $q$-elliptic, and the second $p$-elliptic. (In fact,
we will nearly always be using the elliptic notations, as in the vast
majority of the cases in which we would want to use the full versions, we
will be using the kernel instead! We will also use a similar notation for
the biorthogonal functions, but these will only appear briefly in certain
corollaries not otherwise used.) Relations and operations on single
partitions extend to partition pairs in the obvious way; in particular,
$\blambda\subset\bmu$ denotes the product of the usual inclusion orders on
the two pieces. We will need some additional notations for partitions. Of
particular importance are $\lambda^2$, denoting the partition with
$\lambda^2_i = \lambda_{\lceil i/2\rceil}$, and $2\lambda$, denoting the
partition with $(2\lambda)_i=2\lambda_i$, both extending immediately to
partition pairs. (The latter will appear in the form
$(1,2)(\lambda,\mu)=(\lambda,2\mu)$.) We will also find it convenient to
let $\vec{z}$ denote the tuple $z_1,\dots,z_n$ of arguments, and similarly
for~$\vec{x}$,~$\vec{y}$, etc.

We specifically recall the elliptic Gamma function
\begin{gather*}
\Gampq(z)
:=
\prod_{0\le i,j} \frac{1-p^{i+1} q^{j+1}/z}{1-p^i q^j z},
\end{gather*}
with the convention here (and for $\Gamma^+$, $\theta$, etc.) that multiple
arguments express a product:
\begin{gather*}
\Gampq(x_1,\dots,x_n)
=
\prod_{1\le i\le n} \Gampq(x_i).
\end{gather*}
This satisfies the functional equations
\begin{gather*}
\Gampq(qz) = \theta_p(z)\Gampq(z), \qquad \Gampq(pz)= \theta_q(z)\Gampq(z), \qquad \Gampq(pq/z)= \Gampq(z)^{-1},
\end{gather*}
where
\begin{gather*}
\theta_p(z) := \prod_{0\le i} (1-p^i z)(1-p^{i+1}/z)
\end{gather*}
is a theta function ($\theta_p(\exp(2\pi i x))$ is doubly quasiperiodic), as well as the ``quadratic'' functional equations
\begin{gather}\label{eq:gampqq}
\Gampq(z) = \Gampqq(z,qz),\qquad \Gamppqq\big(z^2\big)= \Gampq(z,-z),
\end{gather}
which will be useful below. The special values
\begin{gather*}
\Gampqq(q) = \frac{1}{\big(q;q^2\big)} = \frac{\big(q^2;q^2\big)}{(q;q)} = (-q;q),\qquad
\Gampq(-1) = \frac{\big(p;p^2\big)\big(q;q^2\big)}{2},\\
\lim_{x\to 1} (1-x)\Gampq(x) = \frac{1}{(p;p)(q;q)}
\end{gather*}
will arise as well, in the process of various (omitted) simplifications. We will also make brief use of the triple gamma function
\begin{gather*}
\Gampqt(z):=\prod_{0\le i,j,k} \big(1-p^{i+1}q^{j+1}t^{k+1}/z\big)\big(1-p^iq^jt^k z\big),
\end{gather*}
with functional equations
\begin{gather*}
\Gampqt(tz) = \Gampq(z)\Gampqt(z),\qquad \Gampqt(pqt/z) = \Gampqt(z),
\end{gather*}
and so forth.

We will also need two families of densities. The simpler of the two is the elliptic Dixon density
\begin{gather*}
\Delta^{(n)}_D(\vec{z};p,q)=
\frac{((p;p)(q;q))^n}{2^n n!}
\prod_{1\le i<j\le n}
\frac{1}{\Gampq(z_i^{\pm 1}z_j^{\pm 1})}
\prod_{1\le i\le n}
\frac{1}{\Gampq(z_i^{\pm 2})}
\frac{dz_i}{2\pi\sqrt{-1}z_i},
\end{gather*}
which will also appear in a form with ``univariate'' parameters,
\begin{gather*}
\Delta^{(n)}_D(\vec{z};u_0,\dots,u_{2m+2n+3};p,q)=\prod_{\substack{1\le i\le n\\0\le r<2m+2n+4}} \Gampq\big(u_r z_i^{\pm 1}\big)\Delta^{(n)}_D(\vec{z};p,q).
\end{gather*}
We allow $m$ negative here, but note that it in general measures the complexity of the integral; e.g., for $m=0$, the integral of this density has an explicit evaluation (\cite[Corollary~3.2]{xforms}, originally conjectured in \cite{vanDiejenJF/SpiridonovVP:2001} as the ``Type~I'' integral; note that the univariate case of both this and the elliptic Selberg evaluation is the elliptic beta integral of~\cite{SpiridonovVP:2001}):
\begin{gather*}
\int_{C^n}\Delta^{(n)}_D(\vec{z};u_0,\dots,u_{2n+3};p,q)=\prod_{0\le r<s<2n+4} \Gampq(u_ru_s),
\end{gather*}
subject to the ``balancing'' condition $\prod_r u_r = pq$. The contour here must separate the double geometric progressions of poles converging to~0 from those converging to $\infty$; if $|u_r|<1$ for all~$r$, we may take the contour to be the unit circle. (This contour may not always exist, but
by general considerations \cite[Section~10]{xforms} such an integral always gives a well-defined meromorphic function.)

The other density we will need is the elliptic Selberg density,
\begin{gather*}
\Delta^{(n)}_S(\vec{z};t;p,q)=\Gampq(t)^n \prod_{1\le i<j\le n} \Gampq\big(t z_i^{\pm 1}z_j^{\pm 1}\big)\Delta^{(n)}_D(\vec{z};p,q).
\end{gather*}
(The ratio between the two densities will also appear quite a few times below.) Again, this will typically be given additional parameters
\begin{gather*}
\Delta^{(n)}_S(\vec{z};u_0,\dots,u_{2m+5};t;p,q)=\prod_{\substack{1\le i\le n\\0\le r<2m+6}} \Gampq\big(u_r z_i^{\pm 1}\big)\Delta^{(n)}_S(\vec{z};t;p,q).
\end{gather*}
This has the evaluation (\cite[Theorem~6.1]{xforms}, originally conjectured in~\cite{vanDiejenJF/SpiridonovVP:2000})
\begin{gather}
\int_{C^n}\Delta^{(n)}_S(\vec{z};u_0,\dots,u_5;t;p,q)=\prod_{0\le i<n} \Gampq\big(t^{i+1}\big) \prod_{0\le r<s<6} \Gampq\big(t^i u_r u_s\big),
\label{eq:ell_Selberg_eval}
\end{gather}
now with balancing condition
\begin{gather*}
t^{2n-2}\prod_{0\le r<6} u_r = pq.
\end{gather*}
The $t$-dependent factors of the density force us to insist that $C$ contains $tC$; again, as long as $u_r$ and $t$ are inside the unit circle, there is no difficulty taking $C=S^1$. There is also a~trans\-for\-ma\-tion for order $m=1$, \cite[Theorem~9.7]{xforms}.

(As an aside, we note that both of the above evaluations are special cases of the analytic form of Proposition \ref{prop:kern_braid} below; similarly, the transformation of the order 1 elliptic Selberg integral is a special case of Theorem~\ref{thm:bailey_xform} below.)

Note that for either density, if two parameters multiply to $pq$, then the reflection relation of the elliptic gamma function causes the two parameters to cancel out, thus reducing the order by~1.

Connected to the elliptic Selberg density is a family of difference operators satisfying (formal) adjointness relations with respect to that
density. The simplest such operator, $D^{(n)}_q(t;p)$, is self-adjoint with respect to the $0$-parameter elliptic Selberg density, and acts on
hyperoctahedrally symmetric functions by
\begin{gather*}
\big(D^{(n)}_q(t;p)f\big)(z_1,\dots,z_n)=
\sum_{\vec{\sigma}\in \{\pm 1\}^n}
\frac{\prod\limits_{1\le i<j\le n} \theta_p\big(t z_i^{\sigma_i}z_j^{\sigma_j}\big)}
 {\prod\limits_{1\le i\le j\le n} \theta_p\big(z_i^{\sigma_i}z_j^{\sigma_j}\big)}
f\big(q^{\sigma_1/2}z_1,\dots,q^{\sigma_n/2}z_n\big).
\end{gather*}
More generally, we define an operator
\begin{gather*}
D^{(n)}_q(u_1,\dots,u_{2m+2};t;p)=
\prod_{\substack{1\le i\le n\\1\le j\le 2m+2}} \frac{1}{\Gampq(u_j z_i^{\pm 1})}
D^{(n)}_q(t;p)
\prod_{\substack{1\le i\le n\\1\le j\le 2m+2}} \Gampq\big(q^{1/2} u_j z_i^{\pm 1}\big);
\end{gather*}
this has the effect of multiplying the term corresponding to $\vec\sigma$ by $\prod\limits_{1\le i\le n,1\le j\le 2m+2} \theta_p(u_j z_i^{\sigma_i})$. In the case of ambiguity regarding the variables on which a given difference operator acts, we will specify those variables as a subscript, as $D^{(n)}_q(t;p)_{\vec{z}}$.

We will also need some finite products. The factors
\begin{gather*}
\Delta^0_\blambda(a|b_0,\dots,b_{n-1};t;p,q)\qquad\text{and}\qquad \Delta_\blambda(a|b_0,\dots,b_{n-1};t;p,q)
\end{gather*}
that appear below are certain multivariate $q$-symbols (see the introduction of \cite{xforms}). The first is defined by
\begin{gather*}
\Delta^0_\blambda(a|b_0,\dots,b_{n-1};t;p,q)=\prod_{0\le r<n} \frac{\cC^0_\blambda(b_r;t;p,q)} {\cC^0_\blambda(pqa/b_r;t;p,q)},
\end{gather*}
where
\begin{gather*}
\cC^0_\blambda(x;t;p,q):=\prod_{1\le i} \theta\big(t^{1-i}x;p,q\big)_{\blambda_i},
\end{gather*}
and
\begin{gather*}
\theta(x;p,q)_{l,m}:=\prod_{0\le j<l} \theta_q\big(p^j x\big)\prod_{0\le j<m} \theta_p(q^j x)=
(-x)^{lm}p^{ml(l-1)/2}q^{lm(m-1)/2}\frac{\Gampq\big(p^l q^m x\big)} {\Gampq(x)}.
\end{gather*}
Note that
\begin{gather*}
\Delta^0_{\lambda,\mu}(a|b_0,\dots,b_{n-1};t;p,q)=\Delta^0_{\lambda,0}(a|b_0,\dots,b_{n-1};t;p,q)
\Delta^0_{0,\mu}(a|b_0,\dots,b_{n-1};t;p,q),
\end{gather*}
and if $n=2m$, $\prod\limits_{0\le r<2m} b_r = (pqa)^m$, then both factors are elliptic subject to this constraint; i.e.,
\begin{gather*}
\Delta^0_{0,\mu}(a|b_0,\dots,b_{2m-1};t;p,q)
\end{gather*}
is invariant under shifting the parameters by integer powers of $p$ such that the balancing condition remains satisfied.

The other $\Delta$-symbol is more complicated:
\begin{gather*}
\Delta_\blambda(a|b_0,\dots,b_{n-1};t;p,q):=
\Delta^0_\blambda(a|b_0,\dots,b_{n-1};t;p,q)
\frac{\cC^0_{2\blambda^2}(pqa;t;p,q)}
 {\cC^-_\blambda(pq,t;t;p,q)\cC^+_\blambda(a,pqa/t;t;p,q)},
\end{gather*}
where
\begin{gather*}
\cC^-_\blambda(x;t;p,q)
:=\prod_{1\le i\le j}
\frac{\theta\big(t^{j-i}x;p,q\big)_{\blambda_i-\blambda_{j+1}}}
 {\theta\big(t^{j-i}x;p,q\big)_{\blambda_i-\blambda_j}},\\
\cC^+_\blambda(x;t;p,q)
:=\prod_{1\le i\le j}
\frac{\theta\big(t^{2-i-j}x;p,q\big)_{\blambda_i+\blambda_j}}
 {\theta\big(t^{2-i-j}x;p,q\big)_{\blambda_i+\blambda_{j+1}}}.
\end{gather*}
The key property of $\Delta_{\blambda}$ is that the $\blambda$-dependent factor of the residue of the elliptic Selberg integrand $\Delta^{(n)}$ at
the point $(\dots,(p,q)^{\blambda}t^{n-i}u_0,\dots)$ is
\begin{gather*}
\Delta_{\blambda}\big(t^{2n-2}u_0^2|t^n,t^{n-1}u_0u_1,\dots,t^{n-1}u_0u_{2m+5};t;p,q\big).
\end{gather*}
The corresponding balancing condition to ensure ellipticity is, for $n=2m$, that $\prod\limits_{0\le r<2m} b_r = (t/pq)(pqa)^{m-1}$.

As with the interpolation functions, we also use an elliptic version:
\begin{gather*}
\Delta_\lambda(a|b_1,\dots,b_m;q,t;p):=\Delta_{0,\lambda}(a|b_1,\dots,b_m;t;p,q),
\end{gather*}
and similarly for $\Delta^0_\lambda$. For
\begin{gather*}
C^{0,\pm}_\lambda(x;q,t;p):=\cC^{0,\pm}_{0,\lambda}(x;t;p,q),
\end{gather*}
we also take the convention of omitting $p$ in the limit $p=0$.

The key property of the interpolation functions is that
\begin{gather*}
\cR^{*(n)}_{\blambda}\big(\dots,(p,q)^{\bmu_i}t^{n-i}a,\dots;a,b;t;p,q\big)=0
\end{gather*}
unless $\blambda\subset\bmu$ \cite[Corollary~8.12]{xforms}; this property and the triangularity property are related by a complementation symmetry, and together determine the interpolation function up to normalization, which is determined by
\begin{gather*}
\cR^{*(n)}_{\blambda}\big(\dots,t^{n-i}v,\dots;a,b;t;p,q\big)=\Delta^0_{\blambda}\big(t^{n-1}a/b|t^{n-1}av,a/v;t;p,q\big).
\end{gather*}
The nonzero values of interpolation functions appear frequently enough to merit their own notation: we define
\begin{gather*}
\binomE{\blambda}{\bmu}_{[a,b];t;p,q}\!:=\Delta_{\bmu}\big(a/b|t^n,1/b;t;p,q\big)
\cR^{*(n)}_{\bmu}\big(\dots,\sqrt{a}(p,q)^{\blambda_i}t^{1-i},\dots;t^{1-n}\sqrt{a},b/\sqrt{a};t;p,q\big);
\end{gather*}
this is independent of the choice of square root, and factors as
\begin{gather*}
\binomE{\lambda,\kappa}{\mu,\nu}_{[a,b];t;p,q}=\binomE{\lambda}{\mu}_{[a,b];p,t;q}\binomE{\kappa}{\nu}_{[a,b];q,t;p},
\end{gather*}
where the first factor is $q$-elliptic in $a$, $b$, $p$, and $t$, and similarly for the second factor. In actuality, we will essentially only use
the alternate normalization of~\cite{bctheta}, which in the $p,q$-symmetric version reads
\begin{gather*}
\obinomE{\blambda}{\bmu}_{[a,b](v_1,\dots,v_k);t;p,q}:=\frac{\Delta^0_{\blambda}(a|b,v_1,\dots,v_k;t;p,q)}
 {\Delta^0_{\bmu}(a/b|1/b,v_1,\dots,v_k;t;p,q)}\binomE{\blambda}{\bmu}_{[a,b];t;p,q}.
\end{gather*}
The binomial coefficients so normalized are products of elliptic functions if $k=3$, $bv_1v_2v_3=(pqa)^2$.

Finally, as mentioned above, we will quite frequently be taking $p$ to be a~formal variable. More precisely, we will take the various parameters to be elements of some field of formal power series, in such a way that $p$ has valuation $<1$. For a nonzero element of such a field, we define
\begin{gather*}
\ord_p(x) = \frac{\log|x|}{\log|p|},
\end{gather*}
and will typically omit $p$. (We will, in fact, only use the subscript in a~few cases, in which it is more natural to let the formal parameter be
$q$.) Note that here $||$ denotes the non-Archimedean valuation, i.e., $\exp(-d)$ where $d$ is the degree of the leading term of the power series.
We will generally take the field to be power series in some fixed $N$-th root of $p$ (with coefficients meromorphic or rational in the remaining variables, as appropriate), but will in general simply refer to it as the field of formal Puiseux series. (For simplicity, we only allow rational exponents with finite common denominator; this could be weakened in general, but we will never need more than fourth roots of $p$ in any
event.) We will also use this order notation in the analytic case, to cover the case in which the formal Puiseux series arises as an actual
Puiseux series. That is, given a limit of functions, integrals, etc., in which $x$ is a parameter or variable, we will define $\ord_p(x)$ to be the limit of $\log|x|/\log|p|$, this time using the usual complex absolute value. In general, we will take these limits in such a way that
$x=p^{\alpha} x_0$ for $x_0$ fixed.

Note that an infinite sum of formal Puiseux series converges (and converges absolutely!) iff its terms converge to 0 in absolute value; similarly an infinite product converges iff its terms converge to~1. For instance, the product defining the elliptic Gamma function $\Gampq(z)$ converges iff $\ord(q), \ord(z), \ord(pq/z)>0$ (i.e., the corresponding Puiseux series have no terms with nonpositive exponent). One can, in fact, ``analytically continue'' $\Gampq$ to a somewhat larger range of inputs. Indeed, taking the logarithm of $\Gampq(z)$ gives
\begin{gather*}
\log\Gampq(z)=\sum_{0\le i,j} \log(1-p^{i+1}q^{j+1}/z) -\log(1-p^iq^jz)\\
\hphantom{\log\Gampq(z)}{} = \sum_{0\le i,j}\sum_{k>0}\frac{p^{ik}q^{jk}z^k-p^{(i+1)k}q^{(j+1)k}z^{-k}}{k}.
\end{gather*}
The sum over $i,j$ is geometric, and thus one concludes
\begin{gather*}
\log\Gampq(z)=\sum_{k>0} \frac{z^k-(pq/z)^k}{k(1-p^k)(1-q^k)}.
\end{gather*}
This continues to converge for $\ord(q)=0$ as long as $\lim\limits_{p\to 0}q$ is
not a root of unity, as well as in the range $\ord(1/q),\ord(z/q),\ord(p/z)>0$. This formal extension introduces an
additional symmetry not present for the usual analytic version, namely
$\Gamm{p,1/q}(z)=1/\Gampq(qz)$ (which follows by comparing corresponding
terms of the series for the logs); analytically, the domains of definitions
of the two sides are entirely disjoint.

\section{The formal kernel}\label{section2}

As we mentioned in the introduction, although nonterminating elliptic
hypergeometric series normally fail to converge, we can finesse this issue
by taking $p$ to be a formal variable. The general principle with natural
series of this kind is that the valuation of the terms depends linearly on
the index of summation; e.g., for a sum over partitions, the term
associated to $\lambda$ will be $O(p^{\alpha|\lambda|})$ for some $\alpha$.
Thus in practice, a series will converge formally iff it becomes the
trivial sum $1$ (or any other single-term sum) in the limit $p\to 0$. In
particular, given a sufficiently large family of finite sums with this
property, we can hope to have a straightforward continuation to the
nonterminating case.

For our purposes, the most important sum will be the one expanding an
interpolation function from one basis in terms of the interpolation
functions from another basis. The coefficients here are ``elliptic
binomial coefficients'', but these are essentially just values of
interpolation functions. As a result, we may express the expansion in the
following form (with some reparametrization).

\begin{lem}[{\cite[Corollary~4.14]{bctheta}}]
For $t_0,u_0, c, q, t\in \C^*$, and any $|p|<1$, we have
\begin{gather*}
R^{*(n)}_\lambda\big(x_1,\dots,x_n;cu_0,c/t^{n-1}u_0;q,t;p\big)\\
\quad{} =
\Delta^0_\lambda\big(t^{2n-2}u_0^2|cu_0/t_0,ct^{n-1}t_0u_0;q,t;p\big)
 \sum_{\mu\subset\lambda}
\Delta_\mu\big(t^{2n-2}t_0u_0/c|t^n,pqt^{n-1}/c^2;q,t;p\big)\\
\quad\quad {} \times R^{*(n)}_\mu\big(x_1,\dots,x_n;t_0,c/t^{n-1}u_0;q,t;p\big)
R^{*(n)}_\mu\big(\dots,q^{\lambda_i}t^{n-i}u_0,\dots;u_0,c/t^{n-1}t_0;q,t;p\big).
\end{gather*}
\end{lem}

As mentioned, we need the sum to be dominated by its first term, and to
understand when that happens, we need to understand the valuations of the
individual components of the summand. To make our initial calculations
easier, we assume (as we always will) that $q$, $t$ have order~0, and, at
least initially, that $t_0$, $u_0$, $x_1,\dots,x_n$ have order~$0$. Since
the $\Delta$ symbol is defined as a~product, it is straightforward to
compute its valuation, and we find that it has order $\ord(c)|\mu|$, as
long as $0<\ord(c)\le 1/2$. (Here, of course, this is only the {\em
 generic} valuation; dividing by this power of $p$ makes the limit a
generically nonzero rational function.) The interpolation functions are a
priori harder to control, but luckily their valuations were computed in~\cite{biorth_degens}, and we find that both interpolation functions have
order $0$. Combining, we find that the~$\mu$ term in the sum has order
$\ord(c)|\mu|$. In particular, we find that even if the second
interpolation function is evaluated at a generic point (with coordinates of
order $0$), the corresponding nonterminating sum will converge formally.
More generally, if $x_1,\dots,x_n$ have order $x\in (-1/2,1/2)$, the
results of~\cite{biorth_degens} indicate that the interpolation function
has order $-|x||\mu|$, and the argument there further demonstrates that
this is a lower bound so long as $|\ord(x_1)|,\dots,|\ord(x_n)|\le x$.

We are led to introduce some additional prefactors to maximize symmetry,
and arrive at the following definition. (In fact, as written, it is not
quite a definition; we will need to show that the sum is independent of the
auxiliary parameters.)

\begin{defn}
 For $p$ a formal variable, $|q|,|t|<1$, and $\vec{x}$, $\vec{y}$, $c$
 parameters such that
\begin{gather*}
\max_i|\ord(x_i)|+\max_i|\ord(y_i)|<\ord(c)\le 1/2,
\end{gather*}
the {\em formal interpolation kernel} $K^{(n)}_c(\vec{x};\vec{y};q,t;p)$ is
defined by the following infinite sum (depending on auxiliary variables
$t_0$, $u_0$ of valuation $0$):
\begin{gather*}
K^{(n)}_c(\vec{x};\vec{y};q,t;p)
:= \prod_{1\le i\le n}
\frac{\Gampq\big(cu_0x_i^{\pm 1},(c/t^{n-1}u_0)x_i^{\pm 1},ct_0y_i^{\pm 1},(c/t^{n-1}t_0)y_i^{\pm 1}\big)}
 {\Gampq\big(ct^{n-i}t_0u_0,ct^{1-i}u_0/t_0,ct^{i+1-2n}/t_0u_0,ct^{i-n}t_0/u_0,t^{i-n}c^2,t^i\big)} \\
\hphantom{K^{(n)}_c(\vec{x};\vec{y};q,t;p):=}{} \times
\sum_{\mu} \Delta_\mu\big(t^{2n-2}t_0u_0/c|t^n,t^{n-1}pq/c^2;q,t;p\big) \\
\hphantom{K^{(n)}_c(\vec{x};\vec{y};q,t;p):=}{} \times
R^{*(n)}_{\mu}\big(\vec{x};t_0,c/t^{n-1}u_0;q,t;p\big)R^{*(n)}_{\mu}\big(\vec{y};u_0,c/t^{n-1}t_0;q,t;p\big).
\end{gather*}
\end{defn}

\begin{rem}
 The constraint $\ord(c)\le 1/2$ is required for the prefactor
 $\prod\limits_{1\le i\le n} \Gampq(t^{i-n}c^2)^{-1}$ to be defined. We could of
 course omit this factor, but at the cost of making a number of later
 formulas (in particular for the analytic kernel) more complicated. We
 find in general that the given expression for
\begin{gather*}
\prod_{1\le i\le n} \Gampq\big(t^{i-n}c^2,t^i\big) K^{(n)}_c(\vec{x};\vec{y};q,t;p)
\end{gather*}
converges formally (and has limit $1$ as $p\to 0$) as long as
\begin{gather*}
\max(|\ord(t_0)|,\max_i|\ord(x_i)|)+ \max(|\ord(u_0)|,\max_i|\ord(y_i)|)<\min(\ord(c),1-\ord(c)),
\end{gather*}
and agrees with the corresponding formal expansion of the analytic kernel
(q.v.) as long as $\ord(t_0)=\ord(u_0)=0$.
\end{rem}

\begin{lem} For $a$, $q$, $t$, $\vec{x}$ of order $0$, and $0<\ord(b)\le 1/2$, the coefficient of $p^\alpha$ in $R^{*(n)}_\lambda(\vec{x};a,b;q,t;p)$ is $0$ if $\alpha<0$, and is otherwise a hyperoctahedrally symmetric Laurent polynomial in $\vec{x}$ of degree at most $\frac{\alpha}{\ord(b)}+|\lambda|$, with coefficients rational functions in the remaining parameters.
\end{lem}

\begin{proof}
 That the coefficient of $p^\alpha$ vanishes for $\alpha<0$ follows by the
 valuation calculation of~\cite{biorth_degens}. That the coefficients are
 hyperoctahedrally symmetric follows from the corresponding symmetry of
 the interpolation functions. Finally, that they are polynomials of the
 appropriate degree follows either by induction using the branching rule
 as in \cite{biorth_degens} or using the expansion formula of
 \cite[Theorem~2.5]{littlewood}. The latter expresses the interpolation
 function as a finite sum in which the dependence of each term on
 $\vec{x}$ is as a product over the variables and their reciprocals.
 (Compare Definition~\ref{def:lifted_interp} below.)
\end{proof}

\begin{rem} One should note that the expression from \cite[Theorem~2.5]{littlewood} involves a significant amount of cancellation; each individual term has order $-|\lambda|\ord(b)$, but all terms of negative order cancel.
\end{rem}

Plugging this into the definition of the formal kernel gives the following result.

\begin{cor} For $t_0$, $u_0$, $\vec{x}$, $\vec{y}$ of order $0$ and $c$ with $0<\ord(c)\le 1/2$, the coefficient of $p^\alpha$ in
\begin{gather*}
\prod_{1\le i\le n} \Gampq\big(t^{i-n}c^2,t^i\big) K^{(n)}_c(\vec{x};\vec{y};q,t;p)
\end{gather*}
is a hyperoctahedrally symmetric Laurent polynomial in each of $\vec{x}$ and $\vec{y}$ of degree at most $\frac{\alpha}{\ord(c)}$, with coefficients
rational functions in the remaining parameters.
\end{cor}

The significance of this result is that it allows us to extend identities
of $K^{(n)}$ from the case in which $\vec{x}$ and $\vec{y}$ have valuation~$0$.

\begin{thm}
 The formal interpolation kernel is well-defined, symmetrical between
 $\vec{x}$ and $\vec{y}$, and has the specialization
\begin{gather*}
K^{(n)}_c\big(\vec{x};\dots,t^{n-i}q^{\mu_i}a,\dots;q,t;p\big)\\
\qquad{} =
\prod_{1\le i\le n}
 \frac{\Gampq\big(ac x_i^{\pm 1},c/t^{n-1}a x_i^{\pm 1}\big)}
 {\Gampq\big(t^{1-i}c^2,t^i\big)}
R^{*(n)}_\mu\big(\vec{x};ac,c/t^{n-1}a;q,t;p\big),
\end{gather*}
for any $a$ of order $0$.
\end{thm}

\begin{proof} By construction, the given specialization holds as long as we take
 $a=u_0$. This, in particular, shows that the sum is independent of $t_0$
 as long as $\vec{y}$ has order $0$. (Indeed, apart from a simple
 prefactor, the coefficients of the Puiseux series are rational functions
 of the parameters, and for a Zariski dense set of possible $\vec{y}$ the
 coefficients are independent of~$t_0$.) This independence of $t_0$ then
 extends to more general valuations of $\vec{y}$ using the corollary: if
 we take the expansion in powers of $p^\alpha$ corresponding to the
 valuation 0 case, the corollary tells us that the result continues to
 converge for the broader range of valuations of $\vec{y}$, and thus the
 identity between the two expansions extends as well.

 Since the sum is unchanged if we swap $t_0$ and $u_0$ as well as
 $\vec{x}$ and $\vec{y}$, we conclude that the sum is also independent of
 $u_0$, so well-defined. The remaining claims are then immediate.
\end{proof}

In particular, if we specialize one set of variables to a geometric
progression, we have an explicit evaluation. This only works directly in
the case that the base of the progression has order $0$, but easily extends
to more general valuations.

\begin{cor}\label{cor:formal_geom_special}
If $|\ord(a)|+\max_i|\ord(x_i)|<\ord(c)\le 1/2$, then
\begin{gather*}
K^{(n)}_c\big(\vec{x};\dots,t^{n-i}a,\dots;q,t;p\big)
=\prod_{1\le i\le n} \frac{\Gampq\big(ac x_i^{\pm 1},c/t^{n-1}a x_i^{\pm 1}\big)} {\Gampq\big(t^{1-i}c^2,t^i\big)}.
\end{gather*}
In particular,
\begin{gather*}
K^{(1)}_c(x;y;q,t;p)=\frac{\Gampq\big(c x^{\pm 1} y^{\pm 1}\big)} {\Gampq\big(c^2,t\big)}.
\end{gather*}
\end{cor}

\begin{rem} This gives an initial indication of why we call this a ``kernel'':
 $K^{(1)}_c$ is essentially the kernel of a univariate integral operator
 considered in~\cite{SpiridonovVP:2004}, where it is shown that the
 operators with kernels $K^{(1)}_c$ and $K^{(1)}_{c^{-1}}$ are inverse in
 a suitable sense. Note that the present kernel is not the only natural
 multivariate extension of the univariate kernel; see~\cite{SpiridonovVP/WarnaarSO:2006}, where several such extensions (with
 much simpler explicit formulas!) are given.
\end{rem}

We also have some special cases with explicit formulas corresponding to similar special cases of the interpolation functions.

\begin{prop}\label{prop:special_t}
The kernel has the special cases
\begin{gather*}
K^{(n)}_{(pq/t)^{1/2}}(\vec{x};\vec{y};q,t;p)=
\prod_{1\le i,j\le n} \Gampq\big((pq/t)^{1/2}x_i^{\pm 1} y_j^{\pm 1}\big),\\
K^{(n)}_c(\vec{x};\vec{y};q,q;p)=\prod_{1\le i<j\le n}\frac{c^{-1}x_iy_i}
 {\theta_p(x_ix_j,x_i/x_j,y_iy_j/y_i/y_j)}\det_{1\le i,j\le n} K^{(1)}_c(x_i;y_j;q;p,q),
\end{gather*}
and the limiting case
\begin{gather*}
\lim_{t\to 1}\prod_{1\le i\le n} \Gampq\big(t^{1-i}c^2,t^i\big)K^{(n)}_c(\vec{x};\vec{y};q,t;p)=
\frac{1}{n!}\sum_{\pi\in X_n}\prod_{1\le i\le n} \Gampq\big(c x_i^{\pm 1}y_{\pi(i)}^{\pm 1}\big).
\end{gather*}
\end{prop}

\begin{proof}In each case, it suffices to verify that the identity holds when $y_i=t^{n-i}q^{\lambda_i}a$ for any partition $\lambda$.
These correspond to the Cauchy, Schur, and monomial cases of the interpolation functions, see~\cite{bctheta}.
\end{proof}

If we rewrite the first special case as an identity for sums, we find that
the result is a~(non\-ter\-mi\-nating) version of the elliptic Cauchy identity
\cite[Theorem~3.6]{littlewood}. (The latter was expressed in terms of certain
plethystic generalizations of the interpolation functions, which (in the
formal case) turn out to be special cases of the symmetric function
analogue, see below.) Na\"\i{}vely, the usual Cauchy identity for
Macdonald polynomials arises by a term-by-term limit, using the fact that
\begin{gather*}
\lim_{p\to 0}\frac{R^{*(n)}_\lambda\big(p^{-1/4}\vec{x};t_0,p^{1/2}u_0;q,t;p\big)}
 {R^{*(n)}_\lambda\big(p^{-1/4}t^{n-1}v,\dots,p^{-1/4}v;t_0,p^{1/2}u_0;q,t;p\big)}
=\frac{P_\lambda(\vec{x};q,t)} {P_\lambda\big(t^{n-1}v,\dots,v;q,t\big)}.
\end{gather*}
If we apply this to $K^{(n)}_c$ termwise, we obtain
\begin{gather*}
\lim_{p\to 0}\prod_{1\le i\le n} \frac{\Gampq\big(t^i\big)}
 {\Gampq\big(t^{i-1}q/c^2\big)}K^{(n)}_{p^{1/2}c}\big(p^{-1/4}\vec{x};p^{-1/4}\vec{y};q,t;p\big) \\
\qquad {} \text{``=''}
\sum_{\mu} c^{|\mu|}\frac{C^0_\mu\big(qt^{n-1}/c^2;q,t\big)C^-_\mu(t;q,t)}
 {C^0_\mu(t^n;q,t) C^-_\mu(q;q,t)} P_\mu(\vec{x};q,t)P_\mu(\vec{y};q,t).
\end{gather*}
The problem, of course, is that the left-hand side is not defined, and the
right-hand side need not converge. Now, the right-hand side {\em does}
converge as a formal series in $\vec{x}$ and/or $\vec{y}$, which is the key
to making the limit rigorous. Indeed, we find that
\begin{gather*}
\lim_{N\to\infty}
\prod_{1\le i\le n} \frac{\Gamma_{p^{4N},q}\big(t^i\big)}
 {\Gamma_{p^{4N},q}\big(t^{i-1}q/c^2\big)}K^{(n)}_{p^{2N}c}\big(p^{-N+1}\vec{x};p^{-N+1}\vec{y};q,t;p^{4N}\big)\\
\qquad {} =\sum_{\mu} c^{|\mu|}
\frac{C^0_\mu\big(qt^{n-1}/c^2;q,t\big)C^-_\mu(t;q,t)}
 {C^0_\mu\big(t^n;q,t\big) C^-_\mu(q;q,t)} P_\mu(p\vec{x};q,t)
P_\mu(p\vec{y};q,t)
\end{gather*}
holds as a limit of formal power series in $p$. In this way, any formula for $K^{(n)}_c$ with $\ord(c)=1/2$ gives rise to a corresponding identity
of Macdonald polynomials.

The name ``kernel'' comes from the fact that $K^{(n)}_c$ forms the kernel of an integral operator having the interpolation functions as (generalized) eigenfunctions. Here we of course define the integral of a formal power series by integrating term-by-term. As long as the~$u_r$ parameters of valuation $0$ are inside the unit circle, the contour can be taken to be a~power of the unit circle; one can then extend to general parameters as in \cite[Section~10]{xforms}. In particular, the integration variables here will always have order~$0$.

\begin{prop}\label{prop:int_eq_formal}
If $u_0$, $u_1$, $u_2$, $u_3$ are parameters of nonnegative order such that \begin{gather*}
t^{n-1}u_0u_1u_2u_3 = pq/c^2,
\end{gather*} then
\begin{gather*}
\prod_{\substack{1\le i\le n\\0\le r<4}} \frac{1} {\Gampq\big(cu_ry_i^{\pm 1}\big)}
\int K^{(n)}_c(\vec{x};\vec{y};q,t;p) R^{*(n)}_{\lambda}(\vec{x};u_0,u_1;q,t;p)
\Delta^{(n)}_S(\vec{x};u_0,u_1,u_2,u_3;t;p,q)\\
=\Delta^0_{\lambda}\big(t^{n-1}u_0/u_1|t^{n-1}u_0u_2,t^{n-1}u_0u_3;q,t;p\big)\!
\bigg(\!\prod_{\substack{1\le i\le n\\ 0\le r<s<4}}\!\!\! \!\! \Gampq\big(t^{n-i}u_ru_s\big)\!\bigg)\!
R^{*(n)}_{\lambda}(\vec{y};cu_0,cu_1;q,t;p).
\end{gather*}
\end{prop}

\begin{proof} We first note that the kernel and interpolation function are both formal power series with coefficients polynomial in $x$ and $y$, so the integral is indeed well-defined, and has coefficients polynomial in~$y$. Specializing to the Zariski dense set $y_i=q^{\lambda_i}t^{n-i}u_0$ reduces to Theorem~9.2 of~\cite{xforms}.
\end{proof}

One thing of particular note about this integral equation is that the
right-hand side has roughly the same form as other known equations for the
interpolation functions. For instance, the integral operators considered
in \cite{xforms} satisfy precisely such an equation, except that
$c=\sqrt{t}$ has valuation~0. Though this is of course outside the range
of formal convergence, it still suggests an ``identity'' by comparing
integrands:
\begin{gather}
K^{(n)}_{t^{1/2}}(\vec{x};\vec{y};q,t;p)\text{``=''}\frac{\prod\limits_{1\le i\le n,1\le j\le n} \Gampq\big(t^{1/2}x_j^{\pm 1}y_i^{\pm 1}\big)}
{\Gampq(t)^{2n}\prod\limits_{1\le i<j\le n} \Gampq\big(t x_i^{\pm 1}x_j^{\pm 1},t y_i^{\pm 1}y_j^{\pm
 1}\big)}.\label{eq:nonsense}
\end{gather}
Though this is nonsense for the {\em formal} kernel, we will see below that it indeed holds for the {\em analytic} kernel.

Perhaps even more surprising is the fact that essentially the same right-hand side appears in \cite[Lemma~9.8]{xforms}, which describes a {\em
 difference} equation satisfied by the interpolation functions. Thus in general we expect that when $c=q^{-m/2}$, the kernel should act as a
multivariate difference operator of order $m$ (i.e., shifting each variable by at most $m/2$). Presumably this can again be made rigorous at the analytic level, but we will content ourselves with using it in the formal context. (See also Section~\ref{section4} below.)

It is of course straightforward to extend the above integral equation to an equation satisfied by the formal kernel, by the usual ``compare when~$\vec{y}$ is a partition'' argument. This gives us the following identity, which we will refer to as the ``braid relation''.

\begin{prop}[braid relation]\label{prop:kern_braid}
Suppose $0<\ord(c),\ord(d)$ and $\ord(c)+\ord(d)\le 1/2$. Then
for any parameters $u_0$, $u_1$ of positive valuation such that
$u_0u_1 = pq/c^2d^2$,
\begin{gather*}
\int K^{(n)}_c(\vec{z};\vec{x};q,t;p) K^{(n)}_d(\vec{z};\vec{y};q,t;p)\Delta^{(n)}_S(\vec{z};u_0,u_1;t;p,q)\\
\qquad{} = \prod_{1\le i\le n}
 \Gampq\big(cu_0x_i^{\pm 1},cu_1x_i^{\pm 1},du_0y_i^{\pm 1},du_1y_i^{\pm 1}\big) K^{(n)}_{cd}(\vec{x};\vec{y};q,t;p).
\end{gather*}
\end{prop}

\begin{rems}
 The reader can remember the name ``braid relation'' by noting that this
 identity has a sort of ``${\rm ABA}= {\rm BAB}$'' structure, where each $A$ is an
 instance of the kernel, and each $B$ is multiplication by a pair of
 elliptic Gamma functions. (It is in fact essentially the only nontrivial
 braid relation for a certain twisted action of a Coxeter group of type~$E_n$; see Section~\ref{section4} below for further discussion.)
\end{rems}

\begin{rems}Using this, one can show that the sum defining $K^{(n)}_c$ converges to the correct value on a wider range of valuations, namely as long as
\begin{gather*}
\max(|\ord(t_0)|,\max_i|\ord(x_i)|)+\max(|\ord(u_0)|,\max_i|\ord(y_i)|)<\ord(c)\le 1/2.
\end{gather*}
Indeed, one can use the braid relation together with the integral equation to express the sum as an integral involving only the special cases in which one of the valuations is~$0$, for which the above arguments suffice. We omit the details.
\end{rems}

\looseness=-1 There are, of course, corresponding identities for $c=t^{1/2}$ or $c=q^{-m/2}$, since all we are using is that the corresponding operators preserve the space of formal series with polynomial coefficients, and act in the appropriate way on interpolation functions. Thus, for instance, the above identity continues to hold if we take $c=t^{1/2}$ and expand $K^{(n)}_{t^{1/2}}$ via the (nonsense) formula~\eqref{eq:nonsense}. Similarly, there is a difference equation, which we postpone until the analytic case.

The integral representation for interpolation functions also involves an
integrand very similar to $K^{(n)}_{t^{1/2}}$, and in particular also
extends to an identity for the formal kernel. To extend this to the formal
kernel, we need only establish that the corresponding integral operator
preserves the space of formal Puiseux series with polynomial coefficients.
This follows from the fact that the kernel of the operator can be factored
as a formal Puiseux series with polynomial coefficients times its value for
$p=0$; since the limiting operator was shown to preserve polynomials in
\cite[Theorem~3.2]{diffintrep_koorn}, the claim follows.

\begin{lem}\label{lem:intrep_t}
 If $0<\ord(c)\le 1/2$ and all other parameters have valuation $0$ $($with $|q|,|t|<1)$, then
\begin{gather*}
K^{(n)}_{ct^{1/2}}(\vec{x};\vec{y},v;q,t;p)
= \frac{\prod\limits_{1\le i\le n} \Gampq\big(ct^{1/2}v^{\pm 1}x_i^{\pm 1}\big)}
{\Gampq(t)^n\Gampq\big(tc^2\big)\prod\limits_{1\le i<j\le n}\Gampq\big(t x_i^{\pm 1}x_j^{\pm 1}\big)}\\
\hphantom{K^{(n)}_{ct^{1/2}}(\vec{x};\vec{y},v;q,t;p)=}{}\times \int
K^{(n-1)}_c(\vec{z};\vec{y};q,t;p)
\frac{\prod\limits_{\substack{1\le i\le n-1,\\ 1\le j\le n}} \Gampq\big(t^{1/2}x_j^{\pm 1}z_i^{\pm 1}\big)}
 {\prod\limits_{1\le i\le n-1} \Gampq\big(t c v^{\pm 1} z_i^{\pm 1}\big)} \Delta^{(n-1)}_D(\vec{z};p,q).
\end{gather*}
\end{lem}

This is a special case of a much more general integral formula. We omit the convergence conditions, as we will see shortly that the identity holds
(as do those above) whenever both sides converge.

\begin{thm}\label{thm:gen_branch}
For any integers $0\le k\le n$, we have the following identity
\begin{gather*}
K^{(n)}_{cd}\big(\vec{x} ;\vec{y},t^{k-1}v,\dots,v;q,t;p\big) \\
\qquad{}= \prod_{1\le i\le k}
\frac{\Gampq\big(t^{1-i}c^2\big)} {\Gampq\big(t^{1-i}c^2d^2\big)}
\prod_{1\le i\le n} \frac{\Gampq\big(dcvx_i^{\pm 1}\big)}
 {\Gampq\big((cv/d)x_i^{\pm 1}\big)} \prod_{1\le i\le n-k} \frac{\Gampq\big(c^2vy_i^{\pm 1}\big)} {\Gampq\big(t^kvy_i^{\pm 1}\big)} \\
\qquad\quad{}\times
\int K^{(n)}_c\big(\vec{x};\vec{z},t^{k-1}v/d,\dots,v/d;q,t;p\big)\\
\qquad\quad{}\times \prod_{1\le i\le n-k}
 \frac{\Gampq\big(\big(t^kv/d\big)z_i^{\pm 1}\big)} {\Gampq\big(c^2dvz_i^{\pm 1}\big)}
K^{(n-k)}_d(\vec{z};\vec{y};q,t;p) \Delta^{(n-k)}_S(\vec{z};t;p,q).
\end{gather*}
\end{thm}

\begin{proof}Again, it suffices to verify this when $x$ has been specialized to a~partition. In that case, both resulting interpolation functions can be expanded via the generalized branching rule of~\cite[Theorem~4.16]{bctheta}, and the claim follows upon applying Proposition~\ref{prop:int_eq_formal} term-by-term.
\end{proof}

\begin{rem}
Note that when $k=0$, this is just the braid relation, while when $k=n$, it is just Corollary~\ref{cor:formal_geom_special}; it also agrees in the usual sense with the integral representation, by taking $k=1$, $c=t^{1/2}$ and replacing that instance of the kernel as per usual.
\end{rem}

\section{The interpolation kernel}\label{section3}

The key benefit of Lemma \ref{lem:intrep_t} is that it expresses the $n$-dimensional formal kernel as an integral of the $(n-1)$-dimensional formal kernel, and thus we can iterate to obtain an expression as an~$(n(n-1)/2)$-dimensional integral, in which the integrand is a suitable product of elliptic Gamma functions alone. In particular, the resulting integrand is the formal power series expansion of an honest meromorphic
function.

\begin{thm}There exists a function $\cK^{(n)}_c(\vec{x};\vec{y};t;p,q)$, meromorphic on the region
\begin{gather*}
\{c,x_1,\dots,x_n,y_1,\dots,y_n,q,t\in \C^*, 0\le |p|,|q|<1\},
\end{gather*}
such that any Puiseux expansion of this function with $\ord(q)=\ord(t)=0$,
$\max_i|\ord(x_i)|+\max_i|\ord(y_i)|<\ord(c)\le 1/2$ is equal to $K^{(n)}_c$.
This function satisfies the symmetries
\begin{gather*}
\cK^{(n)}_c(\vec{x};\vec{y};t;p,q)
=
\cK^{(n)}_c(\vec{x};\vec{y};t;q,p)
=
\cK^{(n)}_c(\vec{y};\vec{x};t;p,q),
\end{gather*}
and on a suitable open subset can be defined inductively by
\begin{gather*}
\cK^{(n)}_c(\vec{x};\vec{y},v;t;p,q)=
\frac{\prod\limits_{1\le i\le n} \Gampq\big(cv^{\pm 1}x_i^{\pm 1}\big)}
{\Gampq(t)^n\Gampq\big(c^2\big)\prod\limits_{1\le i<j\le n}\Gampq\big(t x_i^{\pm 1}x_j^{\pm 1}\big)}\\
\hphantom{\cK^{(n)}_c(\vec{x};\vec{y},v;t;p,q)=}{} \times \int
\cK^{(n-1)}_{t^{-1/2}c}(\vec{z};\vec{y};t;p,q)
\frac{\prod\limits_{\substack{1\le i\le n-1,\\ 1\le j\le n}} \Gampq\big(t^{1/2}x_j^{\pm 1}z_i^{\pm 1}\big)}
 {\prod\limits_{1\le i\le n-1} \Gampq\big(t^{1/2}c v^{\pm 1} z_i^{\pm 1}\big)}\Delta^{(n-1)}_D(\vec{z};p,q),
\end{gather*}
with base case $\cK^{(0)}_c=1$.
\end{thm}

\begin{proof}
We first note that if $|t|^{1/2}<|x_i|<|t|^{-1/2}$ and
$\max_i|\ord(y_i)|<\ord(c)$ for each $i$, then the integrand has the
following property: if every integration variable is on the unit circle,
then every elliptic Gamma function in the integrand has argument of
absolute value between $|pq|$ and~$1$. Thus if we fix the other
parameters, the integral is holomorphic (apart from an algebraic
singularity) near $p=0$, and the Puiseux series expansion of the integral
is the same as the term-by-term integral of the Puiseux series expansion of
the integrand. In other words, in this range, the formal kernel is
actually a convergent Puiseux series, and converges to the value of the
integral. The existence of a meromorphic extension follows from
\cite[Theorem~10.2]{xforms}.

In particular, it follows that all of the above identities for the formal
kernel continue to hold (with suitably deformed contours) for $\cK^{(n)}$,
in particular the braid relation. Thus to extend to the full set of
valuations $|\ord(x)|+|\ord(y)|<\ord(c)\le 1/2$, it suffices to write
$c=c_1c_2$ with $\ord(c_1)>|\ord(x)|$, $\ord(c_2)>|\ord(y)|$, and note that
the integration variables in the braid relation have order $0$.

The symmetry between $p$ and $q$ is by inspection of the integral
representation, while the symmetry between $\vec{x}$ and $\vec{y}$ follows
from the corresponding symmetry of the formal kernel.
\end{proof}

\begin{rem}
 Presumably if we multiply by $\prod\limits_{1\le i\le n}\Gampq\big(t^{1-i}c^2,t^i\big)$,
 then the analytic kernel has a formal expansion with rational function
 coefficients whenever
\begin{gather*}
\max_i(|\ord(x_i)|)+\max_i(|\ord(y_i)|)<\min(\ord(c),1-\ord(c)),
\end{gather*}
and this expansion agrees with the sum defining the formal kernel. This
certainly holds when $\ord(\vec{x})=0$, as the integral representation
remains valid in that case. For more general valuations of $\vec{x}$, some
sort of symmetry breaking limit will be required.
\end{rem}

Of course, when we specialize $\vec{x}$ to a partition, we recover the
integral representation of the corresponding interpolation function. In
fact, we obtain even more: the analytic kernel is manifestly symmetric
between $p$ and $q$, and thus we can also obtain $q$-elliptic interpolation
functions by a suitable specialization. In fact, we can obtain the full
analytic interpolation functions of \cite{xforms}.

\begin{prop}Let $\lambda$, $\mu$ be a pair of partitions with at most $n$ parts. Then we have the following identity of meromorphic functions
\begin{gather*}
\cR^{*(n)}_{\lambda,\mu}(\vec{z};a,b;t;p,q)=\prod_{1\le i\le n}
\frac{(pq/ab)^{-2\lambda_i\mu_i}\Gampq\big(t^{n-i}ab,t^i\big)}
 {\Gampq\big(a z_i^{\pm 1},b z_i^{\pm 1}\big)} \\
\hphantom{\cR^{*(n)}_{\lambda,\mu}(\vec{z};a,b;t;p,q)=}{}\times
\cK^{(n)}_c\big(\vec{z};\dots,p^{\lambda_i}q^{\mu_i}t^{n-i}a/c,\dots;t;p,q\big),
\end{gather*}
where $c=\sqrt{t^{n-1}ab}$.
\end{prop}

\begin{rem} This might appear at first glance to be incompatible with Proposition~\ref{prop:special_t}. For instance, for $t=q$, Proposition~\ref{prop:special_t} says that the kernel can be expressed as a simple determinant; on the other hand, the general interpolation function for $t=q$ can only be expressed as a sum of $n!$ determinants in general. We can resolve this by noting that the interpolation function specialization only holds for {\em generic} values of the parameters; if we first specialize $a$, $b$, $p$, $q$, $t$ before specializing the variables, we can obtain a different result. This can only occur at poles of the interpolation kernel, but it follows easily from the results below on such poles that when $t=q$, the point $y_i = p^{\lambda_i}q^{\mu_i}t^{n-i}a/\sqrt{t^{n-1}ab}$ is on such a polar divisor whenever $\lambda\ne 0$.
\end{rem}

The fact that the general interpolation function is a special case of the kernel is a quite power\-ful tool, as it allows us to extend integral
formulas involving $p$-elliptic interpolation functions to integral formulas involving general interpolation functions. Indeed, any formula
involving $p$-elliptic interpolation functions that satisfies suitable formal convergence properties implies a corresponding identity for the
formal kernel, thus a corresponding identity for the analytic kernel, so by specializing gives the identity for general interpolation functions!

In addition to the symmetry between $p$ and $q$, there is another symmetry of the interpolation kernel that does not make sense for the formal kernel; in fact, this additional symmetry also does not make sense for interpolation functions. The key point is that the analytic version of
Theorem~\ref{thm:gen_branch} {\em also} gives an explicit integral representation, by taking $k=1$, $c=\sqrt{pq/t}$. Comparing the two
integral representations gives the following.

\begin{prop} The interpolation kernel satisfies the following identity
\begin{gather*}
\cK^{(n)}_c(\vec{x};\vec{y};pq/t;p,q)=\Gampq(t)^{2n}
\prod_{1\le i<j\le n} \Gampq\big(t x_i^{\pm 1}x_j^{\pm 1},t y_i^{\pm 1}y_j^{\pm 1}\big) \cK^{(n)}_c(\vec{x};\vec{y};t;p,q) \\
\hphantom{\cK^{(n)}_c(\vec{x};\vec{y};pq/t;p,q)}{} =
\frac{\Delta^{(n)}_S(\vec{x};t;p,q)} {\Delta^{(n)}_D(\vec{x};p,q)}\frac{\Delta^{(n)}_S(\vec{y};t;p,q)} {\Delta^{(n)}_D(\vec{y};p,q)}\cK^{(n)}_c(\vec{x};\vec{y};t;p,q).
\end{gather*}
In particular,
\begin{gather*}
\cK^{(n)}_{\sqrt{t}}(\vec{x};\vec{y};t;p,q)=
\frac{\prod\limits_{1\le i,j\le n} \Gampq\big(t^{1/2}x_j^{\pm 1}y_i^{\pm 1}\big)}
{\Gampq(t)^{2n}\prod\limits_{1\le i<j\le n} \Gampq\big(t x_i^{\pm 1}x_j^{\pm 1},t y_i^{\pm 1}y_j^{\pm 1}\big)}.
\end{gather*}
\end{prop}

Here, of course, the product formula for $\cK^{(n)}_{\sqrt{t}}$ follows via the symmetry from the product formula for $\cK^{(n)}_{\sqrt{pq/t}}$.

\begin{cor}\label{cor:t_symmetry}
The functions
\begin{gather*}
\cK^{(n)}_{pq/t}(z_1,\dots,z_n;z_{n+1},\dots,z_{2n};t;p,q)
\end{gather*}
and
\begin{gather*}
\prod_{1\le i<j\le n} \Gampq\big(t z_i^{\pm 1}z_j^{\pm 1},t z_{n+i}^{\pm 1}z_{n+j}^{\pm 1}\big)
\cK^{(n)}_t(z_1,\dots,z_n;z_{n+1},\dots,z_{2n};t;p,q)
\end{gather*}
are invariant under permutation and inversion of the $2n$ variables.
\end{cor}

\begin{proof} The first claim is a simple consequence of the braid relation for $c=d=\sqrt{pq/t}$; the second follows by the $t\mapsto pq/t$ symmetry.
\end{proof}

Of course, simply knowing that the kernel is meromorphic is of only limited use without more specific information about the poles. It is difficult to control {\em all} of the poles, but we can at least control the poles depending on the $x$ and $y$ variables.

\begin{thm}\label{thm:kern_poles}
The product
\begin{gather*}
\cK^{(n)}_c(\vec{x};\vec{y};t;p,q)
\prod_{1\le i<j\le n}
\big((pq/t) x_i^{\pm 1}x_j^{\pm 1},(pq/t) y_i^{\pm 1}y_j^{\pm 1};p,q\big)
\prod_{1\le i,j\le n} \big(c x_i^{\pm 1}y_j^{\pm 1};p,q\big)
\end{gather*}
is a holomorphic function of $x_1,\dots,x_n,y_1,\dots,y_n$, for generic
$p$, $q$, $t$, $c$.
\end{thm}

\begin{proof}
 We proceed by induction on $n$, so that we need only analyze an
 $(n-1)$-dimensional integral, rather than an $(n(n-1)/2)$-dimensional
 integral. (Note that for $n=1$, we can verify the claim by inspection,
 while for $n=2$, the integral representation is just an order~$1$
 elliptic beta integral, and the claim follows from known properties of
 such integrals.)

 The construction of meromorphic extensions of integrals in \cite{xforms}
 comes with a very crude bound on the set of possible poles. Indeed, if
 we multiply the integrand by
\begin{gather*}
\prod_{1\le i\le n-1,1\le j\le n} \big(t^{1/2}x_j^{\pm 1} z_i^{\pm 1};p,q\big)
\prod_{1\le i,j\le n-1} \big(t^{-1/2} c y_j^{\pm 1} z_i^{\pm 1};p,q\big) \\
\qquad {}\times \prod_{1\le i\le n-1} \big(\big(pq/t^{1/2}c\big) y_n^{\pm 1} z_i^{\pm 1};p,q\big)
\prod_{1\le i<j\le n-1} \big((pq/t) z_i^{\pm 1}z_j^{\pm 1};p,q\big),
\end{gather*}
the result is a holomorphic function of the integration variables. It
follows that the integral (ignoring prefactors) is holomorphic whenever
there exists a contour $C$ invariant under $z\mapsto 1/z$ such that $C$
contains $(pq/t)C$ as well as every point of the form
\begin{alignat*}{3}
& p^j q^k t^{1/2}x_i^{\pm 1},\qquad && 0\le j,k;1\le i\le n, & \\
& p^j q^k t^{-1/2} c y_i^{\pm 1},\qquad && 0\le j,k;1\le i\le n-1, & \\
& p^j q^k \big(pq/t^{1/2}c\big) y_n^{\pm 1},\qquad&& 0\le j,k.&
\end{alignat*}
We thus find that for $|pq/t|<1$, the integral can only introduce poles
where two numbers (duplication allowed) from these lists multiply to a
nonnegative power of $t/pq$. Of course, this allows plenty of poles that
we claim do not occur, and does not control the multiplicities of those
poles that should occur. (There are results from~\cite{xforms} that could
be used to control the multiplicities of these poles, and rule some of them
out entirely, but this would still give a wild overestimate of the polar
divisor!)

The key fact that allows us to control the poles is that the integral
representation, and thus the corresponding upper bound on the set of poles,
has less symmetry than the actual kernel. In particular, the poles
involving $y_n$ are quite different than those involving $y_1$ through
$y_{n-1}$, but the result should be invariant under permuting {\em all} of
the $y$ variables. In addition, we know from the formal kernel that the
analytic kernel is invariant under swapping the $x$ and $y$ variables. We
find (for generic $c$) that the only poles consistent with these symmetries
are those of the form given.

Applying the $t\mapsto pq/t$ symmetry shows that the same bound on poles
applies for $|t|<1$, and since $|p|,|q|<1$, we conclude that the bound on
poles holds in general.
\end{proof}

\begin{rem}We can also gain some control over the poles that depend on $c$ but not the~$x$ and~$y$ variables, using the braid relation. The point there is that most of the poles coming from the integral in the braid relation depend on the auxiliary parameters, so cannot actually be present. We find that the only possible such poles arise on divisors of the form
\begin{gather*}
c^2 = p^j q^k t^{-l}
\end{gather*}
with $\max(0,l)<\min(j,k)$. (Strictly speaking, this only applies for $|pq|<|t|<1$, but should hold in general; note also that in that region, there are no poles depending only on $t$, $p$, and~$q$.)
\end{rem}

One thing control over the poles allows us to do is take certain limits
involving pinched contours. For instance, the case $d=\sqrt{t}$ of the
braid relation becomes singular whenever a~subsequence of $\vec{x}$ is a geometric progression of
step $t$. Since such limits occur below, we give the corresponding limit
in significant generality. With this in mind, let $[x;t]_k$ denote the
geometric progression $t^{(k-1)/2}x,t^{(k-3)/2}x,\dots,t^{(1-k)/2}x$.

\begin{prop}\label{prop:branch_k}
Let $k_1,\dots,k_m$ be a sequence of positive integers with $k_1+\cdots +k_m=n$. Then for otherwise generic parameters,
\begin{gather*}
\cK^{(n)}_{c\sqrt{t}}([x_1;t]_{k_1},\dots,[x_m;t]_{k_m}; \vec{y},v;t;p,q)\\
\qquad = \frac{\prod\limits_{1\le i\le m} \Gampq\big(c t^{1-k_i/2}x_i^{\pm 1} v^{\pm 1}\big)}
 {\Gampq\big(tc^2\big)\prod\limits_{1\le i\le m} \Gampq\big(t^{k_i}\big)
\prod\limits_{1\le i<j\le m} \Gampq\big(t^{(k_i+k_j)/2} x_i^{\pm 1}x_j^{\pm 1}\big)} \\
\qquad\quad{}\times \int \cK^{(n-1)}_c(\vec{z},[x_1;t]_{k_1-1},\dots,[x_m;t]_{k_m-1};\vec{y};t;p,q)\\
\qquad\quad{}\times \Delta^{(m-1)}_D\big(\vec{z};pq v^{\pm 1}/tc,
t^{k_1/2}x_1^{\pm 1},\dots,t^{k_m/2}x_m^{\pm 1};p,q\big).
\end{gather*}
\end{prop}

\begin{proof}
 If $m=n$, this is just the usual integral representation; in general, one
 can proceed by induction in $n-m$. Indeed, the limit $x_m\to
 t^{-(k_m+k_{m-1})/2}x_{m-1}$ of the left-hand side is the general case
 with $m-1$ geometric sequences, so it suffices to verify that the above
 formula is consistent with this limit. Before taking the limit, the
 constraint on the contour $C$ for the integrals is that (a) $C=C^{-1}$
 (corresponding to the symmetry of the integral), (b) $C$ contains
 $(pq/t)C$ (corresponding to the factors $((pq/t)z_i^{\pm 1}z_j^{\pm
 1};p,q)$ of the poles), and (c) $C$ contains every doubly-geometric
 sequence of poles converging to $0$. If $p$ is sufficiently small and
 the parameters are otherwise generic, the only obstruction to these
 conditions is the requirement that $C$ contain $t^{k_m/2}x_m$ and exclude
 $t^{-k_{m-1}/2}x_{m-1}$. Thus we can compute the limit by moving
 the contour through $t^{k_m/2}x_m$ before taking the limit; the prefactor
 $\Gampq(t^{(k_{m-1}+k_m)/2} x_m/x_{m-1})$ ensures that only the residues
 contribute to the limit, which is then straightforward to compute.
\end{proof}

Similarly, the case $d=\sqrt{t}$ of the braid relation has the following
geometric progression limit.

\begin{prop}\label{prop:braid_k}
 If $k_1,\dots,k_m$ are positive integers summing to $n$, and $u_0u_1 =
 pq/tc^2$, then
\begin{gather*}
\cK^{(n)}_{c\sqrt{t}}([x_1;t]_{k_1},\dots,[x_m;t]_{k_m};\vec{y};t;p,q)
\prod_{1\le i\le n} \Gampq\big(c u_0 y_i^{\pm 1},cu_1 y_i^{\pm 1}\big)\\
=\frac{1}{\prod\limits_{1\le i\le m}\Gampq\big(t^{k_i},t^{k_i/2} u_0 x_i^{\pm 1},t^{k_i/2} u_1 x_i^{\pm 1}\big)
\prod\limits_{1\le i<j\le m}\Gampq\big(t^{(k_i+k_j)/2} x_i^{\pm 1}x_j^{\pm 1}\big)}\\
\times \!\int\!
\cK^{(n)}_c(\vec{z},[x_1;t]_{k_1-1},\dots,[x_m;t]_{k_m-1};\vec{y};t;p,q)
\Delta^{(m)}_D\big(\vec{z};u_0,u_1,t^{k_1/2}x_1^{\pm 1},\dots,t^{k_m/2}x_m^{\pm 1};p,q\big).
\end{gather*}
\end{prop}

A natural question, given that the kernel has the above simple poles is whether we can characterize the residues along those poles.
The poles involving two $x$ or two $y$ variables can be resolved using the
$t\mapsto pq/t$ symmetry; for the simplest instance of the remaining poles,
we have the following. Note that when $y$ is specialized to a partition,
this is simply the case $k=1$ of equation (3.43) of \cite{bctheta}.

\begin{lem}\label{lem:kern_res}
The interpolation kernel has the limiting case
\begin{gather*}
\lim_{x_n\to cy_n} \cK^{(n)}_c(x_1,\dots,x_n;y_1,\dots,y_n;t;p,q)
\frac{\Gampq\big(t,c^2\big)} {\Gampq\big(c x_n^{\pm 1}y_n^{\pm 1}\big)} \\
\qquad{}= \prod_{1\le i\le n-1}
\frac{\Gampq\big(c x_i^{\pm 1} y_n, y_i^{\pm 1}/y_n\big)} {\Gampq\big(t c x_i^{\pm 1}y_n,ty_i^{\pm 1}/y_n\big)}
\cK^{(n-1)}_c(x_1,\dots,x_{n-1};y_1,\dots,y_{n-1};t;p,q).
\end{gather*}
\end{lem}

\begin{proof}
If we represent the left-hand side via the integral representation,
branching on $y_n$, the limit becomes a simple substitution, and the
resulting integral is just the case $d=t^{1/2}$ of the braid relation.
\end{proof}

One advantage of the kernel over interpolation functions is the fact that
the braid relation acts as a sort of Bailey lemma (see~\cite{SpiridonovVP:2004} for the univariate version). In particular, this
allows us to greatly simplify (and generalize to the kernel) the arguments
of \cite[Section~9]{xforms}. The main identity there is \cite[Theorem~9.7]{xforms}
(which generates the $W(E_7)$ symmetry of the elliptic Selberg integral),
which becomes the following identity in terms of the interpolation kernel.

\begin{thm}\label{thm:bailey_xform}
Let $v_0$, $v_1$, $w_0$, $w_1$, $c$, $d$, $u$ be parameters such that
$u^2=v_0v_1c^2/pq=pq/d^2w_0w_1$. Then
\begin{gather*}
\int_{C^n} \cK^{(n)}_c(\vec{z};\vec{x};t;p,q)\cK^{(n)}_d(\vec{z};\vec{y};t;p,q)
\Delta^{(n)}_S(\vec{z};v_0,v_1,w_0,w_1;t;p,q)\\
\qquad{} = \prod_{1\le i\le n} \Gampq\big(cv_0x_i^{\pm 1},cv_1x_i^{\pm 1}\big)
\prod_{1\le i\le n} \Gampq\big(dw_0y_i^{\pm 1},dw_1y_i^{\pm 1}\big) \\
\qquad\quad{}\times
\int_{C^n}
\cK^{(n)}_{c/u}(\vec{z};\vec{x};t;p,q)
\cK^{(n)}_{du}(\vec{z};\vec{y};t;p,q)
\Delta^{(n)}_S(\vec{z};v_0/u,v_1/u,w_0u,w_1u;t;p,q).
\end{gather*}
\end{thm}

\begin{proof} Use the braid relation to expand $\cK^{(n)}_c$ on the left as an integral
involving~$\cK^{(n)}_{c/u}$ and~$\cK^{(n)}_{\vphantom{c/u}u}$, then change the order of
integration and apply the braid relation again. Note that there is a~range
of parameters where the contours can all be taken to be the unit circle, so
the change in order of integration is legal, and extends to an identity of
meromorphic functions.
\end{proof}

\begin{rem} This argument is essentially the same as the proof of the multivariate
elliptic Bailey transformation, \cite[Theorem~4.9]{bctheta}, except that we
have replaced the elliptic binomial coefficients by the interpolation
kernel.
\end{rem}

The left-hand side is invariant under the natural action of $S_4$ on
$v_0$, $v_1$, $w_0$, $w_1$, and together with the above transformation gives an
action of $D_4$. As in~\cite{xforms}, this gives another identity,
corresponding to the third nontrivial double coset of $S_4$ in $D_4$.

\begin{cor}\label{cor:ker_comm}
Let $u_0u_1u_2u_3c^2d^2=p^2q^2$. Then
\begin{gather*}
\int \cK^{(n)}_c(\vec{z};\vec{x};t;p,q)\cK^{(n)}_d(\vec{z};\vec{y};t;p,q)\Delta^{(n)}_S(\vec{z};u_0,u_1,u_2,u_3;t;p,q)\\
\qquad {}= \prod_{\substack{1\le i\le n\\0\le r\le 3}} \Gampq\big(d u_r y_i^{\pm 1}\big)
\prod_{\substack{1\le i\le n\\0\le r\le 3}} \Gampq\big(c u_r x_i^{\pm 1}\big) \\
\qquad\quad{}\times \int
\cK^{(n)}_d(\vec{z};\vec{x};t;p,q)
\cK^{(n)}_c(\vec{z};\vec{y};t;p,q)
\Delta^{(n)}_S(\vec{z};pq/cdu_0,\dots,pq/cdu_3;t;p,q).
\end{gather*}
\end{cor}

This can be viewed as a sort of commutation relation; indeed, it corresponds directly to a~commutation relation for the corresponding
integral operators acting on interpolation functions, or for the formal difference operators considered below.

We note some special cases of interest. If $c$ and $d$ are both $\sqrt{t}$
(or, by the $t\mapsto pq/t$ symmetry, if both are equal to $\sqrt{pq/t}$),
the commutation relation becomes an explicit integral transformation
originally proved by van de Bult,~\cite{vandeBultFJ:2009}. If one is
$\sqrt{t}$ and the other is $\sqrt{pq/t}$, the result is a special case of
the elliptic Dixon transformation, \cite[Theorem~3.1]{xforms}.

We record the following degeneration (\`a la Propositions~\ref{prop:branch_k} and~\ref{prop:braid_k} above) for use in Section~\ref{section6} below.

\begin{prop}\label{prop:commut_k}
If $k_1,\dots,k_m$ are positive integers with sum $n$, and $u_0u_1u_2u_3=p^2q^2/tc^2$, then
\begin{gather*}
\int \cK^{(n)}_c (\vec{z};[x_1;t]_{k_1},\dots,[x_m;t]_{k_m};t;p,q) \cK^{(n)}_{\sqrt{t}}(\vec{z};\vec{y};t;p,q)
\Delta^{(n)}_S(\vec{z};u_0,u_1,u_2,u_3;t;p,q)\\
\qquad {}= \frac{
\prod\limits_{\substack{1\le i\le n\\0\le r<4}} \Gampq\big(\sqrt{t} u_r y_i^{\pm 1}\big)
\prod\limits_{\substack{1\le i\le m\\0\le r<4}} \Gampq\big(t^{(1-k_i)/2}cu_r x_i^{\pm 1}\big)}
{\prod\limits_{1\le i\le m}\Gampq\big(t^{k_j}\big)
\prod\limits_{1\le i<j\le m}\Gampq\big(t^{(k_i+k_j)/2} x_i^{\pm 1}x_j^{\pm 1}\big)}\\
\qquad\quad{}\times \int
\cK^{(n)}_c(\vec{z},[x_1;t]_{k_1-1},\dots,[x_m;t]_{k_m-1};\vec{y};t;p,q) \\
\qquad\quad{}\times
\Delta^{(m)}_D\big(\vec{z};pq/t^{1/2}cu_0,\dots,pq/t^{1/2}cu_3,t^{k_1/2}x_1^{\pm 1},\dots,t^{k_m/2}x_m^{\pm 1};p,q\big).
\end{gather*}
\end{prop}

As noted above, if $c=q^{-1/2}$, the integral equation for interpolation functions has the same right-hand side as a known difference equation.
This extends to the following difference analogue of the braid relation.

\begin{prop}\label{prop:diff_braid}
The interpolation function satisfies the generalized eigenvalue equation
\begin{gather*}
D^{(n)}_q\big(t_0,p/c^2t_0;t;p\big)_{\vec{x}}\cK^{(n)}_{q^{1/2}c}(\vec{x};\vec{y};t;p,q)=
\prod_{1\le i\le n} \theta_p\big(ct_0y_i^{\pm 1}\big) \cK^{(n)}_c(\vec{x};\vec{y};t;p,q).
\end{gather*}
\end{prop}

\begin{proof}
It suffices to prove this for the formal kernel, and thus when $\vec{y}$ is specialized to a~partition; this is simply \cite[equation~(3.34)]{bctheta}.
\end{proof}

\begin{rem}
 This can also be proved by induction using the integral representation, together with the special case $c=\sqrt{t}$ of Proposition~\ref{prop:diff_comm} below (see \cite[Theorem~7.9]{xforms} for a direct
 proof). One can also show that for $0<\ord(c)<1/2$ or generic $c$ of
 order $1/2$, $\cK^{(n)}_c(\vec{x};\vec{y};t;p,q)$ is determined up to a
 factor independent of $\vec{x}$ by the fact that
\begin{gather*}
\prod_{1\le i\le n} \theta_p\big(p^{1/2}vy_i^{\pm 1}\big)^{-1}
D^{(n)}_q\big((pq)^{1/2}v^{\pm 1}/c;t;p\big)_{\vec{x}} \cK^{(n)}_c(\vec{x};\vec{y};t;p,q)
\end{gather*}
is independent of $v$, together with the existence of a formal expansion.
Indeed, by Lemma~\ref{lem:diff_uniq} below the limit of the equation as
$p\to 0$ has no nonconstant solutions, and thus any solution of this system
of equations becomes constant in that limit; it follows that any two
nonzero solutions are proportional. (This is essentially Nakayama's lemma:
any nonzero solution must have constant leading coefficient, and thus we
can repeatedly subtract constant multiples of a~fixed solution to make the
other solution have valuation as small as we would like; i.e., expressing
that other solution as a formal limit of constant multiples of the fixed
solution.) This would allow one to develop most of the theory of the
interpolation kernel without using interpolation functions, though of
course not the Cauchy-type series expression itself (which plays a crucial
role in constructing the symmetric function variant of the formal kernel).
\end{rem}

If we view this formally as the special case $(c,d)\mapsto
\big(q^{-1/2},q^{1/2}c\big)$ of the braid relation, then we immediately find that
we obtain corresponding special cases of the Bailey transformation and the
commutation relation. (We can also obtain identities involving difference
operators alone, but postpone consideration of those to Section~\ref{section4}.) For the Bailey transformation, we have the
following.

\begin{prop}
Let $v_0$, $v_1$, $w_0$, $w_1$, $c$, $u$ be parameters such that
$u^2=v_0v_1c^2/pq=pq^2/w_0w_1$. Then
\begin{gather*}
D^{(n)}_q \big(q^{-1/2}v_0,q^{-1/2}v_1,q^{-1/2}w_0,q^{-1/2}w_1;t;p\big)_{\vec{y}}
\cK^{(n)}_c(\vec{y};\vec{x};t;p,q) \\
\qquad {}=
\prod_{1\le i\le n}\frac{\Gampq\big(cv_0x_i^{\pm 1},cv_1x_i^{\pm 1}\big)} {\Gampq\big(q^{-1/2} v_0y_i^{\pm 1},q^{-1/2} v_1y_i^{\pm 1}\big)}\\
\qquad\quad{}\times \int_{C^n} \cK^{(n)}_{c/u}(\vec{z};\vec{x};t;p,q)\cK^{(n)}_{q^{-1/2}u}(\vec{z};\vec{y};t;p,q) \Delta^{(n)}_S(\vec{z};v_0/u,v_1/u,w_0u,w_1u;t;p,q).
\end{gather*}
\end{prop}

The commutation relation becomes the following identity.

\begin{prop}\label{prop:diff_comm}
Let $u_0$, $u_1$, $u_2$, $u_3$, $c$ be parameters such that $u_0u_1u_2u_3c^2=p^2q$.
Then
\begin{gather*}
D^{(n)}_q(u_0,u_1,u_2,u_3;t;p)_{\vec{x}}
\cK^{(n)}_c(\vec{x};\vec{y};t;p,q)\\
\qquad {} =
D^{(n)}_q\big(pq^{1/2}/cu_0,\dots,pq^{1/2}/cu_3;t;p\big)_{\vec{y}}
\cK^{(n)}_c(\vec{x};\vec{y};t;p,q).
\end{gather*}
\end{prop}

We can also obtain identities by specializing one of the sets of variables
to a geometric progression. This has the effect of replacing one of the
interpolation kernels by a product of elliptic Gamma functions. (Of
course, we could replace both sets of variables by geometric progressions,
but this would simply recover results of \cite{xforms}, albeit with new
proofs.)

Specializing the braid relation in this way gives the following
generalization of the Kadell-type integral of \cite[Corollary~9.3]{xforms}.

\begin{prop}
Let $u_0$, $u_1$, $u_2$, $u_3$, $c$ be parameters such that $t^{n-1}u_0u_1u_2u_3 = pq/c^2$. Then
\begin{gather*}
\int\cK^{(n)}_c(\vec{z};\vec{x};t;p,q)\Delta^{(n)}_S(\vec{z};u_0,u_1,u_2,u_3;t;p,q)=
\prod_{\substack{1\le i\le n\\0\le r<s<4}}\! \Gampq\big(t^{n-i}u_ru_s\big)
\prod_{\substack{1\le i\le n\\0\le r<4}}\! \Gampq\big(cu_rx_i^{\pm 1}\big).
\end{gather*}
\end{prop}

\begin{rem} We can also obtain transformations in this way, but omit the (straightforward) details. The one thing one should note is that (in direct analogy to \cite[Corollary~9.13]{xforms}), the symmetry group is extended from $D_4$ to $D_6$, and we acquire an additional double coset.
\end{rem}

We also note the following curious identity, a multivariate analogue of the main result of \cite{vdBultFJ:2011}; the proof below is a direct adaptation of van de Bult's argument for the univariate case. Note that since both $\vec{x}$ and $\vec{y}$ are specialized to $\vec{z}$, there is no way to specialize this to a statement about interpolation functions.

\begin{thm}
The integral
\begin{gather*}
\int \cK^{(n)}_c(\vec{z};\vec{z};t;p,q)\Delta^{(n)}_S\big(\vec{z};u_0^{\pm 1}\sqrt{pq/c},u_1^{\pm 1}\sqrt{pq/c},u_2^{\pm 1}\sqrt{pq/c},u_3^{\pm 1}\sqrt{pq/c};t;p,q\big)
\end{gather*}
is invariant under $u_r\mapsto u_r/\sqrt{u_0u_1u_2u_3}$.
\end{thm}

\begin{proof}
Take the identity of Theorem \ref{thm:bailey_xform}, specialized so that $v_0w_0 = pq/cd$ and $\vec{y}=\vec{x}$. If we multiply both sides by
\begin{gather*}
\prod_{1\le i\le n} \Gampq\big(t_0 x_i^{\pm 1},t_1 x_i^{\pm 1}\big) \Delta^{(n)}_S(\vec{x};t;p,q)
\end{gather*}
with $t_0t_1 = pq/c^2d^2$, the integrals over $x$ on both sides are special cases of the braid relation. The result is the general case of the claimed identity.
\end{proof}

\begin{rem} As in \cite{vdBultFJ:2011}, this symmetry, together with the visible
 symmetries, generates the Weyl group $W(F_4)$.
\end{rem}

The action of the kernel on interpolation functions extends in a natural
way to an action on biorthogonal functions, generalizing the difference and
integral equations of \cite[Section~8]{xforms}.

\begin{prop}
The multivariate elliptic biorthogonal functions satisfy the following
integral equation, for $t^{2n-2}t_0t_1t_2t_3u_0u_1 = pq$,
\begin{gather*}
\prod_{1\le i\le n}
\frac{\Gampq\big(t^{n-1}t_0u_0t_1x_i^{\pm 1}/c^2\big)}
 {\Gampq\big(t_0x_i^{\pm 1},t_1x_i^{\pm 1},u_0x_i^{\pm 1}\big)}
\int \cK^{(n)}_c(\vec{z};\vec{x};t;p,q)\\
\times \tcR^{(n)}_{\blambda} (\vec{z};t_0/c\colon t_1/c,ct_2,ct_3;u_0/c,cu_1;t;p,q)
\Delta^{(n)}_S\big(\vec{z};t_0/c,u_0/c,t_1/c,pqc/t^{n-1}t_0u_0t_1;t;p,q\big) \\
=\prod_{1\le i\le n}
 \frac{\Gampq\big(t^{n-i}t_0t_1/c^2,t^{n-i}t_0u_0/c^2,t^{n-i}t_1u_0/c^2\big)}
 {\Gampq\big(t^{n-i}t_0t_1,t^{n-i}t_0u_0,t^{n-i}t_1u_0\big)}
\tcR^{(n)}_{\blambda}(\vec{x};t_0\colon t_1,t_2,t_3;u_0,u_1;t;p,q).
\end{gather*}
\end{prop}

\begin{proof}
Simply apply the usual integral equation to the binomial formula \cite[Definition~12]{bctheta} term-by-term.
\end{proof}

If we set $u_1 = 1/t^{n-1}c^2t_2$, then the biorthogonal function in the
integrand becomes an interpolation function, giving a variant of
\cite[Theorem~9.4]{xforms}, and a representation of the general biorthogonal
function as an integral involving the kernel and an interpolation function.
Ana\-ly\-ti\-cally continuing the interpolation function to another instance of
the kernel gives an analytic continuation of the biorthogonal function as a~function of the indexing partition. In the absence of a particular
application for this analytic continuation, we omit the details.

If we instead expand the biorthogonal function via the binomial formula, we obtain the following integral equation.

\begin{cor}\label{cor:int_eq_interp_ii}
For otherwise generic parameters satisfying $t^{n-1}u_0u_1u_2u_3 = pq/c^2$,
one has
\begin{gather*}
\int \cK^{(n)}_c(\vec{z};\vec{x} ;t;p,q)\cR^{*(n)}_\blambda(\vec{z};t_0,u_0;t;p,q)\Delta^{(n)}_S(\vec{z};u_0,u_1,u_2,u_3;t;p,q) \\
\qquad{} =
\prod_{1\le i\le n}\bigg(\prod_{0\le r<s<4} \Gampq\big(t^{n-i}u_ru_s\big)
\prod_{0\le r<4} \Gampq\big(cu_rx_i^{\pm 1}\big)\bigg) \\
 \qquad\quad{}\times \sum_{\bmu\subset\blambda}
\obinomE{\blambda}{\bmu}_{[t^{n-1}t_0/u_0,c^2](t^{n-1}t_0u_1,t^{n-1}t_0u_2,t^{n-1}t_0u_3);t;p,q}
R^{*(n)}_\bmu(\vec{x};t_0/c,cu_0;t;p,q).
\end{gather*}
\end{cor}

We close by mentioning a special case with an unexpected determinantal
representation. If we combine the difference equation with the explicit
formula for the case $c=\sqrt{pq/t}$, we obtain the expression
\begin{gather*}
\begin{split} & \cK^{(n)}_{(p/t)^{1/2}}(\vec{x};\vec{y};t;p,q)=
\prod_{1\le i\le n} \theta_p\big((p/t)^{1/2}t_0y_i^{\pm 1}\big)^{-1}\\
& \hphantom{\cK^{(n)}_{(p/t)^{1/2}}(\vec{x};\vec{y};t;p,q)=}{} \times D^{(n)}_q(t_0,t/t_0;t;p)_{\vec{x}}
\prod_{1\le i,j\le n} \Gampq\big((pq/t)^{1/2}x_i^{\pm 1} y_j^{\pm 1}\big).
\end{split}
\end{gather*}
The right-hand side is a sum of $2^n$ terms, each of which can be expressed as an explicit product of Gamma and theta functions. We find that although the individual terms depend on~$q$, the {\em ratios} of the terms do not. As a result, we conclude that
\begin{gather*}
\frac{\cK^{(n)}_{(p/t)^{1/2}}(\vec{x};\vec{y};t;p,q)} {\prod\limits_{1\le i,j\le n} \Gampq\big((p/t)^{1/2}x_i^{\pm 1} y_j^{\pm 1}\big)}
\end{gather*}
is independent of $q$. Setting $q=t$ gives the determinantal
expression
\begin{gather*}
\frac{\cK^{(n)}_{(p/t)^{1/2}}(\vec{x};\vec{y};t;p,q)}
 {\prod\limits_{1\le i,j\le n} \Gampq\big((p/t)^{1/2}x_i^{\pm 1} y_j^{\pm 1}\big)}
= \frac{\theta_p(t)^n\prod\limits_{1\le i,j\le n} \theta_p\big((p/t)^{1/2}x_i^{\pm 1} y_j^{\pm 1}\big)}
{(p/t)^{n(n-1)/4}\prod\limits_{1\le i<j\le n} x_iy_i\theta_p(x_i x_j,x_i/x_j,y_i y_j,y_i/y_j)}
\\
\hphantom{\frac{\cK^{(n)}_{(p/t)^{1/2}}(\vec{x};\vec{y};t;p,q)}
 {\prod\limits_{1\le i,j\le n} \Gampq\big((p/t)^{1/2}x_i^{\pm 1} y_j^{\pm 1}\big)}=}{}\times
\det_{1\le i,j\le n}
\frac{1}{\theta_p\big((p/t)^{1/2} x_i^{\pm 1} y_j^{\pm 1}\big)}.
\end{gather*}
This determinant has appeared in work of Filali \cite{FilaliG:2011} on a
certain variant of the ``8VSOS'' model, a generalization of the usual
6-vertex model. In the Macdonald limit, this becomes a known expression
for the Izergin--Korepin determinant
\begin{gather*}
\frac{\prod\limits_{1\le i,j\le n} (x_i-ty_j)(y_j-tx_i)} {\prod\limits_{1\le i<j\le n} (x_i-x_j)(y_i-y_j)}
\det_{1\le i,j\le n} \frac{1}{(x_i-ty_j)(y_j-tx_i)}
\end{gather*}
(essentially the partition function of the 6-vertex model \cite{IzerginAG:1987,KorepinVE:1982}) as a sum of Macdonald polynomials~\cite{WarnaarSO:2008}.

Applying the $t\mapsto pq/t$ symmetry gives a similar expression for $c=\sqrt{t/q}$. The case $c=t=p^{1/3}$ is of particular interest,
since this is in the intersection of the $c=(p/t)^{1/2}$ and $c=t$ cases; this implies (using Corollary \ref{cor:t_symmetry} above) that
\begin{gather*}
\prod_{1\le i,j\le n} \theta_p\big(p^{1/3}x_i^{\pm 1} y_j^{\pm 1}\big)
\prod_{1\le i<j\le n} x_iy_i\theta_p(x_i x_j,x_i/x_j,y_i y_j,y_i/y_j)^{-1}
\det_{1\le i,j\le n} \frac{1}{\theta_p\big(p^{1/3} x_i^{\pm 1} y_j^{\pm 1}\big)}
\end{gather*}
is invariant under arbitrary permutations of the $2n$ variables, recovering a result of \cite[Appendix~C]{ZinnJustinP:2013}. This corresponds to the well-known fact that the partition function for the 6-vertex model acquires additional symmetries when the parameter is a cube root of unity.

\section{Formal difference operators}\label{section4}

Although the analytic kernel most naturally corresponds to a family of
integral operators, it is difficult to make this precise, given issues with
contours; even basic questions concerning the domain of the operators are
difficult to approach. Now, we recall that when $c=q^{-n/2}$, the integral
operator at least formally becomes a difference operator. Although this is
only a sparse set of specializations, it turns out that there is a natural
analytic continuation in $c$. At first glance, this seems impossible,
since the number of different shifts appearing in the operator for
$c=q^{-n/2}$ depends on $n$; however, we can avoid this issue by working
with {\em formal} difference operators.

For $c\in \C^*$, let ${\mathfrak D}_c$ be the vector space of formal sums
of the form
\begin{gather}
\left(\sum_{\vec{k}\in \N^n} F_{\vec{k}}(\vec{x}) \prod_{1\le i\le n} T_i^{k_i}\right)T(c),\label{eq:general_formal_diff_form}
\end{gather}
where each coefficient $F_{\vec{k}}$ is a meromorphic function. We can
multiply two such formal sums using the following rules:
\begin{gather*}
T_iT_j=T_jT_i,\qquad T(c)T_i=T_iT(c),\qquad T(c)T(d)=T(cd),
\end{gather*}
and, for any meromorphic function $F$,
\begin{gather*}
T_i F(x_1,\dots,x_n)=F(x_1,\dots,x_{i-1},qx_i,x_{i+1},\dots,x_n) T_i,\\
T(c) F(x_1,\dots,x_n) =F(cx_1,\dots,cx_n) T(c).
\end{gather*}
(And, of course, we multiply meromorphic functions in the usual way.) In
this way, we obtain a product ${\mathfrak D}_c\otimes {\mathfrak D}_d\to {\mathfrak
 D}_{cd}$. Since $T(1)$ is the identity for this product, we will omit it
from the notation for ${\mathfrak D}_1$.

We call the resulting $\C^*$-graded algebra the {\em algebra of formal
 difference operators}. Note that any formal difference operator with
only finitely many nonzero coefficients acts in a natural way on the space
of meromorphic functions (i.e., right-multiply by the function, then take
the sum of the coefficients), thus justifying the name.

One important observation about formal difference operators is that a
formal difference operator of the form \eqref{eq:general_formal_diff_form}
is invertible whenever $F_0\ne 0$. (This is by the usual argument for
formal power series: if $c=1$, $F_0=1$, we invert using the power series
for $1/(1+z)$; in general, we can always extract the invertible (right)
factor $F_0(\vec{x})T(c)$ to reduce to that case.) Similarly, the algebra
of formal difference operators has no zero-divisors.

We associate a formal difference operator to the interpolation kernel in
the following way:
\begin{gather*}
{\cal D}^{(n)}_c(q,t;p):=\!
\bigg(\!
\sum_{\vec{k}\in \N^n}\!
\bigl(
2^n n!
\Res_{z_i=q^{k_i}c x_i, 1\le i\le n}
\cK^{(n)}_c(\vec{z};\vec{x};t;p,q)
\Delta^{(n)}_S(\vec{z};t;p,q)
\bigr)\! \prod_{1\le i\le n} \!\! T_i^{k_i}\!
\bigg)
T(c).
\end{gather*}
Roughly speaking, this arises by considering an integral
\begin{gather*}
\int
f(\vec{z})
\cK^{(n)}_c(\vec{z};\vec{x};t;p,q)
\Delta^{(n)}_S(\vec{z};t;p,q),
\end{gather*}
and attempting to compute it as an infinite sum of residues, taking into
account only the simplest possible residues. Note that by applying Lemma~\ref{lem:kern_res} repeatedly, we may compute the leading coefficient of ${\cal D}^{(n)}_c(q,t;p)$.

\begin{prop}
The formal difference operator ${\cal D}^{(n)}_c(q,t;p)$ has leading coefficient
\begin{gather*}
[T(c)]{\cal D}^{(n)}_c(q,t;p)=\prod_{1\le i\le j\le n}
 \frac{\Gampq(1/x_ix_j)} {\Gampq\big(1/c^2x_ix_j\big)}
\prod_{1\le i<j\le n} \frac{\Gampq\big(t/c^2x_ix_j\big)} {\Gampq(t/x_ix_j)}.
\end{gather*}
\end{prop}

For the next few lemmas, we will view the operators $D^{(n)}_q(u_0,\dots,u_{2m-1};t;p)$ as elements of~${\mathfrak D}_{q^{-1/2}}$.
The first two lemmas are direct translations of Propositions~\ref{prop:diff_braid} and~\ref{prop:diff_comm}, respectively.

\begin{lem}\label{lem:diff_op_braid_q}
For any $c$, $u$, we have
\begin{gather*}
\prod_{1\le i\le n} \theta_p\big(u x_i^{\pm 1}\big) {\cal D}^{(n)}_{q^{-1/2}c}(q,t;p)=
{\cal D}^{(n)}_c(q,t;p)D^{(n)}_q(uc,pc/u;t;p).
\end{gather*}
\end{lem}

\begin{lem}\label{lem:diff_op_comm_q}
If $u_0u_1u_2u_3 = p^2q/c^2$, then
\begin{gather*}
{\cal D}^{(n)}_c(q,t;p)
D^{(n)}_q(cu_0,cu_1,cu_2,cu_3;t;p)
=
D^{(n)}_q(u_0,u_1,u_2,u_3;t;p)
{\cal D}^{(n)}_c(q,t;p).
\end{gather*}
\end{lem}

The first lemma is particularly useful, for the following reason.

\begin{lem}\label{lem:diff_op_uniq}
 Suppose $q$ is non-torsion in $\C^*/\langle p\rangle$, and let $D\in
 {\mathfrak D}_c$ be an operator with leading coefficient $0$ such that
\begin{gather*}
\prod_{1\le i\le n} \frac{1}{\theta_p(u x_i^{\pm 1})}
D
D^{(n)}_q(uc,pc/u;t;p)
\end{gather*}
is independent of $u$. Then $D=0$.
\end{lem}

\begin{proof}
Write
\begin{gather*}
D=\sum_{\vec{k}\in \N^n} F_{\vec{k}}(\vec{x}) \prod_{1\le i\le n} T_i^{k_i}
T(c),
\end{gather*}
with $F_0(\vec{x})=0$. The fact that the given product of operators is
independent of $u$ implies that
\begin{gather*}
D D^{(n)}_q(uc,pc/u;t;p)
\end{gather*}
vanishes if we set $u=x_j$ for any $1\le j\le n$. We take $j=n$ for
notational simplicity; the other cases are analogous. If we take the
coefficient of $\prod_i T_i^{k_i} T\big(q^{-1/2}c\big)$ in this product, we obtain
a linear relation between the coefficients $F_{\vec{l}}$ for $0\le
\vec{l}\le \vec{k}$ (in the product partial order). The coefficient of
$F_{\vec{k}}$ in this relation is
\begin{gather*}
\frac{\prod\limits_{1\le i\le n} \theta_p\big(q^{-k_i}x_n/x_i,q^{k_i}x_i x_n\big)
\prod\limits_{1\le i<j\le n} \theta_p\big(t/c^2q^{k_i+k_j}x_ix_j\big)}
{\prod\limits_{1\le i\le j\le n} \theta_p\big(1/c^2q^{k_i+k_j}x_ix_j\big)},
\end{gather*}
and thus as long as $\theta_p\big(q^{-k_n}\big)\ne 0$, we obtain an expression for
$F_{\vec{k}}$ in terms of coefficients $F_{\vec{l}}$ with $\sum_i
l_i<\sum_i k_i$. Since $F_0=0$, this implies by induction that
$F_{\vec{k}}=0$.
\end{proof}

In particular, ${\cal D}^{(n)}_c(q,t;p)$ is uniquely determined by Lemma~\ref{lem:diff_op_braid_q}, and the proof of Lemma~\ref{lem:diff_op_uniq}
gives a recurrence for computing its coefficients. By inspection, that
recurrence gives us the following result.

\begin{prop}
If $q$ is not torsion in $\C^*/\langle p\rangle$, then the coefficients of
${\cal D}^{(n)}_c(q,t;p)$ have no $\vec{x}$-independent poles.
\end{prop}

In other words, the operator ${\cal D}^{(n)}_c(q,t;p)$ is well-defined whenever
$q$ is non-torsion. (This is in contrast to $\cK^{(n)}_c$, which certainly
{\em does} have poles depending on $c$ but not on the variables!)

\begin{prop}
If $c^2\in p^\Z$, then
\begin{gather*}
{\cal D}^{(n)}_c(q,t;p)=\prod_{1\le i\le j\le n} \frac{\Gampq(1/x_ix_j)} {\Gampq\big(1/c^2x_ix_j\big)}
\prod_{1\le i<j\le n} \frac{\Gampq\big(t/c^2x_ix_j\big)} {\Gampq(t/x_ix_j)}T(c).
\end{gather*}
\end{prop}

\begin{proof}Both sides have the same leading coefficient, so it suffices to show that
their difference satisfies the hypothesis of Lemma~\ref{lem:diff_op_uniq}.
Since ${\cal D}^{(n)}_c(q,t;p)$ certainly satisfies the equation, we reduce to
showing that
\begin{gather*}
\prod_{1\le i\le n} \frac{1}{\theta_p\big(u x_i^{\pm 1}\big)} \prod_{1\le i\le j\le n}
 \frac{\Gampq(1/x_ix_j)} {\Gampq\big(1/c^2x_ix_j\big)} \prod_{1\le i<j\le n} \frac{\Gampq\big(t/c^2x_ix_j\big)}
 {\Gampq(t/x_ix_j)}T(c)D^{(n)}_q(uc,pc/u;t;p)
\end{gather*}
is independent of $u$. This reduces to checking that
\begin{gather*}
\frac{\theta_p\big(p^l u x,u/p^l x\big)}{\theta_p(u x,u/x)}
\end{gather*}
is independent of $u$, where $c^2=p^l$, which in turn reduces easily to the case $l=1$.
\end{proof}

In particular, ${\cal D}^{(n)}_1(q,t;p)=1$. Plugging this into Lemma~\ref{lem:diff_op_braid_q} gives the following identification.

\begin{prop} We have ${\cal D}^{(n)}_{q^{-1/2}}(q,t;p)=D^{(n)}_q(t;p)$. More generally, for any nonnegative integer $m$, ${\cal D}^{(n)}_{q^{-m/2}}(q,t;p)$ has finite support, with theta function coefficients, and the correspon\-ding true difference operator commutes with the natural action of the
hyperoctahedral group $($by permuting and inverting the variables$)$.
\end{prop}

\begin{proof} If we write
\begin{gather*}
{\cal D}^{(n)}_{q^{-m/2}}(q,t;p)=\prod_{1\le i\le n} \frac{1}{\theta_p(u x_i^{\pm 1})}
{\cal D}^{(n)}_{q^{-(m-1)/2}}(q,t;p)D^{(n)}_q\big(uq^{-(m-1)/2},pq^{-(m-1)/2}/u;t;p\big),
\end{gather*}
we see that the case $m=1$ immediately gives the first claim, while the second claim follows by induction from the corresponding fact for $D^{(n)}_q\big(uq^{-(m-1)/2},pq^{-(m-1)/2}/u;t;p\big)$.
\end{proof}

\begin{rem} Indeed, we see that ${\cal D}^{(n)}_{q^{-m/2}}(q,t;p)$ is an operator of the form considered in~\cite{xforms} (introduced in the proof of Theorem~9.7 op.\ cit.); in that notation, we have ${\cal D}^{(n)}_{q^{-m/2}}(q,t;p) = D^{(n)}_{0,m}(t;p,q)$. We also note that a straightforward induction shows that we may replace the integral operator corresponding to $\cK^{(n)}_{q^{-m/2}}$ in the braid relation and other identities by the operator ${\cal D}^{(n)}_{q^{-m/2}}(q,t;p)$ in the same way as for $m=1$.
\end{rem}

The key identity satisfied by our formal difference operators is the
following analogue of the braid relation, Proposition~\ref{prop:kern_braid}.

\begin{prop}\label{prop:formal_op_braid}
If $t_0t_1=pq/c^2d^2$, then
\begin{gather*}
\prod_{1\le i\le n} \Gampq\big(t_0 c x_i^{\pm 1},t_1 c x_i^{\pm 1}\big)
{\cal D}^{(n)}_{cd}(q,t;p) \prod_{1\le i\le n} \Gampq\big(t_0 d x_i^{\pm 1},t_1 d x_i^{\pm 1}\big) \\
\qquad {}= {\cal D}^{(n)}_c(q,t;p)
\prod_{1\le i\le n} \Gampq\big(t_0 x_i^{\pm 1},t_1 x_i^{\pm 1}\big)
{\cal D}^{(n)}_d(q,t;p).
\end{gather*}
\end{prop}

\begin{proof} We give two arguments. The first is to use the residue definition of the
coefficients of ${\cal D}^{(n)}_{cd}(q,t;p)$, and expand using the braid relation
to obtain a limit of integrals. The natural contour conditions on the integral
cannot be satisfied, so we must first move the contour before taking the
limit; the result is a sum of residues, and gives the desired result.

The second, more algebraic argument, is to note that it suffices to show
that the operator
\begin{gather*}
\prod_{1\le i\le n} \frac{1}{\Gampq\big(t_0 c x_i^{\pm 1},t_1 c x_i^{\pm 1}\big)}
{\cal D}^{(n)}_c(q,t;p)
\prod_{1\le i\le n}\Gampq\big(t_0 x_i^{\pm 1},t_1 x_i^{\pm 1}\big)
{\cal D}^{(n)}_d(q,t;p)\\
\qquad{}\times \prod_{1\le i\le n} \frac{1}{\Gampq\big(t_0 d x_i^{\pm 1},t_1 d x_i^{\pm 1}\big)}
\end{gather*}
satisfies the hypothesis of Lemma~\ref{lem:diff_op_uniq}, since it clearly has the correct leading coefficient. This in turn is a straightforward argument using first Lemma~\ref{lem:diff_op_comm_q} then Lemma~\ref{lem:diff_op_braid_q}, and noting that this eliminates~$u$ from the expression entirely.
\end{proof}

Taking $d=c^{-1}$ gives the following result.

\begin{cor}\label{cor:diff_op_inverse}
The operators ${\cal D}^{(n)}_c(q,t;p)$ and ${\cal D}^{(n)}_{1/c}(q,t;p)$ are inverses.
\end{cor}

The name ``braid relation'' for this identity (and thus for Proposition~\ref{prop:kern_braid}) comes from the following observation. For a~nonnegative integer $m\ge 2$, consider the following involutions acting on~$(\C^*)^{m+1}$:
\begin{gather*}
s_D\colon \ (c,u,v_1,\dots,v_{m-1})\to (1/c,uc,v_1,\dots,v_{m-1}), \\
s_\Gamma\colon \ (c,u,v_1,\dots,v_{m-1})\to (cu,1/u,v_1,v_2u,v_3,\dots,v_{m-1}), \\
s_1\colon \ (c,u,v_1,\dots,v_{m-1})\to (c,u,1/v_1,v_1v_2,v_3,\dots,v_{m-1}), \\
s_2\colon \ (c,u,v_1,\dots,v_{m-1})\to (c,uv_2,v_1v_2,1/v_2,v_2v_3,v_4,\dots,v_{m-1}),
\end{gather*}
and, for $3\le k\le m-1$,
\begin{gather*}
s_k\colon \ (c,u,v_1,\dots,v_{m-1})\to
(c,u,v_1,\dots,v_{k-2},v_{k-1}v_k,1/v_k,v_kv_{k+1},v_{k+2},\dots,v_{m-1}).
\end{gather*}
Under the identification of $\Aut((\C^*)^{m+1})$ with $\GL_{m+1}(\Z)$, we
find that these are precisely the simple reflections in the standard
reflection representation of a Coxeter group of type ``$E_{m+1}$'', i.e.,
the sequence
\begin{gather*}
E_3=A_1A_2, \ E_4=A_4, \ E_5=D_5, \ E_6, \ E_7, \ E_8, \ E_9=\tilde{E}_8, \ \dots.
\end{gather*}

\begin{thm}\label{thm:formal_diff_En}
There is an assignment of a formal difference operator
\begin{gather*}
{\cal D}^{(n)}_w(g;q,t;p)
\end{gather*}
to any element $w\in W(E_{m+1})$ and any element $g\in (\C^*)^{m+1}$
satisfying
\begin{gather*}
{\cal D}^{(n)}_{s_D}(g;q,t;p)={\cal D}^{(n)}_{c(g)}(q,t;p), \\
{\cal D}^{(n)}_{s_\Gamma}(g;q,t;p)=\prod_{1\le i\le n} \frac{1}{\Gampq\big(\sqrt{pq} u(g) v_1(g)^{\pm 1}x_i^{\pm 1}\big)}, \\
{\cal D}^{(n)}_{s_i}(g;q,t;p)=1,\qquad 1\le i\le m-1
\end{gather*}
as well as the $($cocycle$)$ conditions
\begin{gather*}
{\cal D}^{(n)}_{\text{\rm id}}(g;q,t;p)=1,\\
{\cal D}^{(n)}_{w_1w_2}(g;q,t;p)={\cal D}^{(n)}_{w_1}(w_2(g);q,t;p){\cal D}^{(n)}_{w_2}(g;q,t;p).
\end{gather*}
\end{thm}

\begin{proof}
Since the simple reflections generate $W(E_{m+1})$, we need simply show
that the corresponding operators satisfy analogues of the relations of
$W(E_{m+1})$. For those relations not involving $s_D$ or $s_\Gamma$, there
is nothing to show (all operators involve are the identity); similarly, the
commutation relations between $s_D$ and $s_i$ or between $s_\Gamma$ and
$s_i$ for $i\ne 2$ are all trivial to verify. The remaining relations are
$s_D^2=\text{id}$, $s_D s_2 s_D=s_2 s_D s_2$, $s_\Gamma^2=\text{id}$, and $s_\Gamma s_D
s_\Gamma = s_D s_\Gamma s_D$. The first two relations follow easily from
the reflection principle for elliptic Gamma functions, and the third is
just Corollary \ref{cor:diff_op_inverse}. Thus the only nontrivial
relation is the braid relation $s_\Gamma s_D s_\Gamma = s_D s_\Gamma s_D$,
and the corresponding operator identity is Proposition~\ref{prop:formal_op_braid}.
\end{proof}

\begin{rem}
 A somewhat different interpretation of the univariate instance of the
 braid relation as an actual braid relation was given in
 \cite{DerkachevSE/SpiridonovVP:2013}.
\end{rem}

One application of this construction is that it associates an identity of
difference operators to any pair of words for the same element of
$W(E_{m+1})$. For instance, take $m=4$, and consider the element
\begin{gather*}
s_D s_\Gamma s_2 s_1 s_3 s_2 s_\Gamma s_D\in W(E_5)=W(D_5).
\end{gather*}
It is straightforward to verify that this normalizes the subgroup $W(D_4)$
generated by $s_\Gamma$ and $s_i$, $1\le i\le 3$, and thus any element of
that subgroup gives rise to a different representation of the corresponding
difference operator by left- and right-multiplying by elements of $W(D_4)$.
Since $s_i$ act trivially, there are a total of 8 resulting
representations, giving two transformations. After reparametrizing so that
the original operator is
\begin{gather*}
{\cal D}^{(n)}_c(q,t;p)\prod_{1\le i\le n} \Gampq\big(t_0 x_i^{\pm 1},t_1 x_i^{\pm 1},t_2 x_i^{\pm 1},t_3 x_i^{\pm 1}\big)
{\cal D}^{(n)}_d(q,t;p)
\end{gather*}
with $t_0t_1t_2t_3=(pq/cd)^2$, we obtain the representations
\begin{gather*}
\prod_{1\le i\le n} \Gampq\big(c t_0 x_i^{\pm 1},c t_1 x_i^{\pm 1}\big)
{\cal D}^{(n)}_{c/e}(q,t;p)
\prod_{1\le i\le n} \Gampq\big((t_0/e) x_i^{\pm 1},(t_1/e) x_i^{\pm 1}\big)\\
\qquad{} \times \prod_{1\le i\le n} \Gampq\big(t_2 e x_i^{\pm 1},t_3 e x_i^{\pm 1}\big)
{\cal D}^{(n)}_{de}(q,t;p)
\prod_{1\le i\le n} \Gampq\big(t_2 d x_i^{\pm 1},t_3 d x_i^{\pm 1}\big),
\end{gather*}
where $e = \sqrt{c^2t_0t_1/pq} = \sqrt{pq/d^2t_2t_3}$, and
\begin{gather*}
\prod_{\substack{1\le i\le n\\0\le r<4}} \Gampq\big(c t_r x_i^{\pm 1}\big)
{\cal D}^{(n)}_d(q,t;p)
\prod_{\substack{1\le i\le n\\0\le r<4}} \Gampq\big((pq/cdt_r) x_i^{\pm 1}\big)
{\cal D}^{(n)}_c(q,t;p)
\prod_{\substack{1\le i\le n\\0\le r<4}} \Gampq\big(t_r d x_i^{\pm 1}\big).
\end{gather*}
These, of course, are just the analogues of Theorem~\ref{thm:bailey_xform} and Corollary~\ref{cor:ker_comm} respectively.

In \cite{elldaha}, we use this cocycle of formal difference operators over $W(E_{m+1})$ (or, rather, an extension to algebraic elliptic curves) to construct isomorphisms between certain noncommutative rational varieties. (In particular, we will see that the above appearance of $W(E_{m+1})$ is related to the appearance of that Coxeter group in the theory of rational surfaces.)

One particularly nice consequence is related to the following fact.

\begin{thm}\label{thm:fourier} Suppose $q$ is non-torsion in $\C^*/\langle p\rangle$. Let ${\cal A}_c$ denote the algebra generated by the difference operators $D^{(n)}_q(u_0,u_1,u_2,u_3;t;p)$ with $u_0u_1u_2u_3 = p^2q/c^2$. Then there is an isomorphism ${\cal F}_c:{\cal A}_c\cong {\cal A}_{c^{-1}}$ such that
\begin{gather*}
{\cal F}_c(D^{(n)}_q(u_0,u_1,u_2,u_3;t;p)) = D^{(n)}_q(cu_0,cu_1,cu_2,cu_3;t;p).
\end{gather*}
Moreover, for any $D\in {\cal A}_c$, we have the identity
\begin{gather*}
D_{\vec{x}} \cK^{(n)}_c(\vec{x};\vec{y};t;p,q)= {\cal F}_c(D)^{\text{ad}}_{\vec{y}} \cK^{(n)}_c(\vec{x};\vec{y};t;p,q),
\end{gather*}
where $\text{ad}$ denotes the formal adjoint with respect to $\Delta^{(n)}_S(t;p,q)$.
\end{thm}

\begin{proof} Since the isomorphism ${\cal F}_c$ is conjugation by the (invertible) formal operator ${\cal D}^{(n)}_c(q,t;p)$, it is an isomorphism, and acts in the correct way on the generators by Lemma~\ref{lem:diff_op_comm_q}.

For the claim about $\cK^{(n)}_c$, we need merely note that it holds for the generators, and is preserved under multiplication of operators. For the generators, we need merely note that
\begin{gather*}
D^{(n)}_q(cu_0,cu_1,cu_2,cu_3;t;p)^{\text{ad}}=D^{(n)}_q\big(pq^{1/2}/cu_0,\dots,pq^{1/2}/cu_3;t;p\big),
\end{gather*}
so that the claim is simply Proposition \ref{prop:diff_comm}.
\end{proof}

\begin{rem} Modulo issues with contours, the second claim should be viewed as saying that~${\cal F}_c$ agrees with conjugation by the integral operator associated to $\cK^{(n)}_c$. Since
\begin{gather*}\begin{split}&
{\cal F}_c\bigg(
\prod_{1\le i\le n} \theta_p\big(v x_i^{\pm 1}\big)
D^{(n)}_q\big(u,pq/c^2u;t;p\big)\bigg)
={\cal F}_c\big(D^{(n)}_q\big(u,pq/c^2u,v,p/v;t;p\big)\big)\\
&\qquad{} =D^{(n)}_q(cu,pq/cu,cv,cp/v;t;p) =D^{(n)}_q(cv,cp/v;t;p)\prod_{1\le i\le n} \theta_p\big(q^{-1/2}cu x_i^{\pm 1}\big),
\end{split}
\end{gather*}
we see that, roughly speaking, ${\cal F}_c$ interchanges multiplication and difference operators (see also Lemma~\ref{lem:diff_op_braid_q} above), and thus $\cK^{(n)}_c$ can be viewed as the kernel of a generalized Fourier transform. (Indeed, in a suitable limit, $\cK^{(1)}_c$ becomes $\exp(-xy)$\dots) Although this would give a~more natural definition of ${\cal F}_c$ than the one via generators or via conjugation by ${\cal D}^{(n)}_c$, there is a significant difficulty in that we would need to show that the relevant algebra of difference operators acts faithfully on $\cK^{(n)}_c$. Although it follows from the formal difference operator approach that the operators {\em generically} act faithfully, it is difficult to determine the precise hypersurfaces on which faithfulness fails; in contrast, since formal difference operators form a domain, that definition only fails when $q$ is torsion. In addition, there are a number of natural conditions on difference operators (e.g., support, vanishing of leading coefficients along suitable divisors) that can be defined in terms of modules over the ring of formal difference operators, and are thus preserved by~${\cal F}_c$. In \cite{elldaha}, we characterize the algebra ${\cal A}^{(n)}_c$ (in fact, a somewhat larger algebra to which the claim still applies), which will enable us to show, for instance, that there is an identity of the above form in which both~$D$ and ${\cal F}_c(D)$ are (general) instances of the van Diejen Hamiltonian~\cite{vDiejenJF:1994}, up to an additive scalar. (This is a multivariate analogue of the results of~\cite{RR_Sklyanin}.)
\end{rem}

There is, in fact, a region in which we can control faithfulness of difference operators.

\begin{lem}
 Suppose $|pq|<|t|,|c|^2<1$, and let $m$ be a nonnegative integer such that $q^{-k}c^2\notin p^{\Z}$ for $1\le k\le m$. Then the $(m+1)^n$ functions
\begin{gather*}
\cK^{(n)}_c\big(q^{\vec{k}}\vec{x};\vec{y};t;p,q\big),\qquad \vec{k}\in \{0,\dots,m\}^n
\end{gather*}
are linearly independent over the field of meromorphic functions independent of $\vec{y}$.
\end{lem}

\begin{proof} When $|t|,|pq/t|<1$, the integral representation gives us an easy inductive proof that $\cK^{(n)}_c(;;t;p,q)$ has no poles depending only on $p$, $q$, $t$. As we remarked following Theorem~\ref{thm:kern_poles}, we can then use the braid relation to understand those poles depending on $c$, $p$, $q$, $t$. It turns out that any such pole has $|c|^2\le |pq|$, and thus cannot occur in the given region of parameter space. We thus conclude that
\begin{gather*}
\cK^{(n)}_c(\vec{x};\vec{y};t;p,q) \prod_{1\le i<j\le n}\big((pq/t) x_i^{\pm 1}x_j^{\pm 1},(pq/t) y_i^{\pm 1}y_j^{\pm 1};p,q\big)
\prod_{1\le i,j\le n} \big(c x_i^{\pm 1}y_j^{\pm 1};p,q\big)
\end{gather*}
is a holomorphic function of the parameters as well as the variables.
Moreover, since $|c|<1$, we find that (for generic $\vec{y}$) the
$\vec{y}$-dependent poles in $\vec{x}$ are at most order 1. As a result,
the residue of $\cK^{(n)}_c(\vec{x};\vec{y};t;p,q)$ along any such pole can
be computed via the limit from generic $c$.

Now, for $\vec{l}\in \{0,\dots,m\}^n$, consider the matrix of residues
\begin{gather*}
\Res_{\vec{x}=q^{-\vec{l}}c\vec{y}} \cK^{(n)}_c\big(q^{\vec{k}}\vec{x};\vec{y};t;p,q\big)=
\Res_{\vec{x}=q^{\vec{k}-\vec{l}}c\vec{y}}\cK^{(n)}_c(\vec{x};\vec{y};t;p,q).
\end{gather*}
For generic $\vec{y}$, this vanishes unless $k_i-l_i\ge 0$ or $k_i-l_i<-m$
(the latter coming from the possibility that $q^{-j}c^2\in p^\Z$ for some
$j>m$). Thus this matrix of residues is triangular; since the diagonal
residues are nonzero meromorphic functions of $\vec{y}$, the matrix of
residues is nonsingular. The claim follows immediately.
\end{proof}

We close by noting some simple consequences for these formal operators
arising from properties of the interpolation kernel. The simplest is the
$t\mapsto pq/t$ symmetry.

\begin{prop}
We have
\begin{gather*}
{\cal D}^{(n)}_c(q,pq/t;p)=\prod_{1\le i<j\le n} \Gampq\big(t x_i^{\pm 1}x_j^{\pm 1}\big)
{\cal D}^{(n)}_c(q,t;p) \prod_{1\le i<j\le n} \Gampq\big(t x_i^{\pm 1}x_j^{\pm 1}\big)^{-1}.
\end{gather*}
\end{prop}

\begin{proof}Using Lemma \ref{lem:diff_op_uniq}, we may immediately reduce to the case $c=q^{-1/2}$, where this is straightforward.
\end{proof}

\begin{rem} In fact, this symmetry came first (via a somewhat different approach to these operators, see~\cite{elldaha}); it was only later that it became apparent that the symmetry extended to the kernel itself.
\end{rem}

The explicit formula for the univariate interpolation kernel gives the
following expression (which can also be obtained by using the proof of
Lemma \ref{lem:diff_op_uniq} to obtain a first-order recurrence satisfied
by the coefficients).

\begin{prop}
We have
\begin{gather*}
{\cal D}^{(1)}_c(q,t;p) = \frac{\Gampq\big(1/x_1^2\big)} {\Gampq\big(1/c^2x_1^2\big)}\sum_k
q^{-k^2} \big(c^2x_1\big)^{-2k} \frac{\theta_p\big(q^{2k}c^2x_1^2\big)} {\theta_p\big(c^2x_1^2\big)}
\frac{\theta_p\big(c^2 x_1^2,c^2;q\big)_k} {\theta_p\big(q x_1^2,q;q\big)_k}T_1^k T(c).
\end{gather*}
\end{prop}

This should be compared with the formula of \cite{CooperS:2002} for the powers of the Askey--Wilson ope\-rator, as well as the elliptic analogue
considered in \cite{IRS}. There is a similar, but more complicated, expression for ${\cal D}^{(n)}_{t^{1/2}}(q,t;p)$, which we omit, as well as
the corresponding expression for ${\cal D}^{(n)}_{(pq/t)^{1/2}}(q,t;p)$.

Similar reductions to $c=q^{-1/2}$ give the following.

\begin{prop}
We have
\begin{gather*}
{\cal D}^{(n)}_c(q,1;p)=\prod_{1\le i\le n} {\cal D}^{(1)}_c(q,1;p)_{x_i},\\
{\cal D}^{(n)}_c(q,q;p)=c^{-n(n-1)/2}
\prod_{1\le i<j\le n} \frac{1}{x_i^{-1}\theta_p\big(x_i x_j^{\pm 1}\big)}
\prod_{1\le i\le n} {\cal D}^{(1)}_c(q,q;p)_{x_i} \prod_{1\le i<j\le n} x_i^{-1}\theta_p\big(x_ix_j^{\pm 1}\big).
\end{gather*}
\end{prop}

Of course, we also obtain a formula for ${\cal D}^{(n)}_c(q,pq;p)_{\vec{x}}$ coming from the $t\mapsto pq/t$ symmetry. This arises from a quasiperiodicity of the coefficients under $t\mapsto pt$, which we omit. (The construction of \cite{elldaha} shows something far stronger: if we divide by the leading coefficient, and introduce suitable additional factors, the resulting formal difference operator is not only elliptic in all parameters and variables, but extends to algebraic elliptic curves in a canonical (thus modular) way.)

\section{The kernel as symmetric function}\label{section5}

As we mentioned in the introduction, another important extension of the
formal kernel involves analytically continuing in the dimension, along the
same lines as the lifting of Koornwinder polynomials to symmetric functions
in~\cite{bcpoly}. Clearly the analytic definition by induction in the
dimension will not be of use in this regard, so we must return to the
deformed Cauchy identity definition. We thus see that our first order of
business must be to extend the interpolation functions themselves to
symmetric functions. Such an extension will be a symmetric function (more
properly, a formal series in $p$ with symmetric function coefficients)
depending on an auxiliary parameter $T$ such that when $T=t^n$ and we
specialize the variables to $x_1,1/x_1,\dots,x_n,1/x_n$, we recover the
formal series expansion of the relevant $n$-variable interpolation function.

For any partition $\mu$, we recall from \cite{bcpoly} the specialization
$\la \mu\ra_{q,t,T;a}$ of the ring $\Lambda$ of symmetric functions
defined by
\begin{gather*}\begin{split}&
p_k(\la \mu\ra_{q,t,T;a})=
\sum_{1\le i} \big(\big(q^{k\mu_i}-1\big)t^{-ki}(aT)^k+\big(q^{-k\mu_i}-1\big)t^{ki}(aT)^{-k}\big)\\
& \hphantom{p_k(\la \mu\ra_{q,t,T;a})=}{}
+ a^k \frac{1-T^k}{1-t^k} + a^{-k} \frac{1-T^{-k}}{1-t^{-k}},\end{split}
\end{gather*}
where $p_k$ denotes the power sum symmetric function. Note that the summands vanish once $\mu_i=0$, so this is a well-defined homomorphism $\Lambda\to \Q(q,t)[a,T,1/a,1/T]$. One significance of this specialization is that for any symmetric function $f$ and $n\ge \ell(\mu)$,
\begin{gather*}
f(\la \mu\ra_{q,t,t^n;a}) = f\big(\big(q^{\mu_1}t^{n-1}a\big)^{\pm 1},\dots,(q^{\mu_n}a)^{\pm 1}\big).
\end{gather*}
In particular, if we have a symmetric function lift of the interpolation
functions, its values under this specialization would be forced. (In fact,
those values suffice to uniquely determine the lift, but would require some
subtle facts about valuations to give existence.)

There is one special case of the interpolation functions which is quite
straightforward to lift. When $t^n ab=pq$, the interpolation function has
an explicit expression as a product over the variables. Thus the only
potential obstacle to lifting is formal convergence, and this turns out not
to be an issue. With this in mind, we define, for $b$ a formal parameter
with $0<\ord(b)<1$, a family of symmetric functions
\begin{gather*}
F_\lambda(\hat{x};b;q,t;p):=\exp\bigg({-}\sum_{k\ge 1} \frac{p^{k/2}p_k(\la \lambda'\ra_{t,q,1;p^{-1/2}b})p_k(\hat{x})} {k\big(1-p^k\big)}\bigg).
\end{gather*}
The constraint on $\ord(b)$ ensures that
\begin{gather*}
\ord\big(p^{k/2}p_k(\la \lambda'\ra_{t,q,1;p^{-1/2}b})\big)>0,
\end{gather*}
and thus the sum converges formally to a series of positive order, making
the exponential well-defined as well. This is indeed a lift of the ``Cauchy''
special case of the interpolation function: for any $n\ge \ell(\lambda)$,
we have
\begin{gather*}
F_\lambda\big(z_1^{\pm 1},\dots,z_n^{\pm 1};b;q,t;p\big)=
\prod_{1\le i\le n,1\le j\le \lambda_1}
\frac{\theta_p\big((p/b) t^{-\lambda'_j} q^j z_i^{\pm 1}\big)} {\theta_p\big((p/b) q^j z_i^{\pm 1}\big)}\\
\hphantom{F_\lambda\big(z_1^{\pm 1},\dots,z_n^{\pm 1};b;q,t;p\big)}{}
=R^{*(n)}_\lambda\big(z_1,\dots,z_n;pq/t^nb,b;q,t;p\big).
\end{gather*}

With this in mind, we can define more general lifted interpolation functions using connection coefficients.

\begin{defn}\label{def:lifted_interp}
 For $0<\ord(b)<1$, the lifted interpolation function $\hat{R}^{*}_\lambda(\hat{x};a,b;q,t,T;p)$ is given by the finite sum
\begin{gather*}
\hat{R}^{*}_\lambda(\hat{x};a,b;q,t,T;p) = \sum_{\mu\subset\lambda}\obinomE{\lambda}{\mu}_{[Ta/tb,Tab/pq];q,t;p}
\Delta^0_\mu\big(pq/tb^2|pq/ab;q,t;p\big)F_\mu(\hat{x};b;q,t;p).
\end{gather*}
\end{defn}

In the statement of the following result, we denote $\lc(a):=p^{-\ord(a)}a$, $\lc(b):=p^{-\ord(b)}b$, and assume that each is independent of~$p$.

\begin{thm} For $0<\ord(b)<1$, $0\le \ord(a)\le 1$, the lifted interpolation function
 is a~holomorphic formal Puiseux series in $p$ with coefficients in
 $\Lambda\otimes \Q(\lc(a),\lc(b),q,t)$, and satisfies
\begin{gather*}
\hat{R}^{*}_\lambda\big(z_1^{\pm 1},\dots,z_n^{\pm 1};a,b;q,t,t^n;p\big)\\
\qquad{} = \begin{cases}
\Delta^0_\lambda\big(t^{n-1}a/b|t^n;q,t;p\big)
R^{*(n)}_\lambda(z_1,\dots,z_n;a,b;q,t;p), & n\ge \ell(\lambda),\\
0, &\ell(\lambda)>n.
\end{cases}
\end{gather*}
Moreover, for any partition $\mu$,
\begin{gather*}
\hat{R}^*_\lambda(\la \mu\ra_{q,t,T;a};a,b;q,t,T;p)=0
\end{gather*}
unless $\mu\supset\lambda$. More generally,
\begin{gather*}
\hat{R}^*_\lambda(\la \mu\ra_{q,t,T;a};a,b;q,t,T;p)= \Delta_\lambda(Ta/tb|t/Tab;q,t;p)^{-1}\binomE{\mu}{\lambda}_{[T^2a^2/t^2,Tab/t];q,t;p}.
\end{gather*}
\end{thm}

\begin{proof} The specialization for $n\ge \ell(\lambda)$ follows from the connection coefficient identity for interpolation functions, and the claim about the coefficients is manifest from the definition. (The only nontrivial contribution comes from the elliptic binomial coefficient, but this is the restriction of an algebraic function to a Tate curve, so has rational function coefficients.) Since the interpolation functions that~$\hat{R}$ specializes to are holomorphic Puiseux series in~$p$~\cite{biorth_degens}, the same is true for $\hat{R}$.

The vanishing condition and relation to binomial coefficients hold since they hold for all sufficiently large $n$; it then follows that the specialization vanishes for $n\ge \ell(\lambda)$, since it is a~polynomial Puiseux series such that every coefficient vanishes at all partitions with at most~$n$ parts.
\end{proof}

\begin{rems} The reader should be cautioned that vanishing for $n<\ell(\lambda)$ only holds for {\em generic} values of the parameters. Indeed, when $tab=pq$, the lifted interpolation function is (up to a prefactor) $F_\lambda$, which never vanishes! Also, although the lifted interpolation function is holomorphic, the definition in general involves a great deal of cancellation; the individual terms typically have negative order.
\end{rems}

\begin{rems}More generally, one could consider the sum
\begin{gather*}
\sum_{\kappa\subset\mu\subset\lambda}\obinomE{\lambda}{\mu}_{[a/tb,ab/pq];q,t;p}\obinomE{\mu}{\kappa}_{[pq/b^2,pqT/ab];q,t;p} F_\mu\big(\hat{x};t^{-1/2}b;q,t;p\big).
\end{gather*}
This is a formal symmetric function analogue of the skew interpolation functions of~\cite{littlewood}, as can be seen by applying the specialization
\begin{gather*}
p_k\mapsto \sum_{0\le r<2n} \frac{v_r^k-v_r^{-k}}{t^{k/2}-t^{-k/2}},\qquad T\mapsto v_0\cdots v_{2n-1}.
\end{gather*}
\end{rems}

\begin{rems}
We also note that the constant ($p^0$) term of the lifted interpolation
function is essentially just the lifted interpolation polynomial of~\cite{bcpoly}:
\begin{gather*}
\lim_{p\to 0}\hat{R}^*_\lambda\big(\hat{z};a,p^{1/2}b;q,t,T;p\big)=t^{-2n(\lambda)}q^{n(\lambda')}
(-aT/t)^{|\lambda|}C^-_\lambda(t;q,t)\bar{P}^*_\lambda(\hat{z};q,t,T;a).
\end{gather*}
\end{rems}

It will be useful to have better control over the poles of the coefficients of the lifted interpolation function. This largely reduces to controlling the poles of the interpolation functions themselves. It follows by induction from the branching rule \cite[Theorem~4.16]{bctheta} that for $\ord(a)=0$,
\begin{gather*}
\Delta^0_\lambda\big(t^{n-1}a/b|t^n;q,t;p\big) R^{*(n)}_\lambda(z_1,\dots,z_n;a,b;q,t;p)
\end{gather*}
has no $z$-independent poles for generic $q,t$. Indeed, one has
\begin{gather*}
\Delta^0_\lambda\big(t^{n}a/b|t^{n+1};q,t;p\big)R^{*(n+1)}_\lambda (z_1,\dots,z_{n+1};a,b;q,t;p)\\
\quad {}= \sum_\kappa \obinomE{\lambda}{\kappa}_{[t^na/b,t](t^naz_{n+1},t^na/z_{n+1});q,t;p}
\Delta^0_\kappa\big(t^{n-1}a/b|t^n;q,t;p\big)R^{*(n)}_\kappa(z_1,\dots,z_n;a,b;q,t;p),
\end{gather*}
so the only relevant factor is
\begin{gather*}
\obinomE{\lambda}{\kappa}_{[t^na/b,t];q,t;p},
\end{gather*}
which can be controlled using the explicit product formula \cite[Corollary~4.5]{bctheta}.

This immediately implies that for generic $q$, $t$, the only poles of the
coefficients of the lifted interpolation function
$\hat{R}^*_\lambda(\hat{x};a,b;q,t,T;p)$ are functions of $T$ alone. (Any
other pole would be visible in the reduction to interpolation functions for
all sufficiently large $n$.) However, a careful look at the definition
shows that there can be no such poles. Indeed, for each factor of the summand, the parameters that appear generate a field over which $T$ is transcendental! We thus find that the coefficients of $\hat{R}^*_\lambda(\hat{x};a,b;q,t,T;p)$ lie in $\Lambda\otimes \Q(q,t)\big[\lc(a)^{\pm 1},\lc(b)^{\pm 1},T^{\pm 1}\big]$.

The specialization to ordinary interpolation functions immediately gives us
some symmetries of the lifted interpolation function, by checking that the
identity holds for $T=t^n$ for all sufficiently large $n$.

\begin{prop}
We have
\begin{gather*}
\hat{R}^*_\mu(;-a,-b;q,t,T;p) =\hat{R}^*_\mu(;a,b;q,t,T;p), \\
\hat{R}^*_\mu(;p/a,p/b;1/q,1/t,1/T;p) =\hat{R}^*_\mu(;a,b;q,t,T;p).
\end{gather*}
\end{prop}

There is also a plethystic symmetry of the following form. Let $\tau_{a;t}$ denote the endomorphism of $\Lambda$ given by
\begin{gather*}
\tau_{a;t}(p_k) = p_k + \frac{a^k-(t/a)^k}{1-t^k},
\end{gather*}
noting that $\tau_{a;t}^{-1}=\tau_{t/a;t}$, and that any two such endomorphisms
commute. In that light, we adopt the shorthand
\begin{gather*}
\tau_{a_1,\dots,a_m;t}=\prod_{1\le r\le m}\tau_{a_r;t}.
\end{gather*}
We note the particularly nice special cases
\begin{gather*}
\tau_{\sqrt{t};t}f= \tau_{-\sqrt{t};t}f=f, \\
(\tau_{1;t}f)(z_1,\dots,z_n)=f(z_1,\dots,z_n,1), \\
(\tau_{-1;t}f)(z_1,\dots,z_n)=f(z_1,\dots,z_n,-1).
\end{gather*}

\begin{prop}
The function $\tau_{a;t} \hat{R}^*_\mu(;a,b;q,t,T/a;p)$ is independent of $a$.
\end{prop}

\begin{proof}
We equivalently need to show
\begin{gather*}
\tau_{t/a',a;t} \hat{R}^*_\mu(\hat{x};a,b;q,t,T/a;p) =\hat{R}^*_\mu(\hat{x};a',b;q,t,T/a';p)
\end{gather*}
for all $a$, $a'$. If we specialize to $T=t^{n+m}a$, $a'=t^m a$ for $m\ge 0$, $n$ sufficiently large, and specialize $\hat{x}$ to $z_1^{\pm
 1},\dots,z_n^{\pm 1}$, this becomes an identity of ordinary interpolation functions, \cite[equation~(3.43)]{bctheta}. The claim then follows in the usual way.
\end{proof}

\begin{rem}
Note that this symmetry gives rise to the expression
\begin{gather*}
\hat{R}^*_\mu(;a,b;q,t,T;p)=\tau_{Ta,t/a;t}\hat{R}^*_\mu(;Ta,b;q,t,1;p),
\end{gather*}
giving an alternate argument for the lack of poles depending only on $T$.
\end{rem}

As in the Koornwinder case, a major benefit of lifting to symmetric functions is the action of a slightly modified Macdonald involution. Recall from \cite{bcpoly} that $\tilde\omega_{q,t}$ is the involution acting on symmetric functions by
\begin{gather*}
\tilde\omega_{q,t}p_k=(-1)^{k-1}\frac{q^{k/2}-q^{-k/2}}{t^{k/2}-t^{-k/2}} p_k,
\end{gather*}
and satisfies
\begin{gather*}
(\tilde\omega_{q,t}f)(\la \mu\ra_{t,q,1/T;-\sqrt{qt}/a}) =f(\la \mu'\ra_{q,t,T;a}).
\end{gather*}
Also note that
\begin{gather*}
\tilde\omega_{q,t} = \tilde\omega_{1/q,1/t},\\
\tilde\omega_{q,t} \tau_{a;t} = \tau_{-a/\sqrt{qt};1/q} \tilde\omega_{q,t},\\
\tilde\omega_{q,t}^{-1} = \tilde\omega_{1/t,1/q}.
\end{gather*}

\begin{prop}
The lifted interpolation function satisfies the symmetry
\begin{gather*}
\tilde\omega_{q,t}\hat{R}^*_\mu(;a,b;q,t,T;p)= \hat{R}^*_{\mu'}\big(;-a/\sqrt{qt},-b/\sqrt{qt};1/t,1/q,T;p\big).
\end{gather*}
\end{prop}

\begin{proof}Indeed, we can verify this by direct calculation in the Cauchy case $a=pq/Tb$, and the connection coefficients transform correctly.
\end{proof}

The Cauchy and Littlewood identities of \cite{littlewood} directly
translate to the lifted interpolation functions; again, we need simply
observe that the claim holds for a sufficiently general class of
specializations to ordinary interpolation functions.

\begin{prop}
If $\ord(a),\ord(b)=0$, then
\begin{gather*}
\sum_{\mu}
\Delta_\mu(ab/t(pqt)^{1/2}|;q,t;p)
\hat{R}^*_{\mu}\big(\hat{x};a,(pqt)^{1/2}/b;q,t,1;p\big)
\hat{R}^*_{\mu}\big(\hat{y};b,(pqt)^{1/2}/a;q,t,1;p\big) \\
\qquad {}= \prod_{i,j} \frac{\big((pqt)^{1/2}x_iy_j;p,q\big)} {\big((pq/t)^{1/2}x_iy_j;p,q\big)}.
\end{gather*}
\end{prop}

\begin{prop}
If $\ord(a)=0$, then
\begin{gather*}
\sum_{\mu} \Delta_\mu\big(a^2/t(pqt)^{1/2}|;q,t^2;p\big)\hat{R}^*_{\mu^2}\big(\hat{x};ta,(pqt)^{1/2}/a;q,t,1;p\big)
=\prod_{i<j} \frac{\big((pq t)^{1/2}x_ix_j;p,q\big)} {\big((pq/t)^{1/2}x_ix_j;p,q\big)}.
\end{gather*}
\end{prop}

Applying the modified Macdonald involution to the latter sum immediately
gives a dual Littlewood identity.

\begin{prop}
For $\ord(a)=0$,
\begin{gather*}
\sum_{\mu} \Delta_{\mu}\big(a^2/(pt^3/q)^{1/2}|;q^2,t;p\big)
\hat{R}^*_{2\mu}\big(\hat{x};a,(pt/q)^{1/2}/a;q,t,1;p\big) \\
\qquad {}= \prod_i
\frac{\big((pq^3t)^{1/2} x_i^2;p,q^2\big)} {\big((pq/t)^{1/2} x_i^2;p,q^2\big)} \prod_{i<j} \frac{\big((pq t)^{1/2}x_ix_j;p,q\big)} {\big((pq/t)^{1/2}x_ix_j;p,q\big)}.
\end{gather*}
\end{prop}

\begin{proof}
This is a straightforward exercise in duality: the key point is that the
logarithm of the right-hand side of the original Littlewood identity has a simple
expression:
\begin{gather*}
\sum_{k\ge 1}
\frac{\big(t^{-k/2}-t^{k/2}\big)e_2[p_k(\hat{x})]}{k\big(p^{-k/2}-p^{k/2}\big)\big(q^{-k/2}-q^{k/2}\big)},
\end{gather*}
so we need simply apply the Macdonald involution to the terms of this sum and simplify.
\end{proof}

At this point, it is relatively straightforward to come up with a candidate for the lifted kernel.

\begin{defn}
For $0<\ord(c)<1$, the {\em lifted kernel} is defined by the following sum:
\begin{gather*}
\hat{K}_c(\hat{x};\hat{y};q,t,T;p)
:= \frac{\Gampqt\big(ct_0u_0,tcu_0/Tt_0,tct_0/Tu_0,t^2c/Tt_0u_0\big)}
 {\Gampqt\big(cTt_0u_0,tcu_0/t_0,tct_0/u_0 ,t^2c/T^2t_0u_0\big)}\\
\qquad {} \times \prod_i
\frac{((pq/cu_0)x_i,(pqTu_0/tc)x_i;p,q)}
 {(cu_0x_i,(ct/Tu_0)x_i;p,q)}
\prod_i
\frac{((pq/ct_0)y_i,(pqTt_0/tc)y_i;p,q)}
 {(ct_0y_i,(ct/Tt_0)y_i;p,q)}\\
\qquad{} \times \sum_{\mu}
\Delta_\mu\big(T^2t_0u_0/t^2c|pqT/tc^2,pqTt_0u_0/t^2c;q,t;p\big) \\
\qquad\qquad{} \times \hat{R}^{*}_\mu(\hat{x};t_0,ct/Tu_0;q,t,T;p)
\hat{R}^{*}_\mu(\hat{y};u_0,ct/Tt_0;q,t,T;p),
\end{gather*}
where $t_0$, $u_0$ are auxiliary parameters with $\ord(t_0)=\ord(u_0)=0$.
\end{defn}

It is fairly straightforward to relate this to the formal kernel. The only
nontrivial issue is that the $\Delta_\mu$ symbol has a pole at $T=t^n$, so
unlike for the lifted interpolation function, we must be careful about
order of specialization. In particular, we must specialize
one or both of the sets of variables {\em before} setting $T\to t^n$;
specializing $\hat{y}\to y_1,1/y_1,\dots,y_n,1/y_n$ has the effect of
cancelling the pole at $T=t^n$, making the remaining limits commute.
We thus find that for all $n\ge 0$,
\begin{gather*}
\lim_{T\to t^n}
\hat{K}_c\big(z_1^{\pm 1},\dots,z_n^{\pm 1};w_1^{\pm 1},\dots,w_n^{\pm 1};q,t,T;p\big)\\
\qquad {}= \prod_{1\le i\le n} \Gampq\big(t^{i-n}c^2,t^i\big) K^{(n)}_c(z_1,\dots,z_n;w_1,\dots,w_n;q,t;p).
\end{gather*}
Since this is independent of $t_0$, $u_0$ for all $n$, the same is true for $\hat{K}_c$, making the latter well-defined.

\begin{prop}
The lifted kernel satisfies the duality
\begin{gather*}
\tilde\omega_{q,t;\hat{x}} \tilde\omega_{q,t;\hat{y}}\hat{K}_c(\hat{x};\hat{y};q,t,T;p)=\hat{K}_c(\hat{x};\hat{y};1/t,1/q,T;p).
\end{gather*}
\end{prop}

For the lifted kernel and interpolation function to be useful, we need to
be able to substitute them into integral identities, and thus need to have
similar symmetric function analogues of the elliptic Selberg integral.
This is mostly straightforward, since in any case in which the elliptic
Selberg integral reduces as $p\to 0$ to a Koornwinder integral, the ratio
between the two integrands is essentially a symmetric function. For
instance, if $\ord(a)>0$, we find
\begin{gather*}
\frac{(a;q)^n\prod\limits_{1\le i<j\le n} \big(a z_i^{\pm 1} z_j^{\pm 1};q\big)}
{(ta;q)^n\prod\limits_{1\le i<j\le n} \big(t a z_i^{\pm 1} z_j^{\pm 1};q\big)}
=\prod_{i<j} \frac{(a x_i x_j;q)} {(t a x_i x_j;q)}\biggr|_{\hat{x}=z_1,1/z_1,\dots,z_n,1/z_n},
\end{gather*}
making it straightforward to express
\begin{gather*}
\frac{\Delta^{(n)}_S(\vec{z};q,t;p)} {\Delta^{(n)}_S(\vec{z};q,t;0)}
\end{gather*}
as a specialization $\hat{x}\mapsto z_1,1/z_1,\dots,z_n,1/z_n$
of a symmetric function, and similarly for the univariate factors.

As a result, to extend an identity involving integrals of formal kernels to
an identity for the lifted kernel, it suffices to understand integrals of
symmetric functions against the Koornwinder density
\begin{gather*}
\lim_{p\to 0}
\Delta^{(n)}_S \big(\vec{z};t_0,t_1,t_2,t_3,p^{1/2}t_4,p^{1/2}t_5;t;p,q\big) \\
\qquad {}= \frac{(q;q)^n}{(t;q)^n n!}
\prod_{1\le i\le n}
\frac{\big(z_i^{\pm 2};q\big)}{\prod\limits_{0\le r<4} \big(t_r z_i^{\pm 1};q\big)}
\prod_{1\le i<j\le n}
\frac{\big(z_i^{\pm 1}z_j^{\pm 1};q\big)}{\big(t z_i^{\pm 1}z_j^{\pm 1};q\big)}
\prod_{1\le i\le n} \frac{dz_i}{2\pi\sqrt{-1}z_i}.
\end{gather*}
(The corresponding evaluation, a limit of \eqref{eq:ell_Selberg_eval}, was
first shown by Gustafson in \cite{GustafsonRA:1990}, but for our purposes
the fact that the corresponding orthogonal polynomials (introduced by
Koornwinder in \cite{KoornwinderTH:1992}) are well-behaved is crucial.)
That is, we want a linear functional $I_K(;q,t,T;t_0,t_1,t_2,t_3)$ such
that for otherwise generic parameters and any symmetric function $f$,
\begin{gather*}
\lim_{T\to t^n}I_K(f;q,t,T;t_0,t_1,t_2,t_3)=I^{(n)}_K(f(z_1,1/z_1,\dots,z_n,1/z_n);q,t;t_0,t_1,t_2,t_3),
\end{gather*}
where $I^{(n)}_K$ denotes the $n$-dimensional Koornwinder integral, normalized to have integral~1. That is, for any symmetric Laurent polynomial $g$,
\begin{gather*}
I^{(n)}_K(g(z_1,\dots,z_n);q,t;t_0,t_1,t_2,t_3) =
\frac{\int g(\vec{z}) \Delta^{(n)}_K(\vec{z};t_0,t_1,t_2,t_3;q,t)} {\int \Delta^{(n)}_K(\vec{z};t_0,t_1,t_2,t_3;q,t)}.
\end{gather*}
Such integrals were already considered in \cite{bcpoly}; we will, however, need some slightly better control over the poles. The key idea of the construction in \cite{bcpoly} is that the normalized integral of a~symmetric Laurent polynomial against the Koornwinder density can be computed by expan\-ding the polynomial in the corresponding orthogonal polynomials and taking the constant term. This extends immediately to symmetric functions
using the symmetric function analogues of the Koornwinder polynomials.

To control the poles, it will be useful to take a slightly different approach. Rather than take as the basic identity the fact that $K_\lambda$ integrates to $\delta_{\lambda 0}$, we use the analogue of Kadell's lemma, which here gives a formula for the integral of a suitable interpolation
polynomial against the Koornwinder density. Thus (where $I_K$ denotes the ``virtual Koornwinder integral'' of~\cite{bcpoly}) we have
\begin{gather*}
I_K(\tilde{P}^*_\lambda(;q,t,T;t_0);q,t,T;t_0,t_1,t_2,t_3) \\
\qquad {}=
\frac{(-t_0T/t)^{-|\lambda|}
q^{-n(\lambda')}
t^{2n(\lambda)}
C^0_\lambda(T,Tt_0t_1/t,Tt_0t_2/t,Tt_0t_3/t;q,t)}
{C^0_\lambda\big(T^2t_0t_1t_2t_3/t^2;q,t\big)
C^-_\lambda(q,t;q,t)}.
\end{gather*}
In particular, we can compute the virtual Koornwinder integral of a given symmetric function by expanding it in lifted interpolation polynomials and
specializing as above. Now, $\tilde{P}^*_\lambda$ is monic, and
\begin{gather*}
(t_0T)^{|\lambda|}\tilde{P}^*_\lambda(;q,t,T;t_0)
\end{gather*}
has no poles for generic $q$, $t$. In particular, in the expansion of the
Macdonald polynomial $P_\lambda(;q,t)$ in terms of lifted interpolation
polynomials, the only poles for generic $q$, $t$ are at $t_0=0$ or $T=0$.
Since (by \cite[Theorem~6.16]{bcpoly}) this expansion is triangular with
respect to the inclusion partial order, we find after integrating
term-by-term that
\begin{gather*}
C^0_\lambda\big(T^2t_0t_1t_2t_3/t^2;q,t\big)
I_K(P_\lambda(;q,t);q,t,T;t_0,t_1,t_2,t_3) \in \Q(q,t)[T,1/T,t_0,1/t_0,t_1,t_2,t_3].
\end{gather*}
The pole at $t_0=0$ can be removed by symmetry; the pole at $T=0$ can also be removed using the explicit formulas for that case in \cite{bcpoly}.

Since the Macdonald polynomials are a basis for generic $q$ and $t$, a
similar statement applies to the poles of the integral of an arbitrary
symmetric function.

\begin{lem}
For any symmetric function $f$ of degree $\le k$,
\begin{gather*}
\prod_{0\le i,j<k} \big(1-q^jt^{-2-i}T^2t_0t_1t_2t_3\big) I_K(f;q,t,T;t_0,t_1,t_2,t_3)\in \Q(q,t)[T,t_0,t_1,t_2,t_3].
\end{gather*}
\end{lem}

We recall from \cite[Corollary~7.6]{bcpoly} the following symmetries of the virtual Koornwinder integral (after fixing a couple of typos):
\begin{gather}
I_K(f;q,t,T;t_0,t_1,t_2,t_3) =I_K(f;1/q,1/t,1/T;1/t_0,1/t_1,1/t_2,1/t_3)\nonumber\\
\hphantom{I_K(f;q,t,T;t_0,t_1,t_2,t_3)}{} =I_K\left(
\tilde{\omega}_{q,t}f ;1/t,1/q,T;\frac{-t_0}{\sqrt{qt}},\frac{-t_1}{\sqrt{qt}},\frac{-t_2}{\sqrt{qt}},\frac{-t_3}{\sqrt{qt}}\right)\nonumber\\
\hphantom{I_K(f;q,t,T;t_0,t_1,t_2,t_3)}{} =I_K(\tau_{t_0,t_1;t}f;q,t,T t_0 t_1/t;t/t_1,t/t_0,t_2,t_3).\label{eq:koorn_tau_symm}
\end{gather}
(We can double-check these identities by setting $f=\tilde{P}^*_\lambda(;q,t,T;t_0)$.)

This last symmetry generates an action of $W(D_4)$ on the parameters, which gives rise to a~symmetry
\begin{gather*}
I_K(f;q,t,T;t_0,t_1,t_2,t_3) =I_K\big(\tau_{t_0,t_1,t_2,t_3;t}f;q,t,T t_0
t_1t_2t_3/t^2;t/t_0,t/t_1,t/t_2,t/t_3\big).
\end{gather*}
It is tempting here to specialize the parameters so that $T=t^n$ and $Tt_0t_1t_2t_3/t^2=t^{n'}$, so that both sides become finite integrals.
The difficulty, of course, is that the specialization to a~finite-dimensional integral only works for otherwise generic parameters, so
we need to ensure that the direction in which we take the limit has no effect. This turns out to be a problem, for the simple reason that the
virtual integral has a pole when $T^2t_0t_1t_2t_3/t^2\in t^\N$! As a~result, we cannot expect to obtain an identity of finite-dimensional
integrals from this symmetry.

Despite this fact, it turns out that the symmetry is quite useful! When
applying the virtual integral below, we will in general have little control
over the parameters of the Koornwinder integral, and in at least one case
find ourselves having to understand the limit in a case when the direction
of the limit is important. Since the polar divisor of the integral has
multiplicity 1 at the generic point with $T^2t_0t_1t_2t_3/t^2\in t^\N$, in
order to compute the limits in a general direction, we only need to
understand the limits in two distinct directions. The symmetry, in
particular, gives us two directions in which we can express the limit as a
finite-dimensional integral.

We thus obtain the following, in the special case of interest below.

\begin{lem}\label{lem:koorn_On_disc}
 Let $n$, $n'$ be integers and let $t_0$, $t_1$, $t_2$, $t_3$ be parameters such that $t_0t_1t_2t_3=t^{n'+2-n}$. Then for any symmetric
 function $f$,
\begin{gather*}
\lim_{T\to t^n} I_K (f;q,t,T;t_0,t_1,t_2,t_3) \\
\qquad {}=
\frac{1}{2}I^{(n)}_K\big(f\big(z_1^{\pm 1},\dots,z_n^{\pm 1}\big);q,t;t_0,t_1,t_2,t_3\big) \\
\qquad\quad {}+ \frac{1}{2} I^{(n')}_K\big(
\big(\tau_{t_0,t_1,t_2,t_3;t}f\big)(z_1^{\pm
 1},\dots,z_{n'}^{\pm 1});q,t;t/t_0,t/t_1,t/t_2,t/t_3\big).
\end{gather*}
\end{lem}

\begin{proof}
Fix $t_0$, $t_1$, $t_2$, and define a function
\begin{gather*}
g(u,v)=I_K\big(f;q,t,t^nu;t_0,t_1,t_2,t^{n'+2-n}v/t_0t_1t_2u\big).
\end{gather*}
Then $(1-uv)g(u,v)$ is holomorphic at $u=v=1$, which implies that
\begin{gather*}
\lim_{u\to 1} g(u,u) = \lim_{u\to 1} \frac{g(1,u)+g(u,1)}{2}.
\end{gather*}
By inspection, $g(1,u)$ is an $n$-dimensional integral, while $g(u,1)$ becomes an $n'$-dimensional integral once we apply the symmetry; in each
case, the resulting expression is holomorphic at $u=1$.
\end{proof}

\begin{rem} We will only need the cases when $(t_0,t_1,t_2,t_3)$ is one of $\big(1,-1,\sqrt{t},-\sqrt{t}\big)$ or $\big(1,-t,\sqrt{t},-\sqrt{t}\big)$, which were implicit in \cite[Proposition~8.4]{bcpoly}. Since the discussion there was invalid, however, it seemed appropriate to give a correct (and more general) proof here.
\end{rem}

It will be useful to know how various natural products transform under duality and the homomorphisms $\tau_{b_1,\dots,b_m;t}$. The key facts are the liftings
\begin{gather*}
\prod_{1\le i\le n} \big(a z_i^{\pm 1};q\big)^{-1}=
\exp\bigg[\sum_{1\le k} \frac{a^k p_k}{\big(1-q^k\big)}\bigg]
\bigg|_{z_1^{\pm 1},\dots,z_n^{\pm 1}}, \\
\frac{(a;q)^n\prod\limits_{1\le i<j\le n} \big(a z_i^{\pm 1} z_j^{\pm 1};q\big)}
 {(ta;q)^n\prod\limits_{1\le i<j\le n} \big(t a z_i^{\pm 1} z_j^{\pm 1};q\big)} =
 \exp\bigg[\sum_{1\le k} \frac{\big(p_k^2-p_{2k}\big)a^k\big(t^k-1\big)}{2k\big(1-q^k\big)}\bigg]
\bigg|_{z_1^{\pm 1},\dots,z_n^{\pm 1}},
\end{gather*}
valid whenever $|a|<1$; given the expressions on the right, it is straightforward to apply either homomorphism.

For duality, we have the following correspondences; in each case, we take
$\ord(q)=\ord(t)=0$, and choose the remaining parameters so that the Gamma
functions have arguments of order in $[0,1]$. Then the claim is that {\em
 if we divide by the limit as $p\to 0$}, the residual functions are
related by $\omega_{q,t}$. For interaction factors, we have (recalling
the formal symmetry $\Gamm{p,1/q}(z)=1/\Gampq(qz)$):
\begin{gather*}
\frac{\Gampq(ta)^n\prod\limits_{1\le i<j\le n} \Gampq\big(ta z_i^{\pm 1}z_j^{\pm 1}\big)}
 {\Gampq(a)^n\prod\limits_{1\le i<j\le n} \Gampq\big(a z_i^{\pm 1}z_j^{\pm 1}\big)} \\
 \qquad{}\mapsto
\frac{\Gamm{p,t^{-1}}(a/q)^n\prod\limits_{1\le i<j\le n} \Gamm{p,t^{-1}}\big((a/q) z_i^{\pm 1}z_j^{\pm 1}\big)}
 {\Gamm{p,t^{-1}}(a)^n\prod\limits_{1\le i<j\le n} \Gamm{p,t^{-1}}\big(a z_i^{\pm 1}z_j^{\pm 1}\big)}
\prod_{1\le i\le n} \frac{\Gamm{p,t^{-2}}\big(a z_i^{\pm 2}/qt\big)}
 {\Gamm{p,t^{-2}}\big(a z_i^{\pm 2}\big)}.
\end{gather*}
If $a=\sqrt{pq/t}$, we can take the
square root to obtain
\begin{gather*}
\Gampq\big(\sqrt{pq/t}\big)^n
 \prod_{1\le i<j\le n}\Gampq\big(\sqrt{pq/t}z_i^{\pm 1}z_j^{\pm 1}\big) \\
\qquad {}\mapsto
\Gamm{p,t^{-1}}\big(\sqrt{pq/t}\big)^n
\prod_{1\le i<j\le n}\Gamm{p,t^{-1}}\big(\sqrt{pq/t}z_i^{\pm 1}z_j^{\pm 1}\big)
\prod_{1\le i\le n} \Gamm{p,t^{-2}}\big(\sqrt{pq/t} z_i^{\pm 2}\big).
\end{gather*}
For univariate factors, we have
\begin{gather*}
\Gampq\big(a z_i^{\pm 1}\big)\mapsto \Gamm{p,t^{-1}}\big(\big({-}a/\sqrt{qt}\big) z_i^{\pm 1}\big), \\
\prod_{1\le i\le n} \Gampq\big(z_i^{\pm 2}\big) \mapsto
\prod_{1\le i\le n} \Gamm{p,t^{-2}}\big(z_i^{\pm 2}/t,z_i^{\pm 2}/qt\big).
\end{gather*}
Combining all of the above gives
\begin{gather*}
\Delta^{(n)}_S(\vec{z};\dots,t_r,\dots;t;p,q) \mapsto
\Delta^{(n)}_S\big(\vec{z};\dots,-t_r/\sqrt{qt},\dots;1/q;p,1/t\big).
\end{gather*}

For $\tau_{b;t}$, we record only the transformation of the density, in the form
\begin{gather*}
\tau_{b_1,\dots,b_m;t} \prod_{1\le r<s\le m} \Gampqt(b_rb_s)
\prod_{1\le r\le m,1\le s\le l} \Gampqt(b_rc_s)
\Delta^{(n)}_S(\vec{z};b_1,\dots,b_m,c_1,\dots,c_l;t;p,q)\\
{}=
\prod_{1\le r<s\le m} \Gampqt\big(t^2/b_rb_s\big)
\prod_{1\le r\le m,1\le s\le l} \Gampqt(tc_s/b_r)
\Delta^{(n)}_S(\vec{z};t/b_1,\dots,t/b_m,c_1,\dots,c_l;t;p,q).
\end{gather*}

\section{The Littlewood kernel}\label{section6}

If we translate Conjecture L1 of \cite{littlewood} into a statement about the kernel, we obtain the following, which turns out to be surprisingly straightforward via our present methods.

\begin{thm}\label{thm:L1_kern}
The interpolation kernel satisfies the integral identity
\begin{gather*}
\int \cK^{(2n)}_c\big(t^{\pm 1/2}\vec{z};\vec{y};t;p,q\big)
\Delta^{(n)}_S\big(\vec{z};v_0,v_1,v_2,v_3;t^2;p,q\big) \\
\qquad {}= \prod_{\substack{1\le i\le 2n\\0\le r\le 3}}
\Gampq\big(t^{-1/2}cv_r y_i^{\pm 1}\big) \\
\qquad\quad{}\times \int
\cK^{(2n)}_c\big(t^{\pm 1/2}\vec{z};\vec{y};t;p,q\big)
\Delta^{(n)}_S\big(\vec{z};pqt/c^2v_0,pqt/c^2v_1,pqt/c^2v_2,pqt/c^2v_3;t^2;p,q\big),
\end{gather*}
subject to the balancing condition $v_0v_1v_2v_3=\big(pqt/c^2\big)^2$.
\end{thm}

\begin{proof} This identity certainly holds in the limit $p\to 0$, $c\sim p^{1/2}$, $v_r\sim 1$, as then both integrals become Koornwinder integrals. Moreover, if we divide both sides by the common limit, then both sides have formal Puiseux series expansions in $p$ with rational function coefficients. It thus suffices to show that the two sides agree for a~Zariski dense set of parameters consistent with this scaling.

Now, suppose we already know a particular case of the identity, with
parameters given by $(c,v_0,v_1,v_2,v_3)$. Using this, it turns out to be
relatively straightforward to establish that the identity also holds in the
case $\big(t^{1/2}c,v_0/t^2,v_1,v_2,v_3\big)$. Indeed, starting with the integral
on the left, we can expand $\cK^{(2n)}_{t^{1/2}c}$ using the braid
relation, in such a way that after exchanging the order of integration
(which is not a problem as long as all parameters are inside the unit
circle), the inner integral becomes the known instance of the
transformation. Apply that instance, then exchange the order of
integration again. At this point, the inner integral is of the form to
which commutation applies (in the form of Proposition \ref{prop:commut_k}).
After commutation, we obtain an integral over two sets of $n$ variables,
one of which we can simplify using Proposition \ref{prop:braid_k}. The
resulting integral is precisely the desired right-hand side.

Now, the identity trivially holds whenever $v_0v_1 = pqt/c^2$, and thus
a simple induction using the preceding paragraph shows that it holds
when $v_0v_1 = pqt^k/c^2$ for any integer $k\le 1$. This is a Zariski
dense set of parameters, and thus the identity holds in general.
\end{proof}

\begin{rems}
 This is dual to Theorem \ref{thm:L2_kern} below, in the sense that if we
 analytically continue both sides in the dimension and apply the modified
 Macdonald involution, we obtain the ana\-ly\-tic continuation of Theorem~\ref{thm:L2_kern}. In particular, if the reader prefers difference
 operators to degenerate integral operators, the reader may first prove
 Theorem~\ref{thm:L2_kern} (say by following the argument given in the remark following said theorem), then apply duality.
\end{rems}

\begin{rems}
An alternate approach involves taking $v_0v_1=t^{2-2n}q^{-m}$, so that the
transformation becomes an identity of theta functions which, when $\vec{y}$
is a suitable partition, becomes \cite[Theorem~4.7]{littlewood}.
\end{rems}

An interesting special case of this transformation comes when $v_2v_3=pq$.
In that case, the left-hand side is independent of $v_2$, while the
right-hand side is (up to simple gamma factors) independent of $v_0$.
We thus immediately obtain the following corollary.

\begin{cor}
The integral
\begin{gather}
\prod_{1\le i\le 2n}\frac{1}{
 \Gampq\big(\sqrt{pqt}v^{\pm 1} x_i^{\pm 1}/c\big)}
\int \cK^{(2n)}_c\big(t^{\pm 1/2}\vec{z};\vec{x};t;p,q\big)
\Delta^{(n)}_S\big(\vec{z};\sqrt{pq}t v^{\pm 1}/c^2;t^2;p,q\big)\label{eq:litt_kern_defn}
\end{gather}
is independent of $v$.
\end{cor}

This suggests the following definition.

\begin{defn}
 The {\em Littlewood kernel} is the meromorphic function ${\cal L}^{(2n)}_c(\vec{x};t;p,q)$ defined for $|p|,|q|<1$ by~\eqref{eq:litt_kern_defn}.
\end{defn}

Note that replacing $\vec{x}$ by $-\vec{x}$ has the same effect on $\cL^{(2n)}_c(\vec{x};t;p,q)$ as negating~$c$ or negating~$v$, and thus
\begin{gather*}
 \cL^{(2n)}_c(\vec{x};t;p,q) =\cL^{(2n)}_c(-\vec{x};t;p,q)=\cL^{(2n)}_{-c}(\vec{x};t;p,q),
\end{gather*}
so that $\cL^{(2n)}_c$ is actually a function of $c^2$. Similarly, although
the right-hand side involves two choices of square root ($\sqrt{pq}$ and
$\sqrt{t}$), either can be negated without changing the function.

When $\vec{x}$ is specialized to a geometric progression, the integral on
the right becomes an elliptic Selberg integral, thus giving an explicit
expression.

\begin{prop}
We have
\begin{gather*}
\cL^{(2n)}_c\big(t^{2n-1}v,\dots,v;t;p,q\big) = \prod_{1\le i\le n}
\frac{\Gampq\big(c^2v^2t^{2i-2},c^2/t^{4n-2i}v^2\big)}
 {\Gampq\big(c^4/t^{2i},t^{2i-1}\big)} \prod_{1\le i\le 2n} \Gampq\big(t^{-i}c^2\big).
\end{gather*}
\end{prop}

Similarly, when $n=1$, the interpolation kernel in the integrand simplifies so that we obtain an elliptic beta integral.

\begin{prop}\label{prop:litt_n2}
We have
\begin{gather*}
\cL^{(2)}_c(x_1,x_2;t;p,q)
= \frac{\Gampq\big(c^2/t\big)}
 {\Gampq\big(c^4/t^2,c^2,t\big)}
\Gampq\big(c^2x_1^{\pm 1}x_2^{\pm 1}/t\big)
= \frac{\Gampq\big(c^2/t\big)}{\Gampq\big(c^2\big)}
\cK^{(1)}_{c^2/t}(x_1;x_2;t;p,q).
\end{gather*}
\end{prop}

When $t=q$ (or $t=p$), the interpolation kernel is essentially a determinant, so that (fol\-lo\-wing~\cite{deBruijnNG:1955}) the Littlewood kernel becomes a pfaffian.

\begin{prop}
For $t=q$, we have
\begin{gather*}
\prod_{1\le i<j\le 2n} x_i^{-1}\theta_p(x_i x_j,x_i/x_j) \cL^{(2n)}_c(\vec{x};q;p,q)\\
\qquad {}= c^{-2n(n-1)} q^{n^2-n} \pf_{1\le i,j\le 2n} \big(x_i^{-1}\theta_p(x_i x_j,x_i/x_j)\cL^{(2)}_c(x_i,x_j;q;p,q)\big).
\end{gather*}
\end{prop}

Another case with a reasonably nice expression is when $c=\sqrt{pq/t}$, so the interpolation kernel in the integrand can be expressed as a product.

\begin{prop}\label{prop:litt_{pqt}}
We have
\begin{gather*}
\cL^{(2n)}_{\sqrt{pq/t}}(\vec{x};t;p,q) =\cK^{(n)}_{pq/t^2}\big(x_1,\dots,x_n;x_{n+1},\dots,x_{2n};t^2;p,q\big).
\end{gather*}
\end{prop}

The name ``Littlewood kernel'' comes from the following formal expansion.

\begin{prop}\label{prop:litt_sum}
If $\max(|\ord(t_0)|,\max_i|\ord(z_i)|)<\ord(c)\le 1/4$, then we have the formal expansion
\begin{gather*}
\cL^{(2n)}_c(\vec{z};t;p,q) =
\frac{\prod\limits_{1\le i\le 2n}
 \Gampq\big(\big(c^2/t^{2n-1}t_0\big) z_i^{\pm 1},\big(c^2t_0/t\big) z_i^{\pm 1}\big)}
{\prod\limits_{1\le i\le n}
 \Gampq\big(c^2 t^{2-2i},c^2 t^{1-2i},t^{2i-3}c^2t_0^2,t^{2i+1-4n}c^2/t_0^2,t^{-2i}c^4,t^{2i-1}\big)}\\
 \hphantom{cL^{(2n)}_c(\vec{z};t;p,q) = }{}
\times \sum_{\mu}
\Delta_{\mu}\big(t^{4n-3}t_0^2/c^2|t^{2n},pq t^{2n}/c^4;q,t^2;p\big)
R^{*(2n)}_{\mu^2}\big(\vec{z};t_0,c^2/t^{2n-1}t_0;q,t;p\big).
 \end{gather*}
\end{prop}

\begin{proof} It suffices to prove this in the case $\ord(t_0)=\ord(z_i)=0$, since both sides have well-behaved Puiseux series expansions. We can then specialize so that $\vec{z}$ is a partition based at $t_0$, at which point the claim follows from \cite[Theorem~4.7]{littlewood}.
\end{proof}

\begin{rems}
Note that if we remove the factor $\Gampq\big(t^{-2i}c^4\big)$, then the right-hand
side converges formally whenever
\begin{gather*}
\max(|\ord(t_0)|,\max_i|\ord(z_i)|)<\min(\ord(c),1/2-\ord(c)).
\end{gather*}
If $|\ord(t_0)|=|\ord(\vec{z})|=0$, it follows from the branching rule
below that the sum converges to the Littlewood kernel for the full
range $0<\ord(c)<1/2$.
\end{rems}

\begin{rems}
As in the case of the deformed Cauchy identity representation of the interpolation function, this becomes a sum of Macdonald polynomials in a~suitable limit. The na\"\i{}ve version of the limit is
\begin{gather*}
\lim_{p\to 0} \prod_{1\le i\le n} \frac{\Gampq\big(t^{2i-1}\big)}
 {\Gampq\big(qt^{2i}/c^4\big)}
\cL^{(2n)}_{p^{1/4}c}\big(p^{-1/4}\vec{x};t;p,q\big) \\
\qquad{} \text{``=''}
\sum_{\mu} \big(c^2/t\big)^{|\mu|)}
\frac{C^-_\mu\big(t;q,t^2\big)C^0_\mu\big(qt^{2n}/c^4;q,t^2\big)}
 {C^-_\mu\big(q;q,t^2\big)C^0_\mu\big(t^{2n-1};q,t^2\big)} P_{\mu^2}(\vec{x};q,t),
\end{gather*}
which can be made rigorous in the same way as the Macdonald limit of the interpolation kernel. Note that when $c=(qt)^{1/4}$ here, the coefficient is essentially the coefficient in the usual Littlewood identity for Macdonald polynomials.
\end{rems}

When $c=(pqt)^{1/4}$, this becomes the usual elliptic Littlewood sum, and we obtain the following evaluation.

\begin{thm}
When $c=(pqt)^{1/4}$, the Littlewood kernel has the product expression
\begin{gather*}
\cL^{(2n)}_{(pqt)^{1/4}}(\vec{x};t;p,q) =\Gampq\big((pq/t)^{1/2}\big)^{2n}
\prod_{1\le i<j\le 2n} \Gampq\big((pq/t)^{1/2}x_i^{\pm 1}x_j^{\pm 1}\big).
\end{gather*}
\end{thm}

\begin{rem}
This can also be proved using integral manipulations alone: if one expands
the interpolation kernel using the degenerate branching rule (Proposition~\ref{prop:branch_k} above), swaps the two resulting integrals, then applies
the degenerate braid relation (Proposition~\ref{prop:braid_k}), one obtains
an $(n-1)$-dimensional integral involving a $(2n-1)$-dimensional instance
of the interpolation kernel. If we then perform the same steps again, we
obtain the $(n-1)$-dimensional instance of the theorem. Working backwards,
this gives an inductive proof of this evaluation.
\end{rem}

Another consequence of the formal deformed Littlewood sum expression is
that the Littlewood kernel satisfies a branching rule.

\begin{cor}
The Littlewood kernel satisfies the branching rule
\begin{gather*}
 \Gampq\big(c^4/t^2,c^2\big) \cL^{(2n)}_c(\vec{x},v;t;p,q)=
\frac{\Gampq\big(c^2/t\big) \prod\limits_{1\le i\le 2n-1} \Gampq\big(c^2v^{\pm 1} x_i^{\pm 1}/t\big)}
{\Gampq(t)^{2n-1}\prod\limits_{1\le i<j\le 2n-1} \Gampq\big(t x_i^{\pm 1}x_j^{\pm 1}\big)}\\
\qquad{}\times \int \cL^{(2n-2)}_{t^{-1/2}c}(\vec{z};t;p,q)
\Delta^{(2n-2)}_D\big(\vec{z};pq t^{1/2}v^{\pm 1}/c^2,t^{1/2}\vec{x}\,^{\pm 1};p,q\big).
\end{gather*}
\end{cor}

\begin{proof}
Expand the left-hand side via the formal sum for $t_0=v$, and note that
this gives an expansion in $(2n-1)$-variable interpolation functions indexed
by partitions with $\le 2n-2$ parts. As a result, we can expand those
interpolation functions using the integral representation; simplifying
gives the desired result.
\end{proof}

Since the Littlewood kernel is defined using the interpolation kernel, we
can use the braid relation to obtain a transformation of sorts.

\begin{thm}\label{thm:quad_litt}
The Littlewood kernel satisfies the integral identity
\begin{gather*}
\int \cK^{(2n)}_{c/d}(\vec{z};\vec{x};t;p,q)\cL^{(2n)}_d(\vec{z};t;p,q) \Delta^{(2n)}_S\big(\vec{z};\sqrt{pq}w^{\pm 1}/c,\sqrt{pqt}v^{\pm 1}/d;t;p,q\big) \\
\qquad {}= \prod_{1\le i\le 2n} \Gampq\big(\sqrt{pq}w^{\pm 1}x_i^{\pm 1}/d\big) \\
\qquad\quad{} \times \int
\cK^{(2n)}_c\big(t^{\pm 1/2}\vec{z};\vec{x};t;p,q\big) \Delta^{(n)}_S\big(\vec{z};\sqrt{pq}t^{\pm 1/2}dw^{\pm 1}/c,\sqrt{pq}t v^{\pm 1}/d^2;t^2;p,q\big).
\end{gather*}
\end{thm}

\begin{proof}
Expand $\cL^{(2n)}_d$ using the definition, exchange the two integrals
(allowable since we can choose the parameters so that all singularities are
inside the unit circle), then use the braid relation to simplify the inner
integral.
\end{proof}

When $\vec{x}$ is specialized to a partition pair, we obtain the following.

\begin{cor}\label{cor:quad_litt}
For otherwise generic parameters satisfying $t^{2n}t_0t_1t_2u_0=pqd^2$,
\begin{gather*}
 \int \cR^{*(2n)}_{\blambda} (\vec{z};t_0/d,u_0/d;t;p,q)
\cL^{(2n)}_{t^{1/2}d}(\vec{z};t;p,q)\\
\qquad\quad{}\times
\Delta^{(2n)}_S\big(\vec{z};t_0/d,t_1/d,t_2/d,u_0/d,\sqrt{pq}v^{\pm
 1}/d;t;p,q\big) \\
\qquad {}= \frac{\Delta^0_{\blambda}\big(t^{2n-1}t_0/u_0|t^{2n-1}t_0t_1/d^2;t;p,q\big)}
 {\Delta^0_{\blambda}\big(t^{2n-1}t_0/u_0|t^{2n}t_0t_1;t;p,q\big)}
\prod_{\substack{0\le i<2n\\0\le r<s<3}}
 \frac{\Gampq\big(t^i t_rt_s/d^2\big)}
 {\Gampq\big(t^{i+1}t_rt_s\big)}\\
\qquad\quad{}\times \int
 \cR^{*(2n)}_{\blambda}\big(t^{\pm 1/2}\vec{z};t^{1/2}t_0,t^{1/2}u_0;t;p,q\big)\\
\qquad\quad{}\times \Delta^{(n)}_S\big(\vec{z}; t_0,tt_0,t_1,tt_1,t_2,tt_2,u_0,tu_0,\sqrt{pq} v^{\pm 1}/d^2;t^2;p,q\big).
 \end{gather*}
\end{cor}

If we take $d=\sqrt{-1}$ in this identity, we find that the right-hand side
agrees (even including the prefactors) with the right-hand side of
Conjecture Q1 of \cite{littlewood}; similarly, the case $d=p^{-1/4}$
recovers (up to shifting $v$) the right-hand side of Conjecture Q2 of
\cite{littlewood}. We may thus view those conjectures as
claims about certain degenerations of the Littlewood kernel (specifically
for $c\in \big\{\sqrt{-t},(t^2/p)^{1/4},(t^2/q)^{1/4}\big\}$). To be precise,
those conjectures do not give {\em formulas} for the Littlewood kernel, but
rather describe how to integrate certain test functions against the
Littlewood kernel.

Since those three degenerate examples all (conjecturally in
\cite{littlewood}, but see below) give rise to vanishing identities, this
suggests that the same should apply to an arbitrary instance of the
Littlewood kernel.

\begin{thm}\label{thm:van_litt}
For generic parameters satisfying $t^{2n}t_0t_1u_0=\sqrt{pqt}$,
the integral
\begin{gather*}
\frac{1}{Z}\int
\tcR^{(2n)}_\blambda (\vec{z};t_0/d{:}dt_0,t_1/d,dt_1;u_0/d,tdu_0;t;p,q ) \cL^{(2n)}_{t^{1/2}d}(\vec{z};t;p,q) \\
\qquad{}\times \Delta^{(2n)}_S\big(\vec{z};t_0/d,t_1/d,u_0/d,\sqrt{pq/t}/d;t;p,q\big)
\end{gather*}
vanishes unless $\blambda$ has the form $\bmu^2$, when the integral is
\begin{gather*}
\frac{\Delta_{\bmu} \big(1/t^2u_0^2|t^{2n},t^{2n-1}t_0^2,1/t^{2n}t_0u_0,1/t^{2n-1}t_0u_0;q,t^2;p\big)}
 {\Delta_{\bmu^2}\big( 1/tu_0^2|t^{2n},t^{2n-1}t_0^2,1/t^{2n}t_0u_0,1/t^{2n-1}t_0u_0;q,t;p\big)}.
\end{gather*}
Here, $Z$ is a normalization constant explicitly given by
\begin{gather*}
Z= \prod_{1\le i\le n}
 \Gampq\big(t^{2i},t^{2i-1}t_0^2,t^{2i-1}t_1^2,t^{2i-1}u_0^2\big) \\
\hphantom{Z=}{}\times \prod_{1\le i\le 2n}
 \Gampq\big(t^{i-1}t_0t_1/d^2,t^{i-1}t_0u_0/d^2,t^{i-1}t_1u_0/d^2,t^{i-1}t_0t_1,t^{i-1}t_0u_0,t^{i-1}t_1u_0\big).
\end{gather*}
\end{thm}

\begin{proof}
 Set $v=\sqrt{t}$, $t_2=\sqrt{pq/t}d^2$ in Corollary \ref{cor:quad_litt},
 and observe that the right-hand side can be expressed as a sum via
 \cite[Corollary~4.8]{littlewood}. Inverting the binomial coefficient in this
 sum turns the remaining interpolation function into a biorthogonal
 function, giving the above identity.
\end{proof}

Although the Littlewood kernel is ill-defined for $c=1$, since the same
applies to the interpolation kernel, na\"{i}ve manipulations suggest that
for suitable test functions $f$,
\begin{gather*}
\lim_{c\to 1}
\int
f(\vec{z})
\cL^{(2n)}_c(\vec{z};t;p,q)
\Delta^{(2n)}_S(\vec{z};t;p,q)
=
\int
f\big(t^{\pm 1/2}\vec{z}\big)
\Delta^{(n)}\big(\vec{z};t^2;p,q\big);
\end{gather*}
the point is that the interpolation kernel for $c=1$ corresponds to the identity as an integral operator. The Littlewood kernel similarly has issues for $c=t^{1/2}$, but again we can introduce a~test function to obtain (essentially via the same limit as Proposition~\ref{prop:braid_k})
\begin{gather*}
\lim_{c\to t^{1/2}} \int f(\vec{z}) \cL^{(2n)}_c(\vec{z};t;p,q)\Delta^{(2n)}_S(\vec{z};t;p,q)\\
\qquad {}= \iint
f(\vec{x},\vec{y})
\prod_{1\le i,j\le n} \Gampq\big(t x_i^{\pm 1} y_j^{\pm 1}\big) \Delta^{(n)}_D(\vec{x};p,q)\Delta^{(n)}_D(\vec{y};p,q).
\end{gather*}
Although this is not well-defined in general, we can check that the
required manipulations are valid when $f$ is the interpolation kernel (or,
more precisely, a suitable product of the interpolation kernel and gamma functions). This gives a new explicit vanishing integral following the above argument.

\begin{cor}
For generic parameters satisfying $t^{2n}t_0t_1u_0=\sqrt{pqt}$, the integral
\begin{gather*}
\frac{1}{Z} \iint \tcR^{*(2n)}_\blambda(\vec{x},\vec{y};t_0{:}t_0,t_1,t_1;u_0,tu_0;t;p,q)
\prod_{1\le i,j\le n} \Gampq(t y_j^{\pm 1} x_i^{\pm 1}) \\
\qquad{}\times \Delta^{(n)}_D\big(\vec{x};t_0,t_1,u_0,\sqrt{pq/t};p,q\big) \Delta^{(n)}_D\big(\vec{y};t_0,t_1,u_0,\sqrt{pq/t};p,q\big)
\end{gather*}
vanishes unless $\blambda$ has the form $\bmu^2$, when the integral is
\begin{gather*}
\frac{\Delta_{\bmu} \big(1/t^2u_0^2|t^{2n},t^{2n-1}t_0^2,1/t^{2n}t_0u_0,1/t^{2n-1}t_0u_0;q,t^2;p\big)}
 {\Delta_{\bmu^2}\big( 1/tu_0^2|t^{2n},t^{2n-1}t_0^2,1/t^{2n}t_0u_0,1/t^{2n-1}t_0u_0;q,t;p\big)}.
\end{gather*}
The normalization constant $Z$ is given by
\begin{gather*}
Z = \prod_{1\le i\le n} \Gampq\big(t^{2i},t^{2i-1}t_0^2,t^{2i-1}t_1^2,t^{2i-1}u_0^2\big)
\prod_{1\le i\le 2n} \Gampq\big(t^{i-1}t_0t_1,t^{i-1}t_0u_0,t^{i-1}t_1u_0\big)^2.
\end{gather*}
\end{cor}

When $t^{2n}t_0u_0=1/d^2$ in the vanishing result, the biorthogonal function becomes an interpolation function. It turns out that there is a
more general vanishing result for interpolation functions.

\begin{thm}\label{thm:van_litt_interp}
For $t^{2n}t_0u_0=1$, the integral
\begin{gather*}
\int
\cR^{*(2n)}_{\blambda}(\vec{z};t_0/d,u_0/d;t;p,q)
\cL^{(2n)}_{t^{1/2}d}(\vec{z};t;p,q)
\Delta^{(2n)}_S\big(\vec{z};t_0/d,u_0/d,\sqrt{pq}v^{\pm 1}/d;t;p,q\big)
\end{gather*}
vanishes unless $\blambda$ has the form $\bmu^2$, when it equals
\begin{gather*}
Z
\frac{\Delta_{\bmu} \big(1/t^2u_0^2|t^{2n},pqt^{2n-2},\sqrt{pq}v^{\pm 1}d^2/u_0;q,t^2;p\big)}
 {\Delta_{\bmu^2}\big( 1/tu_0^2|t^{2n},pqt^{2n-1},\sqrt{pq}v^{\pm 1}d^2/u_0;q,t;p\big)},
\end{gather*}
with
\begin{gather*}
Z=
\prod_{1\le i\le 2n} \Gampq\big(t^{-i}/d^2\big)\\
\hphantom{Z=}{}\times \prod_{1\le i\le n}
 \Gampq\big(t^{2i},t^{1-2i},t^{2i-1}t_0^2,t^{2i-1}u_0^2,t^{2i-1}\sqrt{pq}t_0v^{\pm 1}/d^2,t^{2i-1}\sqrt{pq}u_0v^{\pm 1}/d^2\big).
\end{gather*}
\end{thm}

\begin{proof}
If we attempt to substitute the above parameters into Corollary~\ref{cor:quad_litt}, we find that the integral on the right-hand side
becomes singular (two parameters multiply to $(t^2)^{1-n}$). In
particular, the right-hand side becomes a finite sum in this limit, and in
fact at most one term can be nonzero (corresponding to
$\vec{z}=t^{2n-2i}(p,q)^{\bmu_i}tt_0$ with $\bmu^2=\blambda$). The desired
vanishing property follows; the specific nonzero values are then obtained
by taking the appropriate residue.
\end{proof}

\begin{rem}
 The residue calculation is rather tedious, so it may be worth noting the
 following shortcut: It is quite simple to determine the dependence of the
 right-hand side on $d$ and $v$ (as these only appear in the residue via
 univariate factors of the integrand), so that one can reduce to
 Theorem \ref{thm:van_litt} (taking $d=1$ or $v=\big(d^2\sqrt{pq}/t_0\big)^{\pm 1}$).
\end{rem}

In addition to vanishing results, another nice special case of Theorem~\ref{thm:quad_litt} involves taking $v=t^{1/2}$, $w=cd^2/t$, so that the
integral on the right-hand side becomes an instance of the definition of
the Littlewood kernel. We thus find the following.

\begin{thm}\label{thm:litt_braid}
The Littlewood kernel satisfies the identity
\begin{gather*}
\cL^{(2n)}_c(\vec{x};t;p,q) =
\frac{1}{\prod\limits_{1\le i\le 2n}\Gampq\big(\sqrt{pq}tx_i^{\pm 1}/cd^2,\sqrt{pq}x_i^{\pm 1}/c\big)}\\
\hphantom{\cL^{(2n)}_c(\vec{x};t;p,q) = }{}\times \int
\cK^{(2n)}_{c/d}(\vec{z};\vec{x};t;p,q)
\cL^{(2n)}_d(\vec{z};t;p,q)
\Delta^{(2n)}_S\big(\vec{z};\sqrt{pq}t/c^2d,\sqrt{pq}/d;t;p,q\big).
 \end{gather*}
\end{thm}

\begin{rem}
 Note that when $c=q^{-1/2}d$, the integral on the right-hand side should
 na\"\i vely become a difference operator; of course the corresponding
 identity holds, and by the same proof.
\end{rem}

Again a ``Bailey lemma''-like manipulation gives a transformation.

\begin{cor}\label{cor:weird_quad}
The expression
\begin{gather*}
\prod_{1\le i\le 2n} \Gampq\big(\sqrt{pq}vdx_i^{\pm 1}/e\big)\\
\qquad{}\times \int \cK^{(2n)}_{ce}(\vec{z};\vec{x};t;p,q)\cL^{(2n)}_d(\vec{z};t;p,q) \Delta^{(2n)}_S\big(\vec{z};\sqrt{pq}v^{\pm 1}/cd,\sqrt{pq}/d,\sqrt{pq}t/de^2;t;p,q\big)
\end{gather*}
is invariant under swapping $d$ and $e$.
\end{cor}

We obtain a different transformation by specializing the parameters in Theorem~\ref{thm:quad_litt} so that the right-hand side transforms under
Theorem~\ref{thm:L1_kern}.

\begin{cor}
The expression
\begin{gather*}
\frac{1}{\prod\limits_{1\le i\le 2n} \Gampq\big(\sqrt{pq}t(t^{1/2}v)^{\pm 1}z_i^{\pm 1}/d^2\big)}\\
\quad{}\times \int \cK^{(2n)}_c(\vec{z};\vec{y};t;p,q) \cL^{(2n)}_d(\vec{y};t;p,q) \Delta^{(2n)}_S\big(\vec{y};\sqrt{pq/t}(v/c)^{\pm 1},
 \sqrt{pq}\big(t^{1/2}v\big)^{\pm 1}t/cd^2 ;t;p,q\big)
\end{gather*}
is invariant under $v\mapsto 1/v$.
\end{cor}

In the limit $c\to q^{-1/2}$, this becomes a difference equation.

\begin{cor}\label{cor:litt_diff}
The expression
\begin{gather*}
\frac{1}{\prod\limits_{1\le i\le 2n} \theta_p\big(\sqrt{pqt}v y_i^{\pm 1}\big)}
D^{(2n)}_q\big(\sqrt{p/t}\big(q^{1/2}v\big)^{\pm 1},\sqrt{pq}t\big(t^{1/2}v\big)^{\pm 1}/d^2;t;p\big)_{\vec{y}}
\cL^{(2n)}_d(\vec{y};t;p,q)
\end{gather*}
is invariant under $v\mapsto 1/v$.
\end{cor}

It turns out that in many cases, this 1-parameter family of difference
equations suffices to uniquely determine the Littlewood kernel; see below,
where we use it to evaluate the Littlewood kernel in the case
$d=(pt)^{1/4}$.

As usual for branching rules, the right-hand side of the branching rule for
$\cL^{(2n)}_c$ appears to have less symmetry than the left-hand side. This,
of course, corresponds to a transformation, which generalizes to the
following.

\begin{thm}\label{thm:spec_int_litt}
If $t_0t_1t_2t_3=\big(pq/c^2\big)^2$, then
\begin{gather*}
\frac{1}{
 \prod\limits_{\substack{1\le i\le 2n\\0\le r\le 3}} \Gampq\big(t^{1/2}t_r x_i^{\pm 1}\big)}
\int \cK^{(2n)}_{t^{1/2}}(\vec{y};\vec{x};t;p,q) \cL^{(2n)}_c(\vec{y};t;p,q) \Delta^{(2n)}_S(\vec{y};t_0,t_1,t_2,t_3;t;p,q)
\end{gather*}
is invariant under $t_r\mapsto pq/c^2t_r$.
\end{thm}

\begin{proof}
Expand the Littlewood kernel using the definition, choosing $v$ so that the integral over~$y$ still has only four parameters. Applying the degenerate version of commutation to this integral gives an integral in which the desired symmetry is manifest.
\end{proof}

When two parameters multiply to $pq$, the integral again is independent of the remaining parameter, and once more gives rise to the Littlewood kernel.

\begin{cor}\label{cor:spec_spec_int_litt}
We have
\begin{gather*}
\int \cK^{(2n)}_{t^{1/2}}(\vec{z};\vec{x};t;p,q) \cL^{(2n)}_c(\vec{z};t;p,q)
\Delta^{(2n)}_S\big(\vec{z};\sqrt{pq}v^{\pm 1}/c^2;t;p,q\big) \\
\qquad {} =
\prod_{1\le i\le 2n} \Gampq\big(\sqrt{pqt}v^{\pm 1}x_i^{\pm 1}/c^2,\sqrt{pq/t}v^{\pm 1} x_i^{\pm 1}\big) \cL^{(2n)}_{t^{1/2}c}(\vec{x};t;p,q).
\end{gather*}
\end{cor}

Taking $c=(pqt)^{1/4}$ gives another semi-explicit special case of the Littlewood kernel:
\begin{gather*}
\cL^{(2n)}_{(pqt^3)^{1/4}}(\vec{x};t;p,q) =
\frac{\cK^{(2n)}_{(pqt)^{1/2}}\big((pq/t)^{1/4}\vec{x};(pq/t)^{-1/4}\vec{x};\sqrt{pq/t};p,q\big)}{
\Gampq(t)^{2n}
\prod\limits_{1\le i<j\le 2n} \Gampq\big(t x_i^{\pm 1}x_j^{\pm 1}\big)}.
\end{gather*}
The corollary can also be combined with the ``distributional'' formula for $\cL^{(2n)}_{t^{1/2}}$ to give
\begin{gather}
\cL^{(2n)}_t(\vec{x};t;p,q) = \frac{\cK^{(n)}_t(x_1,\dots,x_n;x_{n+1},\dots,x_{2n};pq/t;p,q)^2}
 {\Gampq(t)^{2n}\prod\limits_{1\le i<j\le 2n} \Gampq\big(t x_i^{\pm 1}x_j^{\pm 1}\big)}.\label{eq:litt_t}
\end{gather}
Two more such formulas will follow from the ``distributional'' expressions of $\cL^{(2n)}_{\sqrt{-t}}$ and $\cL^{(2n)}_{q^{-1/4}t^{1/2}}$; we state
them here, but note that they are properly viewed as corollaries of Theorems~\ref{thm:Q2} and~\ref{thm:Q1} below
\begin{gather*}
\cL^{(2n)}_{q^{-1/4}t}(\vec{x};t;p,q)=
\frac{\cK^{(n)}_{q^{-1/2}t}\big(x_1,\dots,x_n;x_{n+1},\dots,x_{2n};pq/t;p,q^{1/2}\big)}
{\Gampq(t)^{2n}
\prod\limits_{1\le i<j\le 2n} \Gampq\big(t x_i^{\pm 1}x_j^{\pm 1}\big)}, \\
\cL^{(2n)}_{t\sqrt{-1}}(x_1,\dots,x_{2n};t;p,q) =
\frac{\cK^{(n)}_{t^2}\big(x_1^2,\dots,x_n^2;x_{n+1}^2,\dots,x_{2n}^2;p^2q^2/t^2;p^2,q^2\big)}
{\Gampq(t)^{2n} \prod\limits_{1\le i<j\le 2n} \Gampq\big(t x_i^{\pm 1}x_j^{\pm 1}\big)}.
\end{gather*}

Proposition \ref{prop:litt_n2} shows that there is a close connection
between $\cL^{(2)}_c$ and $\cK^{(1)}_{c^2/t}$, and several of the above
examples involve similar connections for $\cL^{(2n)}_c$. This suggests in
general that we should expect $\cL^{(2n)}_c$ to have a nice special case
whenever $\cK^{(n)}_{c^2/t}$ falls in a nice special case. E.g., the
product case $\cL^{(2n)}_{(pqt)^{1/4}}$ corresponds in this way to the
product case $\cK^{(n)}_{\sqrt{pq/t}}$. (This correspondence includes the
degenerate cases considered in Section~\ref{section8} below, which
correspond to $\cK^{(n)}_{-1}$, $\cK^{(n)}_{q^{-1/2}}$,
$\cK^{(n)}_{p^{-1/2}}$ via this heuristic.) Since
$\cK^{(n)}_{(p/t)^{1/2}}$ had an unexpected determinantal expression, this
suggests that we should investigate $\cL^{(2n)}_{(pt)^{1/4}}$. Although
the argument for the interpolation kernel case does not carry over, we can
still use the cases $t=q$ and $2n=2$ as a guide. In particular, if we
guess that dividing by a suitable product makes $\cL^{(2n)}_{(pt)^{1/4}}$
independent of $q$, then there is a~natural possibility for that product. We are thus led to guess that
\begin{gather*}
\Gampq((pt)^{1/2})^{2n}
\prod_{1\le i<j\le 2n} \Gampq\big((pt)^{1/2}x_i^{\pm 1}x_j^{\pm 1}\big)
\cL^{(2n)}_{(pt)^{1/4}}(\vec{x};t;p,q)
\end{gather*}
is independent of $q$; this would give us a pfaffian expression for
$\cL^{(2n)}_{(pt)^{1/4}}(\vec{x};t;p,q)$.

None of the methods we have used above (or will use in Section~\ref{section8}) appears to be applicable to derive such an expression. It
turns out, however, that given such a guess, there {\em is} a method we can
use to prove it. The key observation is that Corollary \ref{cor:litt_diff}
gives a family of difference equations which, in a suitable limit, has a
unique formal solution. Indeed, taking $\ord(v)=0$, $\ord(d)=1/4$ in
Corollary~\ref{cor:litt_diff} gives a difference equation with formal
series coefficients that in the limit $p\to 0$ becomes the equation of
Corollary~\ref{cor:diff_uniq} for $u=\lim\limits_{p\to 0}\sqrt{pq}t/d^2$. Thus as
long as $\lim\limits_{p\to 0}d^4/p$ is not of the form $qt^{n+2-i}$, the equation
has a unique (up to scalars) formal solution. (This follows as in the
Remark following Proposition~\ref{prop:diff_braid}.)

Since in our case $\lim_{p\to 0}d^4/p=t$, there is no difficulty, so it will suffice to prove the equation holds (and verify that we have the correct scalar multiple). It will be convenient to replace~$q$ by~$q^2$ and $t$ by~$p/t^2$, so that the equation
becomes
\begin{gather*}
\begin{split}&
(1-R(v)) \frac{1}{\prod\limits_{1\le i\le 2n} \theta_p\big((t/qv) x_i^{\pm 1}\big)}\\
& \qquad{}\times D^{(2n)}_{q^2}\big(tqv,t/qv,q/v,pqv/t^2;p/t^2;p\big)_{\vec{x}} \cL^{(2n)}_{\sqrt{p/t}}\big(\vec{x};p/t^2;p,q^2\big)=0.
\end{split}
\end{gather*}
After substituting in the claimed value for the Littlewood kernel, we find that we need to show that
\begin{gather}
(1-R(v)) \prod_{1\le i\le 2n} (1+R(x_i))
\frac{\theta_p\big(tqv x_i,(q/v) x_i,t^2/qvx_i\big)}
 {x_i^{2n-1}\theta_p\big(t/qvx_i,x_i^2\big)} \prod_{1\le i<j\le 2n} \frac{\theta_p\big(tx_ix_j,x_ix_j/t^2\big)}
 {\theta_p\big(x_ix_j,t/q^2x_ix_j\big)} \nonumber\\
\qquad{}\times \bigg[
\prod_{1\le i<j\le 2n} \frac{1}{F(qx_i,qx_j;t)}
\pf_{1\le i,j\le 2n} F(qx_i,qx_j;t) \bigg]=0,\label{eq:fund_ident_IK}
\end{gather}
where
\begin{gather*}
F(x,y;t):=
\frac{x^{-1}\theta_p(xy,x/y)}
 {\theta_p(t x^{\pm 1}y^{\pm 1})},
\end{gather*}
and $R(x)$ denotes the operator $x\mapsto 1/x$. The quantity in
brackets is set apart for the following reason.

\begin{lem}
The function
\begin{gather*}
\prod_{1\le i<j\le 2n} \frac{1}{F(x_i,x_j;t)}
\pf_{1\le i,j\le 2n} F(x_i,x_j;t)
\end{gather*}
is holomorphic.
\end{lem}

\begin{proof}
The only poles of the pfaffian come from poles of $F(x_i,x_j;t)$, and thus
are cancelled by the prefactor. The only poles of the prefactor are at
zeros of $x_i^{-1}\theta_p(x_ix_j,x_i/x_j)$, but these are cancelled by the
pfaffian (since the pfaffian is antisymmetric and quasiperiodic).
\end{proof}

It will be helpful to first consider a somewhat simpler version of this
identity.

\begin{lem}\label{lem:litt_IK_small}
For any parameters $q$, $t$, $\vec{x}$, we have the identity
\begin{gather*}
\prod_{1\le i\le 2n-1} (1-R(x_i))
\prod_{1\le i<j\le 2n-1} \frac{\theta_p\big(t^2/x_ix_j,tx_ix_j\big)}
 {\theta_p\big(x_ix_j,t/q^2x_ix_j\big)} \\
\qquad{}\times \bigg[ \prod_{1\le i<j\le 2n} \frac{1}{F(qx_i,qx_j;t)}\pf_{1\le i,j\le 2n} F(qx_i,qx_j;t)\bigg] =0.
\end{gather*}
\end{lem}

\begin{proof} Let $G(q)$ denote the given sum as a function of $q$, after first replacing $x_{2n}\mapsto x_{2n}/q$. We then find (by checking this for every term) that $G(pq)=\big(p/t^2\big)^{(2n-1)(n-1)}G(q)$. As a result, in order to show that $G(q)=0$, it suffices to show that it is
holomorphic. Since the term in brackets is holomorphic, the only poles come from the factors $\theta_p\big(t/q^2x_ix_j\big)$ (and their images under the symmetry). It thus suffices to show that the residue of the sum along any such divisor is 0. Taking the residue in $x_{2n-1}$ along the divisor $q^2x_{2n-2}x_{2n-1}=t$ gives a smaller instance of the identity, and thus the identity follows by induction.
\end{proof}

\begin{lem}\label{lem:litt_IK_big} Equation \eqref{eq:fund_ident_IK} holds.
\end{lem}

\begin{proof}If $G(q)$ denotes the left-hand side of \eqref{eq:fund_ident_IK} as a function of $q$, we note that $G(pq)=p^{2n^2+n}t^{-4n^2+2n}G(q)$,
so that again it suffices to prove that $G(q)$ is holomorphic. There are now two types of poles to consider, coming from the factors $\theta_p(t/qvx_i)$ and $\theta_p\big(t/q^2x_ix_j\big)$. The residue in $v$ along the first type of pole vanishes by Lemma~\ref{lem:litt_IK_small}, while the residue in $x_j$ along the second type of pole vanishes by induction.
\end{proof}

\begin{thm}\label{thm:funny_pfaffian}
When $c=(pt)^{1/4}$, the Littlewood kernel has the following pfaffian expression
\begin{gather*}
\cL^{(2n)}_{(pt)^{1/4}}(\vec{x};t;p,q)=
(t/p)^{n(n-1)/2}\!
\left(\frac{\theta_p(t)}
 {\Gampq((pt)^{1/2})^2 \theta_p((pt)^{1/2})}\right)^n
\!\prod_{1\le i<j\le 2n}\!
\frac{1}{\Gampq\big((pt)^{1/2}x_i^{\pm 1}x_j^{\pm 1}\big)} \\
\hphantom{\cL^{(2n)}_{(pt)^{1/4}}(\vec{x};t;p,q)= }{} \times
\prod_{1\le i<j\le 2n}\frac{1}{x_i^{-1}\theta_p(x_i x_j,x_i/x_j)} \pf_{1\le i,j\le 2n}
 \frac{x_i^{-1}\theta_p(x_i x_j,x_i/x_j)}
 {\theta_p\big((pt)^{1/2} x_i^{\pm 1} x_j^{\pm 1}\big)}.
\end{gather*}
\end{thm}

\begin{proof}
 As we have already noted, that the right-hand side satisfies the
 requisite difference equation is simply (up to reparametrization)
 equation~\eqref{eq:fund_ident_IK}, and thus this fact holds by Lemma~\ref{lem:litt_IK_big}. It follows therefore that the above expression
 holds up to a factor independent of $\vec{x}$. Since this is equivalent
 to the expression
\begin{gather*}
\cL^{(2n)}_{(pt)^{1/4}}(\vec{x};t;p,q) = \Gampq\big((pt)^{1/2}\big)^{-2n}
\prod_{1\le i<j\le 2n} \Gampq\big((pt)^{1/2}x_i^{\pm 1}x_j^{\pm 1}\big)^{-1}
\cL^{(2n)}_{(pt)^{1/4}}(\vec{x};t;p,t),
\end{gather*}
it is straightforward to verify that this takes the correct value when $\vec{x}=\dots,t^{2n-i}v,\dots$, giving the desired result.
\end{proof}

As in the interpolation kernel case, the special case $t=p^{1/3}$ is
particularly nice, since we then have an alternate expression, equation~\eqref{eq:litt_t}. Cancelling common factors gives the following theta
function identity:
\begin{gather*}
\prod_{1\le i<j\le 2n}
 \frac{\theta_p\big(p^{1/3}x_i^{\pm 1}x_j^{\pm 1}\big)}
 {x_i^{-1}\theta_p(x_i x_j,x_i/x_j)} \pf_{1\le i,j\le 2n} \frac{x_i^{-1}\theta_p(x_i x_j,x_i/x_j)}
 {\theta_p\big(p^{1/3} x_i^{\pm 1} x_j^{\pm 1}\big)} \\
\qquad{}=
\biggl[
\frac{\prod\limits_{1\le i,j\le n} \theta_p\big(p^{1/3}x_i^{\pm 1} x_{n+j}^{\pm 1}\big)}
{\prod\limits_{1\le i<j\le n} x_i^{-1}x_{n+1}^{-1}\theta_p(x_i x_j,x_i/x_j,x_{n+i} x_{n+j},x_{n+i}/x_{n+j})}
\det_{1\le i,j\le n} \frac{1}{\theta_p\big(p^{1/3} x_i^{\pm 1} x_{n+j}^{\pm 1}\big)}\biggr]^2.
\end{gather*}
(Some similar factorizations appeared in \cite{ZinnJustinP:2013}, but the above appears to be new.) This is, in a~somewhat disguised way, a special
case of the more general identity
\begin{gather*}
\prod_{1\le i<j\le 2n} \frac{z_i-z_j}{y_i-y_j} \pf_{1\le i,j\le 2n} \frac{(y_i-y_j)^2}{z_i-z_j}=
\biggl[
\prod_{1\le i,j\le n} (z_i-z_{n+j})
\det_{1\le i,j\le n} \frac{y_i-y_j}{z_i-z_{n+j}}
\biggr]^2.
\end{gather*}
(This is the special case of \cite[Theorem~4.7]{OkadaS:1998} in which the two
factors agree.) To see this,
apply a substitution of the form
\begin{gather*}
y_i = C_1\frac{\theta_p\big(a x_i^{\pm 1}\big)}{\theta_p\big(b x_i^{\pm 1}\big)},\qquad
z_i = C_2\frac{\theta_{p^{1/3}}\big(c x_i^{\pm 1}\big)}{\theta_{p^{1/3}}\big(d x_i^{\pm 1}\big)},
\end{gather*}
then remove the unwanted factors from the rows and columns. We also note
here that $z_i$ is a rational function of degree~3 in $y_i$, and dimension
considerations show that the general such function appears in this way.

A particularly nice consequence of the pfaffian expression for
$\cL^{(2n)}_{(pt)^{1/4}}$ arises as the Macdonald polynomial limit from
Proposition \ref{prop:litt_sum}. Recall that to obtain such a limit, we
compare the two expressions for
\begin{gather*}
\cL^{(2n)}_{p^N t^{1/4}}\big(p^{-N+1}\vec{x};t;p^{4N},q\big),
\end{gather*}
and take the limit as $N\to\infty$, noting that the Littlewood kernel
converges formally in this limit, so the limiting identity continues to hold.

\begin{cor}
The Macdonald polynomials satisfy the following summation identity
\begin{gather*}
\sum_\mu
\frac{C^-_\mu\big(t;q,t^2\big)} {C^-_\mu\big(q;q,t^2\big)}
 \prod_{1\le i\le n} \big(1-q^{\mu_i}t^{2n-2i+1}\big) P_{\mu^2}(x_1,\dots,x_{2n};q,t) \\
\qquad {} = (1-t)^n \prod_{1\le i<j\le 2n} \frac{(tx_ix_j;q)}{(x_i-x_j)(q x_ix_j;q)}
\pf_{1\le i,j\le 2n} \frac{x_i-x_j}{(1-x_ix_j)(1-tx_ix_j)}.
\end{gather*}
\end{cor}

This is a special case of Conjecture 1 of
\cite{BeteaD/WheelerM/Zinn-JustinP:2015}. It turns out to be
straightforward to prove that identity as well.

\begin{cor}
The Macdonald polynomials satisfy the following summation identity
\begin{gather*}
\sum_\mu \frac{C^-_\mu\big(t;q,t^2\big)} {C^-_\mu\big(q;q,t^2\big)}
\prod_{1\le i\le n} \big(1-q^{\mu_i}t^{2n-2i}u\big) P_{\mu^2}(x_1,\dots,x_{2n};q,t) \\
\qquad {}=
\prod_{1\le i<j\le 2n} \frac{(tx_ix_j;q)}{(x_i-x_j)(q x_ix_j;q)}
\pf_{1\le i,j\le 2n}
 \frac{(x_i-x_j)(1-u+(u-t)x_ix_j)}{(1-x_ix_j)(1-tx_ix_j)}.
\end{gather*}
\end{cor}

\begin{proof}
Consider the ratio
\begin{gather*}
F(\vec{x};v,q,t) :=
\frac{\sum\limits_\mu
\frac{C^-_\mu\big(t;q,t^2\big)}
 {C^-_\mu\big(q;q,t^2\big)}
\prod\limits_{1\le i\le n}\frac{1-q^{\mu_i}t^{-2i}v} {1-t^{-2i}v}P_{\mu^2}(\vec{x};q,t)}
{\sum\limits_\mu\frac{C^-_\mu\big(t;q,t^2\big)} {C^-_\mu\big(q;q,t^2\big)} P_{\mu^2}(\vec{x};q,t)}
\end{gather*}
of symmetric functions. We can evaluate the denominator using the usual Littlewood identity for Macdonald polynomials, and thus find by the previous corollary that
\begin{gather*}
F\big(x_1,\dots,x_{2n};t^{2n+1},q,t\big)=\frac{(1-t)^n}{\big(t;t^2\big)_n}
\prod_{1\le i<j\le 2n} \frac{1-x_ix_j}{x_i-x_j} \pf_{1\le i,j\le 2n}
 \frac{x_i-x_j}{(1-x_ix_j)(1-tx_ix_j)}.
\end{gather*}
Since the right-hand side is independent of $q$, so is the left-hand side.
Moreover, the left-hand side remains independent of $q$ if we set some of
the variables to $0$, and thus
$F\big(x_1,\dots,x_m;t^{2n+1},q,t\big)$
is independent of $q$ for all integers $n\ge m/2$. Since this is a Zariski
dense set, it follows that $F(x_1,\dots,x_m;v,q,t)$ is independent of $q$.
The claim follows from the case $q=0$, which was established in
\cite{BeteaD/WheelerM/Zinn-JustinP:2015}.
\end{proof}

\begin{rem}
In addition, the usual argument allows us to directly evaluate the case
$q=t$ as a pfaffian, giving an alternate argument.
One can also directly show that the identity is consistent under setting
two of the variables to~0, which allows one to prove the case $u=t^{2k+1}$ of the identity from the case $u=t$, and again the result (which is an identity of polynomials in~$u$) follows.
\end{rem}

\section{More kernels}\label{section7}

Just as Conjecture L1 of \cite{littlewood} has an analogue in terms of the
interpolation kernel (Theorem \ref{thm:L1_kern} above), the same applies to
Conjecture L2.

\begin{thm}\label{thm:L2_kern}
The interpolation kernel satisfies the integral identity{\samepage1
\begin{gather*}
\begin{split} &
 \int \cK^{(n)}_c(\vec{z};\vec{y};t;p,q)
\Delta^{(n)}_S\big(\vec{z};v_0,v_1,v_2,v_3;t;p,q^2\big) \\
& {}=\prod_{\substack{1\le i\le n\\0\le r\le 3}} \Gampq\big(cv_r y_i^{\pm 1}\big)
\int \cK^{(n)}_c(\vec{z};\vec{y};t;p,q)
\Delta^{(n)}_S\big(\vec{z};pq/c^2v_0,pq/c^2v_1,pq/c^2v_2,pq/c^2v_3;t;p,q^2\big),
\end{split}
\end{gather*}
subject to the balancing condition $v_0v_1v_2v_3=(pq/c^2)^2$.}
\end{thm}

\begin{proof}
 Again, this becomes an identity of Koornwinder integrals in the limit
 $p\to 0$, $c\sim p^{1/2}$, and dividing by the common limit gives formal
 Puiseux series with rational function coefficients. But by the remark
 following \cite[Proposition~4.13]{littlewood}, the identity holds on the
 Zariski closed set $v_0v_1=pq^k/c^2$, $k\in \Z$, so holds in general.
\end{proof}

\begin{rem}
 In fact, the argument given there applies directly to the kernel, and
 shows that if the identity holds for $(c,v_0,v_1,v_2,v_3)$, then it holds
 for $(q^{-1/2}c,q^2v_0,v_1,v_2,v_3)$. The argument is essentially the
 same as that given above for Theorem~\ref{thm:L1_kern}, except that
 instead of the braid and commutation relations for~$c=t^{1/2}$, we use
 the corresponding difference equations (Propositions~\ref{prop:diff_braid} and~\ref{prop:diff_comm}). All but one of the
 applications of Fubini reduce to linearity of integration (since the
 difference operator is expressed as a finite sum). The remaining
 application of Fubini is replaced by a combination of the
 self-adjointness of $D^{(n)}_q(t;p)$ with respect to the Selberg density
 together with the fact that
\begin{gather*}
\prod_{\substack{1\le i\le n\\0\le r<2m}} \Gampqq\big(q^{-1/2}v_r z_i^{\pm 1}\big)^{-1}
D^{(n)}_q(t;p)
\prod_{\substack{1\le i\le n\\0\le r<2m}} \Gampqq(v_r z_i^{\pm 1})
\frac{\Delta^{(n)}_S\big(\vec{z};t;p,q^2\big)}
 {\Delta^{(n)}_S(\vec{z};t;p,q)}
\end{gather*}
is invariant under $v_r\mapsto pq^2/v_r$. (This substitution has the same effect on each of the $2^n$ terms as inverting all the variables, so has no effect on the sum.)
\end{rem}

As before, we have the following special case.

\begin{cor}
The integral
\begin{gather*}
\frac{1}{\prod\limits_{1\le i\le n} \Gampq\big(\big(\sqrt{p}/c\big)v^{\pm 1} x_i^{\pm 1}\big)}
\int \cK^{(n)}_c(\vec{z};\vec{x};t;p,q)
\Delta^{(n)}_S\big(\vec{z};\sqrt{p}v^{\pm 1}/c^2;t;p,q^2\big)
\end{gather*}
is independent of $v$.
\end{cor}

\begin{defn}
The {\em dual Littlewood kernel} is the meromorphic function
\begin{gather*}
\cL^{\prime(n)}_c(\vec{x};t;p,q) :=
\frac{1}{\prod\limits_{1\le i\le n} \Gampq\big(\big(\sqrt{p}/c\big)v^{\pm 1} x_i^{\pm 1}\big)}
\int \cK^{(n)}_c(\vec{z};\vec{x};t;p,q) \Delta^{(n)}_S\big(\vec{z};\sqrt{p}v^{\pm 1}/c^2;t;p,q^2\big).
\end{gather*}
\end{defn}

Again, though this appears to depend on a choice of $\sqrt{p}$, this choice
can be absorbed by negating $v$; similarly, $\cL^{\prime(n)}_c$ is invariant
under negating $c$ or $\vec{x}$. We also have the following important
symmetry.

\begin{prop}
The dual Littlewood kernel satisfies the following $t\mapsto pq/t$ symmetry
\begin{gather*}
\cL^{\prime(n)}_c(\vec{x};pq/t;p,q)=\Gampq(t)^{n} \prod_{1\le i<j\le n} \Gampq\big(t x_i^{\pm 1}x_j^{\pm 1}\big)
\cL^{\prime(n)}_c(\vec{x};t;p,q).
\end{gather*}
\end{prop}

\begin{proof}
Apply the $t\mapsto pq/t$ symmetry to the interpolation kernel in the integrand, and simplify using
\begin{gather*}
\Gampq(t)^n \prod_{1\le i<j\le n} \Gampq\big(t z_i^{\pm 1}z_j^{\pm 1}\big)
= \frac{\Delta^{(n)}_S\big(\vec{z};t;p,q^2\big)}
 {\Delta^{(n)}_S\big(\vec{z};pq/t;p,q^2\big)},
\end{gather*}
a straightforward application of \eqref{eq:gampqq}.
\end{proof}

\begin{prop}
The dual Littlewood kernel has the following specialization
\begin{gather*}
\cL^{\prime(n)}_c\big(t^{n-1}v,\dots,v;t;p,q\big)=\prod_{1\le i\le n} \frac{ \Gampq\big(qt^{1-i}c^2\big)
 \Gampqq\big(qt^{i-1}c^2v^2,qt^{i-1}c^2/t^{2n-2}v^2\big)} {\Gampqq\big(q^2t^{1-i}c^4,qt^i\big)}.
\end{gather*}
\end{prop}

\begin{rem}
In particular,
\begin{gather*}
\cL^{\prime(1)}_c(x;t;p,q)=\frac{\Gampq\big(qc^2\big)\Gampqq\big(q c^2 x^{\pm 2}\big)}{\Gampqq\big(qt,q^2c^4\big)}.
\end{gather*}
\end{rem}

When $t=q$, so that the interpolation kernel can be expressed as a~determinant, the remainder of the integrand can be expressed as a pfaffian
(the discussion in~\cite{littlewood} for the interpolation function case
carries over in a simplified form), and thus again~\cite{deBruijnNG:1955}
shows that the dual Littlewood kernel is a pfaffian. Since the entries of
the pfaffian appear no longer to have nice expressions (they are
$2$-dimensional instances of the dual Littlewood kernel), we omit the details.

As above, we have expressions in terms of the interpolation kernel when
$c=\sqrt{pq/t}$ or $c=\sqrt{t}$. We give the former, as the latter has
complicated prefactors and can in any event be obtained via the $t\mapsto
pq/t$ symmetry.

\begin{prop}
We have
\begin{gather*}
\cL^{\prime(n)}_{\sqrt{pq/t}}(\vec{x};t;p,q) = \cK^{(n)}_{pq^2/t}\big(q^{-1/2}\vec{x};q^{1/2}\vec{x};t;p,q^2\big).
\end{gather*}
\end{prop}

Another nice special case is when $c=1$.

\begin{prop}\label{prop:dual_litt_c_1}
We have
\begin{gather}
\cL^{\prime(n)}_1(\vec{x};t;p,q) = \frac{\Delta^{(n)}_S\big(\vec{x};t;p,q^2\big)} {\Delta^{(n)}_S(\vec{x};t;p,q)}.\label{eq:dual_litt_c_1}
\end{gather}
\end{prop}

\begin{proof}
We need to take the limit $c\to 1$ in the expression
\begin{gather*}
\cL^{\prime(n)}_c(\vec{x};t;p,q)=
\frac{1}{\prod\limits_{1\le i\le n} \Gampq\big(\big(\sqrt{p}/c\big)v^{\pm 1} x_i^{\pm 1}\big)}
\int\prod_{1\le i\le n} \Gampqq\big(\big(\sqrt{p}/c^2\big)v^{\pm 1}z_i^{\pm 1}\big)
\frac{\Delta^{(n)}_S\big(\vec{z};t;p,q^2\big)}
 {\Delta^{(n)}_S(\vec{z};t;p,q)} \\
\hphantom{\cL^{\prime(n)}_c(\vec{x};t;p,q)=}{} \times \cK^{(n)}_c(\vec{z};\vec{x};t;p,q)
\Delta^{(n)}_S(\vec{z};t;p,q).
\end{gather*}
Now, we na\"{\i}vely expect
$\cK^{(n)}_c(\vec{z};\vec{x};t;p,q)\Delta^{(n)}_S(\vec{z};t;p,q)$ to behave
like a delta function in this limit, from which the given expression would
follow. To make this precise, we note that although the right-hand side of~\eqref{eq:dual_litt_c_1} has a fairly complicated (though a product) limit
as $p\to 0$, it converges to $1$ as $q\to 0$, and indeed has a polynomial
Puiseux series in $q$. Moreover, the integral operator we are applying to
this formal series differs by formal factors from an instance of the usual
integral operator associated to the kernel, and thus (since the formal
factors cancel) indeed converges to the identity as we would expect.
\end{proof}

Again, the name ``dual Littlewood kernel'' comes from an expression as a
deformation of the dual Littlewood identity.

\begin{prop} For $\max(|\ord(t_0)|,\max_i|\ord(x_i)|)<\ord(c)\le 1/4$, we have the formal expansion
\begin{gather*}
\cL^{\prime(n)}_c(\vec{x};t;p,q) =
\prod_{1\le i\le n}
\frac{\Gampq\big(\big(c^2/t^{n-1}t_0\big)x_i^{\pm 1},qc^2t_0x_i^{\pm 1}\big)}{
 \Gampq\big(t^{1-i}c^2\big) \Gampqq\big(q^2t^{i-1}c^2t_0^2,t^{i+1-2n}c^2/t_0^2,q^2t^{1-i}c^4,qt^i\big)} \\
\hphantom{\cL^{\prime(n)}_c(\vec{x};t;p,q) =}{}\times \sum_{\mu}
R^{*(n)}_{2\mu}\big(\vec{x};t_0,c^2/t^{n-1}t_0;q,t;p\big)
\Delta_{\mu}\big(t^{2n-2}t_0^2/c^2|t^n,pt^{n-1}/c^4;q^2,t;p\big).
\end{gather*}
\end{prop}

\begin{proof}
This follows from the dual of \cite[Theorem~4.7]{littlewood} as above.
\end{proof}

\begin{rem}
As in the Littlewood case, removing the factor $\Gampqq\big(q^2t^{1-i}c^4\big)$ makes the right-hand side have a larger domain of formal convergence, namely
\begin{gather*}
\max(|\ord(t_0)|,\max_i|\ord(x_i)|)<\min(1/2-\ord(c),\ord(c)).
\end{gather*}
The corresponding Macdonald polynomial limit is
\begin{gather*}
\lim_{p\to 0} \prod_{1\le i\le n} \frac{\Gampqq\big(qt^i\big)} {\Gampqq\big(t^{i-1}/c^4\big)}
\cL^{\prime(n)}_{p^{1/4}c}\big(p^{-1/4}\vec{x};t;p,q\big) \\
\qquad{} \text{``=''}
\sum_{\mu} \big(c^2q\big)^{|\mu|}
\frac{C^-_\mu\big(qt;q^2,t\big) C^0_\mu\big(t^{n-1}/c^4;q^2,t\big)}
 {C^-_\mu\big(q^2;q^2,t\big)C^0_\mu\big(qt^n;q^2,t\big)} P_{2\mu}(\vec{x};q,t),
\end{gather*}
and can be made rigorous in the usual way.
\end{rem}

When $c=(p/qt)^{1/4}$, this becomes the usual elliptic dual Littlewood sum.

\begin{cor}\label{cor:dual_litt_prod}
When $c=(p/qt)^{1/4}$, the dual Littlewood kernel has the product expansion
\begin{gather*}
\cL^{\prime(n)}_{(p/qt)^{1/4}}(\vec{x};t;p,q) =
\Gampq\big((pq/t)^{1/2}\big)^n
\prod_{1\le i\le n} \Gampqq\big((pq/t)^{1/2}x_i^{\pm 2}\big)\\
\hphantom{\cL^{\prime(n)}_{(p/qt)^{1/4}}(\vec{x};t;p,q) =}{}\times
\prod_{1\le i<j\le n} \Gampq\big((pq/t)^{1/2}x_i^{\pm 1}x_j^{\pm 1}\big).
\end{gather*}
\end{cor}

\begin{rem}
Of course, this also follows directly by duality from the corresponding evaluation for $\cL^{(2n)}_{(pqt)^{1/4}}$.
\end{rem}

The image of this under the $t\mapsto pq/t$ symmetry has a particularly nice expression.

\begin{cor}\label{cor:Q3}
When $c=q^{-1/2}t^{1/4}$, the dual Littlewood kernel has the expression
\begin{gather*}
\cL^{\prime(n)}_{q^{-1/2}t^{1/4}}(\vec{x};t;p,q)=
\prod_{1\le i\le n} \Gampqq\big(t^{1/2}x_i^{\pm 2}\big)
\frac{\Delta^{(n)}_S\big(\vec{x};t^{1/2};p,q\big)} {\Delta^{(n)}_S(\vec{x};t;p,q)}.
\end{gather*}
\end{cor}

One major difference between the Littlewood kernel and its dual is that the
dual appears not to satisfy any branching rule. There is, however, an
analogue of Theorem~\ref{thm:quad_litt}, with essentially the same ``Bailey
lemma''-type proof.

\begin{thm}\label{thm:quad_dual_litt}
The dual Littlewood kernel satisfies the integral identity
\begin{gather*}
\int \cK^{(n)}_{c/d}(\vec{z};\vec{x};t;p,q)
\cL^{\prime(n)}_d(\vec{z};t;p,q)
\Delta^{(n)}_S\big(\vec{z};\big(\sqrt{pq}/c\big)w^{\pm 1},\big(\sqrt{p}/d\big)v^{\pm 1};t;p,q\big) \\
\qquad{}= \prod_{1\le i\le n} \Gampq\big(\big(\sqrt{pq}/d\big)w^{\pm 1}x_i^{\pm 1}\big)\\
\qquad \quad{}\times \int \cK^{(n)}_c(\vec{x};\vec{z};t;p,q)
\Delta^{(n)}_S\big(\vec{z};\big(d\sqrt{p}q/c\big)q^{\pm 1/2}w^{\pm 1},\sqrt{p}v^{\pm 1}/d^2;t;p,q^2\big).
\end{gather*}
\end{thm}

For comparison with Conjectures Q3, Q4, and Q5 of~\cite{littlewood}, we record the interpolation function version of this identity.

\begin{cor}\label{cor:quad_dual_litt}
For otherwise generic parameters satisfying $t^{n-1}t_0t_1t_2u_0=pq^2d^2$,
\begin{gather*}
\int
\cR^{*(n)}_{\blambda} (\vec{z};t_0/d,u_0/d;t;p,q)
\cL^{\prime(n)}_{q^{-1/2}d}(\vec{z};t;p,q)\\
\qquad\quad{}\times
\Delta^{(n)}_S\big(\vec{z};t_0/d,t_1/d,u_0/d,t_2/d,\sqrt{pq}v^{\pm 1}/d;t;p,q\big) \\
\qquad {}=
\frac{\Delta^0_{\blambda}\big(t^{n-1}t_0/u_0|t^{n-1}t_0t_1/d^2;t;p,q\big)}
 {\Delta^0_{\blambda}\big(t^{n-1}t_0/u_0|t^{n-1}t_0t_1/q;t;p,q\big)}
\prod_{\substack{0\le i<n\\0\le r<s<3}}
 \frac{\Gampq\big(t^it_rt_s/d^2\big)}
 {\Gampq\big(t^it_rt_s/q\big)} \\
\qquad\quad{}\times \int
\cR^{*(n)}_{\blambda}\big(\vec{z};q^{-1/2}t_0,q^{-1/2}u_0;t;p,q\big) \\
 \qquad\qquad\quad{}\times \Delta^{(n)}_S\big(\vec{z};
 q^{\pm 1/2}t_0,q^{\pm 1/2}t_1,q^{\pm 1/2}t_2,q^{\pm 1/2}u_0,
 \sqrt{p}q v^{\pm 1}/d^2;t;p,q^2\big).
\end{gather*}
\end{cor}

Again, a suitable specialization allows us to evaluate the right-hand side
in terms of the dual Littlewood kernel.

\begin{cor}\label{cor:dual_litt_recur}
We have the identity
\begin{gather*}
\cL^{\prime(n)}_c(\vec{x};t;p,q)= \prod_{1\le i\le n}
\frac{1}{\Gampq\big(\sqrt{pq}x_i^{\pm 1}/qcd^2,\sqrt{pq}x_i^{\pm 1}/c\big)} \\
\hphantom{\cL^{\prime(n)}_c(\vec{x};t;p,q)=}{}\times \int
\cK^{(n)}_{c/d}(\vec{z};\vec{x};t;p,q)
\cL^{\prime(n)}_d(\vec{z};t;p,q)
\Delta^{(n)}_S\big(\vec{z};\sqrt{pq}/qc^2d,\sqrt{pq}/d;t;p,q\big).
\end{gather*}
\end{cor}

The case $d=q^{-1/2}t^{1/4}$ is particularly useful.

\begin{cor}\label{cor:dual_litt_formal_in_q}
We have the expression
\begin{gather*}
\cL^{\prime(n)}_c(\vec{x};t;p,q) =\prod_{1\le i\le n} \Gampq\big((pq)^{1/2}cx_i^{\pm 1},(pqt)^{1/2}cx_i^{\pm 1}\big) \\
\hphantom{\cL^{\prime(n)}_c(\vec{x};t;p,q) =}{}\times \int
\cK^{(n)}_{q^{1/2}t^{-1/4}c}(\vec{z};\vec{x};t;p,q)
\Delta^{(n)}_S\big(\vec{z};\pm t^{1/4},-p^{1/2}t^{1/4},p^{1/2}/t^{1/4}c^2;t^{1/2};p,q\big).
\end{gather*}
\end{cor}

Unlike the defining integral for $\cL^{\prime(n)}_c$, this has a well-behaved formal expansion in $q$ for a~range of valuations of $c$, namely $-1/2<\ord_q(c)<0$ (which extends to $\ord_q(c)=0$ if we divide by the limit as $q\to 0$). (This in particular explains why $\cL^{\prime(n)}_{(p/qt)^{1/4}}$ has a nice formal expansion in~$q$.)

The case $c=1$ of Corollary \ref{cor:dual_litt_recur} can be generalized somewhat, as it is then easier to cancel parameters. Note that unlike
Corollary \ref{cor:dual_litt_recur}, the resulting identity is actually equivalent (by the usual argument) to the full Theorem~\ref{thm:quad_dual_litt}.

\begin{cor}\label{cor:dual_litt_pre_van}
The dual Littlewood kernel satisfies the identity
\begin{gather*}
\int \cK^{(n)}_{1/c}(\vec{z};\vec{x};t;p,q)\cL^{\prime(n)}_c(\vec{z};t;p,q)
\Delta^{(n)}_S\big(\vec{z};p^{1/2}v^{\pm 1}/c;t;p,q\big) \\
\qquad {}= \prod_{1\le i\le n} \Gampqq\big(p^{1/2}v^{\pm 1}x_i^{\pm 1}/c^2\big)
\frac{\Delta^{(n)}_S\big(\vec{x};t;p,q^2\big)} {\Delta^{(n)}_S(\vec{x};t;p,q)}.
\end{gather*}
\end{cor}

\begin{proof}
 We find that both sides have well-behaved formal expansions in $q$, so
 long as $1/2<\ord_q(c)<0$, so that we may argue as in the proof of
 Proposition \ref{prop:dual_litt_c_1}.
\end{proof}

We also have the following identity obtained by specializing Theorem~\ref{thm:quad_dual_litt} so that we may apply Theorem~\ref{thm:L2_kern} to
the right-hand side.

\begin{cor}
The integral
\begin{gather*}
\prod_{1\le i\le n} \frac{1}{\Gampq\big(\sqrt{p/q}(q^{-1/2}v)^{\pm 1}z_i^{\pm 1}/d^2\big)}\\
\qquad{}\times \int \cK^{(n)}_c(\vec{z};\vec{x};t;p,q) \cL^{\prime(n)}_d(\vec{z};t;p,q) \Delta^{(n)}_S\big(\vec{z};p^{1/2}q(v/c)^{\pm 1},\sqrt{p/q}\big(q^{-1/2}v\big)^{\pm 1}/cd^2;t;p,q\big)
\end{gather*}
is invariant under $v\mapsto 1/v$.
\end{cor}

We also have an analogue of Theorem \ref{thm:van_litt}, using Corollary~4.14 of~\cite{littlewood} in place of Corollary~4.8 op.\ cit. (Note that
there is a typo there: $qt^{n-2}$ should read $qt^{n-1}$.)

\begin{thm}\label{thm:van_dual_litt}
For generic parameters satisfying $t^{n-1}t_0t_1u_0=\sqrt{p}q$,
\begin{gather*}
\int \tcR^{(n)}_{\blambda}(\vec{z};t_0/d{:}dt_0,t_1/d,dt_1;u_0/d,du_0/q;t;p,q) \cL^{\prime(n)}_{q^{-1/2}d}(\vec{z};t;p,q) \\
\qquad{}\times \Delta^{(n)}_S\big(\vec{z};t_0/d,t_1/d,u_0/d,\sqrt{p}q/d;t;p,q\big)
\end{gather*}
vanishes unless $\blambda$ has the form $(1,2)\bmu$, when it equals
\begin{gather*}
Z
\frac{\Delta_{ \bmu}\big(q/u_0^2|t^n,t^{n-1}t_0^2,1/t^{n-1}t_0u_0,q/t^{n-1}t_0u_0;t;p,q^2\big)}
 {\Delta_{(1,2)\bmu}\big(q/u_0^2|t^n,t^{n-1}t_0^2,1/t^{n-1}t_0u_0,q/t^{n-1}t_0u_0;t;p,q\big)},
\end{gather*}
where
\begin{gather*}
Z=\prod_{0\le i<n}
 \Gampqq\big(t^{i+1},t^it_0^2,t^it_1^2,t^iu_0^2\big)
 \Gampq\big(t^it_0t_1/d^2,t^it_0u_0/d^2,t^it_1u_0/d^2,t^it_0t_1,t^it_0u_0,t^it_1u_0\big).
\end{gather*}
\end{thm}

The analogue of Theorem \ref{thm:van_litt_interp} is even simpler than in
the Littlewood case, as we can simply specialize $\vec{x}$ to a partition
in Corollary \ref{cor:dual_litt_pre_van}. This is particularly lucky
since once $\blambda$ becomes nontrivial, we would need to compute residues
at second order poles of the integrand!

\begin{thm}\label{thm:van_dual_litt_interp}
For $t^{n-1}t_0u_0=q$, the integral
\begin{gather*}
\int
\cR^{*(n)}_\blambda(\vec{z};t_0/d,u_0/d;t;p,q)
\cL^{\prime(n)}_{q^{-1/2}d}(\vec{z};t;p,q)
\Delta^{(n)}_S\big(\vec{z};t_0/d,u_0/d,\sqrt{pq}v^{\pm 1}/d;t;p,q\big)
\end{gather*}
vanishes unless $\blambda$ has the form $(1,2)\bmu$ for some $\bmu$, when
it equals
\begin{gather*}
Z\frac{\Delta_{\bmu}\big(q/u_0^2|t^n,t^{n-1}pq^2,\sqrt{pq}v^{\pm 1}d^2/u_0;t;p,q^2\big)}
 {\Delta_{(1,2)\bmu}\big(q/u_0^2|t^n,t^{n-1}pq,\sqrt{pq}v^{\pm 1}d^2/u_0;t;p,q\big)},
\end{gather*}
with
\begin{gather*}
Z=\prod_{0\le i<n}
 \Gampq\big(t^{-i}q/d^2\big)
 \Gampqq\big(t^{i+1},t^{-i}q,t^it_0^2,t^iu_0^2,\sqrt{pq}v^{\pm 1}t^it_0/d^2,\sqrt{pq}v^{\pm 1}t^i u_0/d^2\big).
\end{gather*}
\end{thm}

We also have an analogue of Theorem \ref{thm:spec_int_litt}, this time in the form of a difference equation.

\begin{thm}\label{thm:spec_diff_dual_litt}
If $t_0t_1t_2t_3 = \big(p/c^2\big)^2$, then
\begin{gather*}
D^{(n)}_q(t_0,t_1,t_2,t_3;t;p)_{\vec{x}} \cL^{\prime(n)}_c(\vec{x};t;p,q)
\end{gather*}
is invariant under $t_r\mapsto p/c^2t_r$.
\end{thm}

\begin{cor}
We have
\begin{gather*}
\cL^{\prime(n)}_{q^{-1/2}c}(\vec{x};t;p,q)
=
\frac{1}{\prod\limits_{1\le i\le n} \theta_p\big(p^{1/2} v x_i^{\pm 1}\big)}
D^{(n)}_q\big(p^{1/2}v^{\pm 1}/c^2;t;p\big)_x \cL^{\prime(n)}_c(\vec{x};t;p,q).
\end{gather*}
\end{cor}

\begin{proof}
When $v=c$, this is a special case of Corollary \ref{cor:dual_litt_recur}. That the right-hand side is independent of $v$ follows from the case
$t_2t_3=p$ of the difference equation.
\end{proof}

In addition to the Littlewood and dual Littlewood kernel, there is one more such kernel we wish to consider, this time related to Conjecture~L3 of~\cite{littlewood}.

\begin{thm}\label{thm:L3_kern}
The interpolation kernel satisfies the integral identity
\begin{gather*}
\int \cK^{(n)}_c\big(\vec{z}\,^2;\vec{y};t;p,q\big)
\Delta^{(n)}_S\big(\vec{z};v_0,v_1,v_2,v_3;t^{1/2};p^{1/2},q^{1/2}\big) \\
\qquad{}=
\prod_{\substack{1\le i\le n\\0\le r\le 3}}
 \Gampq\big(c v_r^2 y_i^{\pm 1}\big) \\
\qquad\quad{}\times\! \int\!
\cK^{(n)}_c(\vec{z}\,^2;\vec{y};t;p,q)
\Delta^{(n)}_S\big(\vec{z};\sqrt{pq}/cv_0,\sqrt{pq}/cv_1,\sqrt{pq}/cv_2,\sqrt{pq}/cv_3;t^{1/2};p^{1/2},q^{1/2}\big),
\end{gather*}
subject to the balancing condition $v_0v_1v_2v_3=pq/c^2$.
\end{thm}

\begin{proof} As usual, both sides have the same limit as $p\to 0$, $c\sim p^{1/2}$, and dividing by the common limit gives formal Puiseux series in $p$ with rational function coefficients, so it suffices to prove a Zariski dense set of special cases.

 In a suitable limit $v_0v_1\to q^{-m/2}$, this becomes an identity involving finite sums of interpolation functions. The interpolation functions are modular (\cite[Section~6]{bctheta}), but the evaluation at~$z_i^2$ is not preserved by the modular group. In other words, the identity depends on a choice of 2-isogeny, which we may replace by any other 2-isogeny. Upon doing so, we find that the identity we require is a special case of Theorem \ref{thm:L2_kern}; to be precise, it is obtained from that identity by first swapping $p$ and $q$ then taking the limit $v_0v_1\to q^{-m}$.
\end{proof}

\begin{rem}
 Unlike Theorems \ref{thm:L1_kern} and \ref{thm:L2_kern}, we have been
 unable to come up with a direct argument (i.e., not using a modular
 transformation). The difficulty is that the relevant analogue of the
 integral and difference operators we used above is the operator
 corresponding to the kernel with $c=-1$, but this is trivial!
\end{rem}

Once more, if we cancel two of the parameters, we find that the result is
independent of the remaining degree of freedom, motivating the following
definition.

\begin{defn}
The {\em Kawanaka kernel} is defined by
\begin{gather*}
\cL^{-(n)}_c(\vec{x};t;p,q) :=
\prod_{1\le i\le n} \frac{1}{\Gamppqq\big(pq v^{\pm 2} x_i^{\pm 1}/c\big)}\\
\hphantom{\cL^{-(n)}_c(\vec{x};t;p,q) :=}{}\times
\int \cK^{(n)}_c\big(\vec{z}\,^2;\vec{x};t^2;p^2,q^2\big)
\Delta^{(n)}_S\big(\vec{z};(pq)^{1/2}v^{\pm 1}/c;t;p,q\big).
\end{gather*}
\end{defn}

Note here that unlike the Littlewood and dual Littlewood kernels, the interpolation kernel in the integrand has parameters $\big(t^2;p^2,q^2\big)$ rather than $(t;p,q)$. This is important, as otherwise the right-hand side would have square roots, and the result would depend on the choices of sign. (The residual choice of a square root of~$pq$ can be absorbed by negating~$v$.)

We have an analogue of the $t\mapsto pq/t$ symmetry.

\begin{prop}
The Kawanaka kernel satisfies the symmetry
\begin{gather*}
\cL^{-(n)}_c(\vec{x};pq/t;p,q) = \Gamppqq\big(t^2\big)^n \prod_{1\le i<j\le n} \Gamppqq\big(t^2 x_i^{\pm 1}x_j^{\pm 1}\big) \cL^{-(n)}_c(\vec{x};-t;p,q).
\end{gather*}
\end{prop}

\begin{prop}
The Kawanaka kernel has the evaluation
\begin{gather*}
\cL^{-(n)}_c\big(t^{2n-2}v,\dots,v;t;p,q\big) =
\prod_{1\le i\le n}
\frac{\Gamppqq\big(t^{2-2i}c^2\big) \Gampq\big({-}t^{i-1}cv,-t^{i+1-2n}c/v\big)}
 {\Gampq\big(t^{1-i}c^2,-t^i\big)}.
\end{gather*}
\end{prop}

When $c=t,pq/t$, we can express the result in terms of the interpolation
kernel; we give the $c=pq/t$ case, as the other follows from the symmetry
(and is more complicated).

\begin{prop}
The Kawanaka kernel has the special case
\begin{gather*}
\cL^{-(n)}_{pq/t}(\vec{x};t;p,q) =
\cK^{(n)}_{pq/t}\big(\sqrt{\vec{x}};-\sqrt{\vec{x}};t;p,q\big).
\end{gather*}
\end{prop}

\begin{proof}
This follows by substituting
\begin{gather*}
\cK^{(n)}_{pq/t}\big(\vec{z}\,^2;\vec{x};t^2;p^2,q^2\big)=
\cK^{(n)}_{\sqrt{pq/t}}\big(\vec{z};\sqrt{\vec{x}};t;p,q\big)
\cK^{(n)}_{\sqrt{pq/t}}\big(\vec{z};-\sqrt{\vec{x}};t;p,q\big)
\end{gather*}
into the definition of the Kawanaka kernel, then simplifying using the braid relation.
\end{proof}

We again omit the pfaffian cases $t=\pm q$, $t=\pm p$; see~\cite{littlewood} for the description of the non-kernel factors of the
integrands as pfaffians in these cases.

As before, the name comes from a formal expansion. The undeformed version
of this expansion is the elliptic analogue of Kawanaka's identity, an
identity of Macdonald polynomials conjectured by Kawanaka in~\cite{KawanakaN:1999}, and proved by Langer, Schlosser, and Warnaar in~\cite{LangerR/SchlosserMJ/WarnaarSO:2009}.

\begin{thm}
If $\max(|\ord(t_0)|,\max_i|\ord(x_i)|)<\ord(c)\le 1/2$, then
\begin{gather*}
\cL^{-(n)}_c(\vec{x};t;p,q)= \prod_{1\le i\le n}
\frac{\Gamppqq\big(c^2t_0x_i^{\pm 1},\big(c^2/t^{2n-2}t_0\big)x_i^{\pm 1}\big)}
 {\Gamppqq\big(t^{2-2i}c^2\big)\Gampq\big(t^{i-1}ct_0,t^{i+1-2n}c/t_0,t^{1-i}c^2,-t^i\big)}\\
\hphantom{\cL^{-(n)}_c(\vec{x};t;p,q)=}{}\!\times \!\sum_\mu
R^{*(n)}_\mu\big(\vec{x};t_0,c^2/t^{2n-2}t_0;q^2,t^2;p^2\big)
\Delta_{\mu}\big(t^{2n-2}t_0/c|t^n,pqt^{n-1}/c^2;q,t;p\big).
\end{gather*}
\end{thm}

\begin{proof}
As usual, it suffices to prove the case $\ord(t_0)=0$, $\vec{x}$ a
partition based at $t_0$, when it follows from Theorem \ref{thm:L3_kern},
as discussed in \cite{littlewood}.
\end{proof}

\begin{rem}
As before, the right-hand side converges formally whenever
\begin{gather*}
\max(|\ord(t_0)|,\max_i|\ord(x_i)|)<\min(\ord(c),1-\ord(c)).
\end{gather*}
The corresponding Macdonald polynomial limit is
\begin{gather*}
\lim_{p\to 0}
\prod_{1\le i\le n} \frac{\Gampq\big({-}t^i\big)}
 {\Gampq\big(qt^{i-1}/c^2\big)}
 \cL^{-(n)}_{p^{1/2}c}\big(\dots,p^{-1/2}x_i,\dots;t;p,q\big) \\
\qquad{} \text{``=''} \sum_\mu
(-c)^{|\mu|} \frac{C^-_\mu(-t;q,t)C^0_\mu\big(qt^{n-1}/c^2;q,t\big)}
 {C^-_\mu( q;q,t)C^0_\mu\big({-}t^n;q,t\big)} P_\mu\big(\vec{x};q^2,t^2\big),
\end{gather*}
which becomes Kawanaka's identity when the left-hand side is specialized to a product.
\end{rem}

This summation has the usual special case giving a product, though in this
case the corresponding summation is actually new.

\begin{thm}
The Kawanaka kernel has the following special case with a product expansion
\begin{gather*}
\cL^{-(n)}_{-(pq/t)^{1/2}}(\vec{x};-t;p,q)= \Gamppqq(pq/t)^n
\!\prod_{1\le i\le n} \!\Gampq\big((pq/t)^{1/2} x_i^{\pm 1}\big)
\!\prod_{1\le i<j\le n}\! \Gamppqq\big((pq/t) x_i^{\pm 1}x_j^{\pm 1}\big).
\end{gather*}
\end{thm}

\begin{proof}
 Both sides clearly have well-behaved formal expansions, so it suffices to
 prove this in the case $x_i = t^{2n-2i}q^{2\lambda_i}t_0$ for some
 partition $\lambda$. The resulting identity of multivariate elliptic
 functions is a modular transform of a special case of Corollary~\ref{cor:dual_litt_prod}.
\end{proof}

\begin{rem}
 If we replace the left-hand side by its formal expansion, the resulting
 formal sum may be viewed as an elliptic analogue of Kawanaka's identity.
 This gives an alternate proof of the latter by a careful limit (i.e.,
 replace $p$ by $p^{4N}$ and multiply $\vec{x}$ by $p$, then take the
 limit $N\to\infty$).
\end{rem}

Again, this is particularly nice after applying the $t\mapsto pq/t$ expression.

\begin{cor}\label{cor:Q7}
For $c=t^{1/2}$, the Kawanaka kernel has the following expression
\begin{gather*}
\cL^{-(n)}_{t^{1/2}}(\vec{x};t;p,q) =
\prod_{1\le i\le n} \Gampq\big({-}t^{1/2} x_i^{\pm 1}\big)
\frac{\Delta^{(n)}_S\big(\vec{x};t;p^2,q^2\big)}
 {\Delta^{(n)}_S\big(\vec{x};t^2;p^2,q^2\big)}.
\end{gather*}
\end{cor}

Like the dual Littlewood kernel, the Kawanaka kernel does not appear to
satisfy any simple branching rule.

The analogues of most of the integral identities are straightforward.

\begin{thm}\label{thm:quad_kaw}
The Kawanaka kernel satisfies the integral identity
\begin{gather*}
\int \cK^{(n)}_c(\vec{x};\vec{z};t;p,q)
\cL^{-(n)}_d\big(\vec{z};t^{1/2};p^{1/2},q^{1/2}\big)
\Delta^{(n)}_S\big(\vec{z};\sqrt{pq}w^{\pm 2}/cd,\sqrt{pq}v^{\pm 2}/d;t;p,q\big) \\
\qquad{}= \prod_{1\le i\le n} \Gampq\big(\sqrt{pq}w^{\pm 2}x_i^{\pm 1}/d\big) \\
\qquad\quad{}\times \int
\cK^{(n)}_{cd}\big(\vec{x};\vec{z}\,^2;t;p,q\big)
\Delta^{(n)}_S\big(\vec{z};\pm(pq)^{1/4}w^{\pm 1}/\sqrt{c},(pq)^{1/4}v^{\pm 1}/d;t^{1/2};p^{1/2},q^{1/2}\big).
\end{gather*}
\end{thm}

\begin{cor}\label{cor:quad_kaw}
For otherwise generic parameters satisfying $t^{n-1}t_0t_1t_2u_0=pqd^2$,
\begin{gather*}
\int {\cal R}^{*(n)}_{\blambda}(\vec{z};t_0/d,u_0/d;t;p,q)
\cL^{-(n)}_d\big(\vec{z};t^{1/2};p^{1/2},q^{1/2}\big) \\
 \qquad\quad{}\times \Delta^{(n)}_S\big(\vec{z};t_0/d,t_1/d,t_2/d,u_0/d,\sqrt{pq}v^{\pm 2}/d;t;p,q\big)
 \\
\qquad{}= \frac{\Delta^0_{\blambda}\big(t^{n-1}t_0/u_0|t^{n-1}t_0t_1/d^2;t;p,q\big)}
 {\Delta^0_{\blambda}\big(t^{n-1}t_0/u_0|t^{n-1}t_0t_1;t;p,q\big)}
\prod_{\substack{0\le i<n\\0\le r<s<3}}
 \frac{\Gampq\big(t^it_rt_s/d^2\big)}
 {\Gampq\big(t^it_rt_s\big)} \\
\qquad\quad{}\times \int
 {\cal R}^{*(n)}_{\blambda}\big(\vec{z}\,^2;t_0,u_0;t;p,q\big) \\
\qquad\quad{}\times
\Delta^{(n)}_S\big(\vec{z};
 \pm\sqrt{t_0},\pm\sqrt{t_1},\pm\sqrt{t_2},\pm\sqrt{u_0},
 (pq)^{1/4}v^{\pm 1}/d;t^{1/2};p^{1/2},q^{1/2}\big).
 \end{gather*}
\end{cor}

\begin{cor}
The Kawanaka kernel satisfies the identity
\begin{gather*}
\cL^{-(n)}_c\big(\vec{x};t^{1/2};p^{1/2},q^{1/2}\big) =
\prod_{1\le i\le n} \frac{1}{\Gampq\big({-}\sqrt{pq}x_i^{\pm 1}/cd^2,-\sqrt{pq}x_i^{\pm 1}/c\big)}\\
\qquad{}\times \int \cK^{(n)}_{c/d}(\vec{x};\vec{y};t;p,q)
\cL^{-(n)}_d\big(\vec{y};t^{1/2};p^{1/2},q^{1/2}\big)
\Delta^{(n)}_S\big(\vec{y};-\sqrt{pq}/c^2d,-\sqrt{pq}/d;t;p,q\big).
\end{gather*}
\end{cor}

\begin{cor}
The expression
\begin{gather*}
\prod_{1\le i\le n} \Gampq\big(\big(\sqrt{pq}ve/d\big)x_i^{\pm 1}\big) \int
\cK^{(n)}_{ce}(\vec{x};\vec{z};t;p,q) \cL^{-(n)}_{d}\big(\vec{z};t^{1/2};p^{1/2},q^{1/2}\big)\\
\qquad{} \times \Delta^{(n)}_S\big(\vec{z};\sqrt{pq}v^{\pm 1}/cd,-\sqrt{pq}/de^2,-\sqrt{pq}/d;t;p,q\big)
\end{gather*}
is invariant under swapping $d$ and $e$.
\end{cor}

If we attempt to obtain an identity by specializing the right-hand side to an instance of Theorem~\ref{thm:L3_kern}, we find that the resulting
identity is trivial. We also do not have an analogue of Theorems~\ref{thm:spec_int_litt} or~\ref{thm:spec_diff_dual_litt}.

The analogue of the vanishing integrals of Theorems \ref{thm:van_litt}
and \ref{thm:van_dual_litt} is again straightforward, now using
\cite[Corollary~4.16]{littlewood}. Note that in this case, the integral never
vanishes.

\begin{thm}\label{thm:van_kaw}
 For generic parameters satisfying $t^{n-1}t_0t_1u_0=\sqrt{pq}$, we have the evaluation
\begin{gather*}
\int \tilde{R}^{(n)}_{\blambda}(\vec{z};t_0/d{:}dt_0,t_1/d,t_1d;u_0/d,du_0;t;p,q)
\cL^{-(n)}_{-d}\big(\vec{z};t^{1/2};p^{1/2},q^{1/2}\big) \\
\qquad\quad {}\times \Delta^{(n)}_S\big(\vec{z};t_0/d,t_1/d,u_0/d,\sqrt{pq}/d;t;p,q\big) \\
\qquad{} = Z
\frac{\Delta_{\blambda}\big({-}1/u_0|t^{n/2},t^{(n-1)/2}t_0,\pm \big(t^{n-1}t_0u_0\big)^{-1/2};t^{1/2};p^{1/2},q^{1/2}\big)}
 {\Delta_{\blambda}\big(1/u_0^2|t^n,t^{n-1}t_0^2,1/t^{n-1}t_0u_0,1/t^{n-1}t_0u_0;t;p,q\big)},
\end{gather*}
where
\begin{gather*}
Z= \prod_{0\le i<n} \Gamphqh\big(t^{(i+1)/2},t^{i/2}t_0,t^{i/2}t_1,t^{i/2}u_0\big) \\
 \hphantom{Z=}{} \times \prod_{0\le i<n}
 \Gampq\big(t^i t_0t_1,t^it_0u_0,t^it_1u_0,t^i t_0t_1/d^2,t^it_0u_0/d^2,t^it_1u_0/d^2\big).
\end{gather*}
\end{thm}

The natural analogue to Theorems~\ref{thm:van_litt_interp} and~\ref{thm:van_dual_litt_interp} would involve an evaluation for the integral
\begin{gather*}
\int \cR^{*(n)}_{\blambda}(\vec{z};t_0/d,u_0/d;t;p,q)
\cL^{-(n)}_{-d}\big(\vec{z};t^{1/2};p^{1/2},q^{1/2}\big)
\Delta^{(n)}_S\big(\vec{z};t_0/d,u_0/d,\sqrt{pq}v^{\pm 1}/d;t;p,q\big),
\end{gather*}
but here we encounter a significant difficulty: the corresponding right-hand side of Corollary~\ref{cor:quad_kaw} now has {\em two} pairs of parameters multiplying to $t^{(1-n)/2}$, and thus the standard residue calculation no longer applies. It is likely one could express the result as a sum of two nice terms, though.

\section{Quadratic transformations}\label{section8}

In \cite[Section~5]{littlewood}, the author developed a sequence of seven
conjectural quadratic transformations, each of which closely resembles a
special case of one of Corollaries \ref{cor:quad_litt},
\ref{cor:quad_dual_litt} or \ref{cor:quad_kaw}. This suggests that the
machinery we have developed should be useful in proving these conjectures,
and this is indeed the case. In fact, it turns out that we have
essentially already proved two of them! For instance, if we take
$d=t^{1/4}$ in Corollary~\ref{cor:quad_dual_litt}, and expand $\cL^{\prime(n)}_{q^{-1/2}t^{1/4}}$ using Corollary~\ref{cor:Q3}, the result is precisely Conjecture Q3 of~\cite{littlewood}. Similarly, Conjecture~Q7 of~\cite{littlewood} follows immediately from Corollaries~\ref{cor:Q7} and~\ref{cor:quad_kaw}. If we take the corresponding
substitutions in Theorems~\ref{thm:quad_dual_litt} and~\ref{thm:quad_kaw}, we obtain the following results.

\begin{thm}[Q3]\label{thm:Q3}
For otherwise generic parameters satisfying $c^2t_0t_1=pq^2t^{1/2}$, $v_0v_1=pq/t$,
\begin{gather*}
\int K^{(n)}_{t^{-1/4}c}(\vec{z};\vec{x};t;p,q)
\prod_{1\le i\le n} \Gampqq\big(t^{1/2} z_i^{\pm 2}\big)
\Delta^{(n)}_S\big(\vec{z};t^{-1/4}t_0,t^{-1/4}t_1,t^{1/4}v_0,t^{1/4}v_1;t^{1/2};p,q\big) \\
\qquad{} = \prod_{1\le i\le n} \Gampq\big(t^{-1/2}ct_0x_i^{\pm 1},t^{-1/2}ct_1x_i^{\pm 1}\big) \\
\qquad\quad{}\times \int K^{(n)}_{q^{-1/2}c}(\vec{z};\vec{x};t;p,q)
\Delta^{(n)}_S\big(\vec{z};q^{\pm 1/2}t_0,q^{\pm 1/2}t_1,q^{1/2}v_0,q^{1/2}v_1;t;p,q^2\big).
\end{gather*}
\end{thm}

\begin{rem}
Note that the factor $\prod\limits_{1\le i\le n} \Gampqq\big(t^{1/2} z_i^{\pm 2}\big)$ is equivalent to adding four parameters $\pm t^{1/4},\pm p^{1/2}t^{1/4}$ to the elliptic Selberg density.
\end{rem}

\begin{thm}[Q7]\label{thm:Q7}
For otherwise generic parameters satisfying $c^2t_0t_1=pqt^{1/2}$, $v_0v_1 = \sqrt{pq/t}$,
\begin{gather*}
\int K^{(n)}_{t^{-1/4}c}(\vec{z};\vec{x};t;p,q)
 \prod_{1\le i\le n} \Gampq\big(t^{1/2} z_i^{\pm 2}\big)
\Delta^{(n)}_S\big(\vec{z};t^{-1/4}t_0,t^{-1/4}t_1,-t^{1/4}v_0^2,-t^{1/4}v_1^2;t^{1/2};p,q\big) \\
\qquad{}=
\prod_{1\le i\le n} \Gampq\big(t^{-1/2}ct_0x_i^{\pm 1},t^{-1/2}ct_1x_i^{\pm 1}\big) \\
\qquad\quad{}\times \int K^{(n)}_{-c}\big(\vec{z}\,^2;\vec{x};t;p,q\big)
\Delta^{(n)}_S\big(\vec{z};\pm \sqrt{-t_0},\pm \sqrt{-t_1},v_0,v_1;t^{1/2};p^{1/2},q^{1/2}\big).
\end{gather*}
\end{thm}

\begin{rem}
Again, the factor $\prod\limits_{1\le i\le n} \Gamphqh\big(t^{1/2} z_i^{\pm 2}\big)$ could be replaced by a quadruple of elliptic Selberg parameters, in this case $t^{1/4}$, $p^{1/2}t^{1/4}$, $q^{1/2}t^{1/4}$, $p^{1/2}q^{1/2}t^{1/4}$.
\end{rem}

Specializing $\vec{x}=t^{n-1}v,\dots,v$ in either theorem gives a quadratic transformation of higher order elliptic Selberg integrals (the same as
setting $\blambda=0$ in the original conjectures of~\cite{littlewood}), originally proved in~\cite{vandeBultFJ:2011}. We also obtain a~multivariate quadratic evaluation from the first case, from the normalization of Theorem~\ref{thm:van_dual_litt_interp}.

\begin{cor}\label{cor:Q3_eval}
For otherwise generic parameters satisfying
$t^{2n-1}t_0t_1=q$, $v_0v_1 = pq/t$,
\begin{gather*}
\int \Delta^{(n)}_S\big(\vec{z};t_0,t_1,v_0,v_1,\pm \sqrt{t},\pm \sqrt{pt};t;p,q\big)\\
\qquad{}= \prod_{0\le i<n}
 \Gampq\big(t^{-2i-1}q\big)
 \Gampqq\big(t^{2i+2},t^{-2i}q,t^{2i+1}t_0^2,t^{2i+1}t_1^2,t^{2i}t_0v_0,t^{2i}t_1v_0,t^{2i}t_0v_1,t^{2i}t_1v_1\big).
\end{gather*}
\end{cor}

\begin{rem}
Of course, since the above Selberg integral has 8 parameters, one can obtain a~large number of other quadratic evaluations by applying the~$W(E_7)$ symmetry of the integral, \cite[Section~9]{xforms} (essentially just the special case of Theorem~\ref{thm:bailey_xform} above in which~$\vec{x}$
and~$\vec{y}$ are geometric progressions).
\end{rem}

Now, if we take $\ord(c)=1/2$, $\ord(t_0)=\ord(t_0)=0$, $\ord(v_0)=\ord(v_1)=1/2$ in Theorem~\ref{thm:Q3}, then the integrals on either side become Koornwinder integrals in the limit $p\to 0$, so that we may apply the results of Section~\ref{section5} to analytically continue in the dimension. Applying the Macdonald involution, reparametrizing, then specializing to a finite-dimensional integral gives the following result.

\begin{thm}\label{thm:Q2}
For otherwise generic parameters satisfying $c^2t_0t_1=pq^{1/2}/t$,
\begin{gather*}
\int
K^{(2n)}_{q^{1/4}c} \big(q^{\pm 1/4}\vec{z};\vec{x};t;p,q\big)
\prod_{1\le i\le n} \Gampq\big(t z_i^{\pm 2}\big)
\Delta^{(n)}_S\big(\vec{z};t_0,t_1,\sqrt{pq}v^{\pm 1};t;p,q^{1/2}\big) \\
\qquad {}=
\prod_{1\le i\le 2n} \Gampq\big(q^{1/2}c t_0 x_i^{\pm 1},q^{1/2}c t_1 x_i^{\pm 1}\big)\\
\qquad\quad{}\times \int K^{(2n)}_{t^{1/2}c}\big(t^{\pm 1/2}\vec{z};\vec{x};t;p,q\big)
\Delta^{(n)}_S\big(\vec{z};t_0,tt_0,t_1,tt_1,p^{1/2}qv^{\pm 1};t^2;p,q\big).
\end{gather*}
\end{thm}

Again, this becomes a quadratic transformation when $\vec{x}=t^{2n-1}w,\dots,w$. Unlike Theorem~\ref{thm:Q3}, this does not give an evaluation of $\cL^{(2n)}_{q^{-1/4}t^{1/2}}$, but does correspond to the following ``distributional'' statement:
\begin{gather*}
\int f(\vec{x}) \cL^{(2n)}_{q^{-1/4}t^{1/2}}(\vec{x};t;p,q)
\Delta^{(2n)}_S(\vec{x};t;p,q) = \int f\big(q^{\pm 1/4}\vec{z}\big)\!
\prod_{1\le i\le n}\! \Gampq\big(t z_i^{\pm 2}\big) \Delta^{(n)}_S\big(\vec{z};t;p,q^{1/2}\big),
\end{gather*}
for suitable test functions $f$. Although it is difficult to make precise the notion of ``suitable'' here, it certainly follows that, since Theorem~\ref{thm:Q2} is obtained from Theorem~\ref{thm:quad_litt} via this substitution, the same substitution applies to the corollaries of this
theorem. In particular, substituting into Corollary~\ref{cor:quad_litt} (i.e., specializing $\vec{x}$ to a partition pair) proves Conjecture~Q2 of~\cite{littlewood}. (To be precise, we must also swap $p$ and $q$, but this is no difficulty.) There is also a quadratic evaluation, but since
two of the parameters multiply to a negative (but large) power of~$t$, there are significant contour issues, so we omit the details.

A similar calculation applies when dualizing Theorem~\ref{thm:Q7}; the only
difference is that in the resulting elliptic Selberg integral, two of the
parameters are $1$ and~$t^{1/2}$, so that we may apply the
$\tau_{1,t^{1/2};t}$ symmetry~\eqref{eq:koorn_tau_symm} without affecting
the finite dimensionality of either the kernel or the integral. We obtain
the following result, which we state in ``distributional'' form for
concision. Substituting this into Corollary \ref{cor:quad_kaw} proves
Conjecture Q6 of \cite{littlewood}.

\begin{thm}\label{thm:Q6}
The Kawanaka kernel has the ``distributional'' limits
\begin{gather*}
\int f(\vec{x}) \cL^{-(2n)}_{-q^{-1/4}}\big(\vec{x};t^{1/2};p^{1/2},q^{1/2}\big) \Delta^{(2n)}_S(\vec{x};t;p,q) \\
\qquad {}=
\int f\big(q^{\pm 1/4}\vec{z}\big) \Delta^{(n)}_S\big(\vec{z};1,p^{1/2},t^{1/2},p^{1/2}t^{1/2};t;p,q^{1/2}\big)
\end{gather*}
and
\begin{gather*}
\int f(\vec{x}) \cL^{-(2n+1)}_{-q^{-1/4}}\big(\vec{x};t^{1/2};p^{1/2},q^{1/2}\big)
\Delta^{(2n+1)}_S(\vec{x};t;p,q) \\
\qquad {}= \Gampqh\big(p^{1/2},t^{1/2},p^{1/2}t^{1/2}\big)
\int
f\big(q^{\pm 1/4}\vec{z},q^{1/4}\big)
\Delta^{(n)}_S\big(\vec{z};t,p^{1/2},t^{1/2},p^{1/2}t^{1/2};t;p,q^{1/2}\big),
\end{gather*}
in the sense that Theorem~{\rm \ref{thm:quad_kaw}} and its corollaries continue to hold after the stated specialization.
\end{thm}

We would expect (following the derivation of \cite{littlewood}) to obtain
another such result by dualizing Theorem~\ref{thm:Q3}, after first swapping~$p$ and~$q$. At first glance, however, this appears impossible, for a very
simple reason: there is no way to assign valuations to the parameters so
that the integral becomes a Koornwinder integral in the limit! (Indeed, it
is not even clear whether we can specialize the valuations in such a way
that the limiting integral has an evaluation at all~\dots) The simplest way
to avoid this problem is to note that Corollary~\ref{cor:dual_litt_pre_van}
implies Theorem~\ref{thm:quad_dual_litt} in much the same way as the
definition of $\cL^{\prime(n)}_c$. And, as we already observed in the proof
of that corollary, the right-hand side of that identity {\em does} have a
well-behaved formal expansion (albeit in~$q$, rather than~$p$).

In other words, we need only dualize the following identity, simply the special case $c=q^{-1/2}t^{1/4}$ of Corollary~\ref{cor:dual_litt_pre_van}, except with $p$ and $q$ swapped.

\begin{lem}
The interpolation kernel satisfies the quadratic evaluation
\begin{gather*}
\int \cK^{(n)}_{p^{1/2}t^{-1/4}} (\vec{z};\vec{x};t;p,q)
\Delta^{(n)}_S\big(\vec{z};\pm t^{1/4},\pm q^{1/2}t^{1/4},p^{1/2}q^{1/2}v^{\pm 1}/t^{1/4};t^{1/2};p,q\big) \\
\qquad {}=
\prod_{1\le i\le n} \Gamppq\big(v^{\pm 1}x_i^{\pm 1}\sqrt{p^2q/t}\big)
\frac{\Delta^{(n)}_S\big(\vec{x};t;p^2,q\big)} {\Delta^{(n)}_S(\vec{x};t;p,q)}.
\end{gather*}
\end{lem}

At this point, if $\ord(v)=\ord(x)=0$, both sides have perfectly
well-behaved formal expansions, and the integral on the left becomes a
Koornwinder integral in the limit, so there is no difficulty in
analytically continuing in the dimension and dualizing. We do encounter
one more difficulty when specializing to finite dimension, however, as the
Koornwinder parameters are then~$\pm 1$, $\pm \sqrt{t}$, and thus we
encounter a pole of the lifted Koornwinder integral. Of course, we have
anticipated this problem, so need only apply Lemma~\ref{lem:koorn_On_disc}.
(This, in particular, explains why Conjecture~Q5 of~\cite{littlewood}
involved sums of two integrals.) In this way, we obtain the following
result, again stated in distributional form.

\begin{thm}\label{thm:Q5}
The dual Littlewood kernel has the ``distributional'' limits
\begin{gather*}
\int f(\vec{z}) \cL^{\prime(2n)}_{q^{-1/2}p^{-1/4}}(\vec{z};t;p,q) \Delta^{(2n)}_S(\vec{z};t ;p,q) \\
\qquad {}= \int f\big(p^{\pm 1/4}\vec{z}\big) \Delta^{(n)}_S\big(\vec{z};\pm 1,\pm t^{1/2};t;p^{1/2},q\big) \\
\qquad\quad {}+ \Gampqq(t,t)\Gamphq(-1,-t)\int f\big(p^{\pm 1/4}\vec{z},\pm p^{1/4}\big) \Delta^{(n-1)}_S\big(\vec{z};\pm t,\pm t^{1/2};t;p^{1/2},q\big)
\end{gather*}
and
\begin{gather*}
\int f(\vec{z}) \cL^{\prime(2n+1)}_{q^{-1/2}p^{-1/4}} (\vec{z};t;p,q) \Delta^{(2n+1)}_S(\vec{z};t;p,q) \\
\qquad {}= \Gamphq\big({-}1,\pm t^{1/2}\big) \int f\big(p^{\pm 1/4}\vec{z},p^{1/4}\big)
\Delta^{(n)}_S\big(\vec{z};t,-1,\pm t^{1/2};t;p^{1/2},q\big) \\
\qquad\quad {}+ \Gamphq\big({-}1,\pm t^{1/2}\big) \int f\big(p^{\pm 1/4}\vec{z},-p^{1/4}\big)
\Delta^{(n)}_S\big(\vec{z};1,-t,\pm t^{1/2};t;p^{1/2},q\big),
\end{gather*}
in the sense that Theorem~{\rm \ref{thm:quad_dual_litt}} and its corollaries continue to hold after the stated specialization.
\end{thm}

\begin{rem}
 This proves Conjecture Q5 of \cite{littlewood}. There is a corresponding
 evaluation coming from Theorem \ref{thm:van_dual_litt_interp}, but as
 this has two versions (depending on the parity of the dimension), each of
 which evaluates a sum of two integrals differing by a simple elliptic
 factor from an integral with an evaluation (so equivalent to a univariate
 sum), we omit the details.
\end{rem}

We now have two conjectures remaining, Q1 and Q4 of \cite{littlewood}. It
is straightforward to verify that these two conjectures (in symmetric
function kernel form) are dual to each other, so it will suffice to prove
one, say Q1 (which is slightly simpler). One natural approach is to follow
the development of \cite{littlewood} in reverse, and prove Q1 by a modular
transform from Q2 (i.e., Theorem \ref{thm:Q2} above). This requires a
suitable choice of algebraic degeneration of the identity, but it turns out
that the relevant special cases of Theorem \ref{thm:van_litt} are suitable
for that purpose, in that, although the corresponding kernel identity is
only a special case of the version of Theorem \ref{thm:quad_litt} we
require, it is sufficiently general that a couple of ``Bailey lemma''-type
steps suffice to prove the full version. (There would normally be a
difficulty, in that the first step of obtaining Theorem \ref{thm:van_litt}
from Theorem \ref{thm:quad_litt} was to specialize the auxiliary parameter
$v$, but that parameter happens to disappear in the special case of
interest.)

Rather than give the details of the above approach, we will take an alternate approach that, although it still takes some advantage of the
special structure of our particular case, has a~better chance of being adaptable to other special cases. The idea is that if we were dealing with
a special case that had a well-behaved formal series expansion, then it would be enough to show that the putative Littlewood kernel satisfied the
integral equation of Corollary~\ref{cor:spec_spec_int_litt} (or, more precisely, the corresponding special case of Theorem~\ref{thm:spec_int_litt}). Although our case is not formal, it turns out that we can finesse this issue, at the cost of having to prove the full
version of Theorem~\ref{thm:spec_int_litt}, which in our case becomes the following.

\begin{lem}
If $u_0u_1u_2u_3 = (pq/t)^4$, then
\begin{gather*}
\int K^{(2n)}_{t^{1/2}} \big({\pm}\sqrt{-\vec{z}};\vec{x};t;p,q\big)
\Delta^{(n)}_S\big(\vec{z};t,pt,qt,pqt,u_0,u_1,u_2,u_3;t^2;p^2,q^2\big) \\
\qquad {}= \prod_{\substack{1\le i\le 2n\\0\le r<4}} \Gampq\big({-}t u_r x_i^{\pm 2}\big) \\
\qquad \quad{}\!\times \int\!
K^{(2n)}_{t^{1/2}}\big({\pm}\sqrt{-\vec{z}};\vec{x};t;p,q\big)
\Delta^{(n)}_S\big(\vec{z};t,pt,qt,pqt,p^2q^2/t^2u_0,\dots,p^2q^2/t^2u_3;t^2;p^2,q^2\big).
\end{gather*}
\end{lem}

\begin{proof}
Since $K^{(2n)}_{t^{1/2}}$ can be written as a product, we find that
\begin{gather*}
K^{(2n)}_{t^{1/2}}\big({\pm}\sqrt{-\vec{z}};\vec{x};t;p,q\big)
\Delta^{(n)}_S\big(\vec{z};t,pt,qt,pqt;t^2;p^2,q^2\big)\\
\qquad{} = \frac{\prod\limits_{1\le i\le n,1\le j\le 2n} \Gamppqq\big({-}t x_j^{\pm 2}z_i^{\pm 1}\big)
\Delta^{(n)}_S\big(\vec{z};p^2q^2/t^2;p^2,q^2\big)}
{\Gampq(t)^{2n}\prod\limits_{1\le i<j\le 2n} \Gampq\big(t x_i^{\pm 1}x_j^{\pm 1}\big)},
\end{gather*}
and thus the given identity reduces to the main theorem of~\cite{vandeBultFJ:2009} (a.k.a.~the case $c=d=\sqrt{pq/t}$ of Corollary~\ref{cor:ker_comm} above).
\end{proof}

\begin{lem}\label{lemma8.13}
We have the following identity of meromorphic functions
\begin{gather*}
 \cL^{(2n)}_c (\vec{x};t;p,q) = \prod_{1\le i\le 2n}\Gamppqq\big(pq c^2 x_i^{\pm 2}\big)\\
 \hphantom{\cL^{(2n)}_c (\vec{x};t;p,q) =}{}\times
\int \cK^{(2n)}_{c/\sqrt{-t}}\big({\pm}\sqrt{-\vec{z}};\vec{x};t;p,q\big)
\Delta^{(n)}_S\big(\vec{z};t,pt,qt,pqt/c^4;t^2;p^2,q^2\big).
\end{gather*}
\end{lem}

\begin{proof}
We find that both sides have the same limit as $p\to 0$ with $0<\ord(c)\le
1/4$, and dividing by the common limit makes both Puiseux series have
rational function coefficients. Thus, by Lemma \ref{lem:int_uniq} below,
it will suffice to show that if we denote the right-hand side by~$F_c(\vec{x})$, then
\begin{gather*}
\prod_{1\le i\le 2n} \frac{1}{\Gampq\big(\sqrt{pqt}v^{\pm 1}x_i^{\pm 1}/c^2,\sqrt{pq/t}v^{\pm 1} x_i^{\pm 1}\big)}\\
\qquad{}\times \int
\cK^{(2n)}_{t^{1/2}}(\vec{z};\vec{x};t;p,q) F_c(\vec{z}) \Delta^{(2n)}_S\big(\vec{z};\sqrt{pq}v^{\pm 1}/c^2;t;p,q\big)
\end{gather*}
is independent of $v$. (This only determines the right-hand side up to a~scalar, but setting $x_i=t^{2n-i}v$ makes the right-hand side an elliptic
Selberg integral, so that we can explicitly evaluate it.) This is a~straightforward combination of commutation (Corollary~\ref{cor:ker_comm}) and the previous lemma.
\end{proof}

\begin{thm}\label{thm:Q1}
The Littlewood kernel has the ``distributional'' limit
\begin{gather*}
\int f(\vec{z}) \cL^{(2n)}_{\sqrt{-t}}(\vec{z};t;p,q)\Delta^{(2n)}_S(\vec{z};t;p,q)\\
\qquad {}= \int f\big({\pm}\sqrt{-\vec{z}}\big) \prod_{1\le i\le n} \Gampq\big(t z_i^{\pm 1}\big)
\Delta^{(n)}_S\big(\vec{z};t,pt,qt,pqt;t^2;p^2,q^2\big),
\end{gather*}
in the sense that Theorem~{\rm \ref{thm:quad_litt}} and its corollaries continue to hold after the stated specialization.
\end{thm}

\begin{proof} Lemma~\ref{lemma8.13} shows that this holds in the special case of Theorem~\ref{thm:litt_braid}. A~straightforward ``Bailey lemma'' step shows that the general case of Theorem~\ref{thm:quad_litt} holds as well.
\end{proof}

\begin{rem}
We again used the fact that the ``$v$'' parameter of Theorem~\ref{thm:quad_litt} disappears when $d=\sqrt{-t}$. This issue would need
to be worked around to make the above argument work for other values of~$d$, but most likely one could show that the known test functions span a~sufficiently large space to include the required test functions.
\end{rem}

This proves Conjecture Q1 of \cite{littlewood}; dualizing in the usual way
gives the following, which proves Conjecture Q4 of \cite{littlewood},
finishing the set. Note that as above, we obtain two cases depending on
the parity of the dimension.

\begin{thm}\label{thm:Q4}
The dual Littlewood kernel has the ``distributional'' limits
\begin{gather*}
\int f(\vec{x}) L^{\prime(2n)}_{q^{-1/2}\sqrt{-1}}(\vec{x};t;p,q) \Delta^{(2n)}_S(\vec{x};t;p,q) = \int f\big({\pm}\sqrt{-\vec{z}}\big)
\Delta^{(n)}_S\big(\vec{z};1,p,t,pt;t^2;p^2,q^2\big)
\end{gather*}
and
\begin{gather*}
\int f(\vec{x}) L^{\prime(2n+1)}_{q^{-1/2}\sqrt{-1}}(\vec{x};t;p,q) \Delta^{(2n+1)}_S(\vec{x};t;p,q) \\
\qquad {}= \Gamppqq(p,t,pt) \int f\big({\pm}\sqrt{-\vec{z}},\sqrt{-1}\big)
\Delta^{(n)}_S\big(\vec{z};t^2,p,t,pt;t^2;p^2,q^2\big),
\end{gather*}
in the sense that Theorem~{\rm \ref{thm:quad_dual_litt}} and its corollaries continue to hold after the stated specialization.
\end{thm}

Since it is quite straightforward to perform the required substitutions into the various corollaries of Theorems~\ref{thm:quad_litt},~\ref{thm:quad_dual_litt}, and~\ref{thm:quad_kaw}, we omit the details. The one exception is Corollary~\ref{cor:weird_quad} and its analogues for the dual Littlewood and Kawanaka kernels, where something interesting occurs.
Recall that Corollary~\ref{cor:weird_quad} gave a transformation between two integrals involving instances of the Littlewood kernel with differing
parameters. If we specialize the parameters so that both kernels
correspond to one of the special cases computed in this section, we obtain
new quadratic transformations. Curiously, if we do the same for the dual
Littlewood kernels instead, we find that most of the ``new''
transformations already appeared in the list corresponding to the
Littlewood kernel. This appears to be related to the relations
\begin{gather*}
L^{(2n)}_{(pqt)^{1/4}}(\vec{x};t;p,q) =\prod_{1\le i\le 2n} \frac{1}{\Gampqq\big((pq/t)^{1/2}x_i^{\pm 2}\big)}
L^{\prime(2n)}_{(p/qt)^{1/4}}(\vec{x};t;p,q),\\
 L^{-(n)}_{-(pq/t)^{1/4}}\big(\vec{x};-t^{1/2};p^{1/2},q^{1/2}\big)\\
 \qquad{} = \prod_{1\le i\le n} \frac{1}{\Gamphq\big((pqt)^{1/4} x_i^{\pm 1},-(pq/t)^{1/4}x_i^{\pm 1}\big)}
L^{\prime(n)}_{(p/qt)^{1/4}}(\vec{x};t;p,q),
\end{gather*}
which arise from the fact that the three kernels all have nearly the same
product expressions. If we substitute the first relation into the case
$d=(pqt)^{1/4}$ of Theorem~\ref{thm:litt_braid}, we obtain an expression
for the Littlewood kernel as an integral involving the {\em dual}
Littlewood kernel. The relevant Selberg density has four parameters, but
if $c^2=-t$ or $c^4=t^2/p$, two of the Gamma factors cancel to give an
expression for the dual Littlewood kernel. Now, this is in fact not a
legal substitution (since specializing $c$ in this way causes issues with
singularities), but this calculation at least suggests that the
corresponding instances of the Littlewood and dual Littlewood kernels
should be closely related. And, indeed, if we (for instance) compare
Theorems \ref{thm:Q1} and \ref{thm:Q2}, we see that the two
``distributions'' differ by simple univariate factors.

In any event, by taking all pairs of parameters coming from the above theorems, we obtain the following special cases of Corollary~\ref{cor:weird_quad} et al.

\begin{cor}
The integral{\samepage
\begin{gather*}\begin{split}&
\prod_{1\le i\le 2n} \Gampq\big(p^{3/4}q^{1/4}vx_i^{\pm 1}\big)\\
& \qquad{}\times \int
K^{(2n)}_{q^{1/4}c}\big(q^{\pm 1/4}\vec{z};\vec{x};t;p,q\big)
\Delta^{(n)}_S\big(\vec{z};(pq)^{1/4}v^{\pm 1}/c,t^{1/2},p^{1/2}t^{1/2};t;p,q^{1/2}\big)\end{split}
\end{gather*}
is invariant under swapping $p$ and $q$.}
\end{cor}

\begin{rem}
This comes from the case $d=q^{-1/4}t^{1/2}$, $e=p^{-1/4}t^{1/2}$ of
Corollary \ref{cor:weird_quad}. If we used the Kawanaka version instead,
we would normally expect to also obtain identities involving
odd-dimensional instances of the interpolation kernel, but it turns out
that the relevant prefactors vanish.
\end{rem}

\begin{cor}
The interpolation kernel satisfies the identities
\begin{gather*}
\int K^{(2n)}_{t^{-1/4}c}(\vec{z};\vec{x};t;p,q) \Delta^{(2n)}_S\big(\vec{z};p^{1/2}q^{1/4}v^{\pm 1}/c,t^{1/4},q^{1/2}t^{1/4};t^{1/2};p,q\big) \\
\qquad {}=
\prod_{1\le i\le 2n} \Gampq\big(\big(p^{1/2}q^{1/4}t^{-1/4}v^{\pm 1}\big)x_i^{\pm 1}\big) \\
\qquad\quad {}\times \int
K^{(2n)}_{q^{1/4}c}\big(q^{\pm 1/4}\vec{z};\vec{x};t;p,q\big)
\Delta^{(n)}_S\big(\vec{z};p^{1/2}(qt)^{1/4}v^{\pm 1}/c,1,t^{1/2};t;p,q^{1/2}\big)
\end{gather*}
and
\begin{gather*}
\int K^{(2n+1)}_{t^{-1/4}c}(\vec{z};\vec{x};t;p,q)
\Delta^{(2n+1)}_S\big(\vec{z};p^{1/2}q^{1/4}v^{\pm 1}/c,t^{1/4},q^{1/2}t^{1/4};t^{1/2};p,q\big) \\
\qquad {}= \Gampqh\big(t^{1/2},\sqrt{p}(qt)^{1/4}v^{\pm 1}/c\big)
\prod_{1\le i\le 2n+1} \Gampq\big(\big(p^{1/2}q^{1/4}t^{-1/4}v^{\pm 1}\big)x_i^{\pm 1}\big) \\
\qquad\quad{}\times \int
K^{(2n+1)}_{q^{1/4}c}\big(q^{\pm 1/4}\vec{z},q^{1/4};\vec{x};t;p,q\big)
\Delta^{(n)}_S\big(\vec{z};p^{1/2}(qt)^{1/4}v^{\pm 1}/c,t,t^{1/2};t;p,q^{1/2}\big).
\end{gather*}
\end{cor}

\begin{cor}
The interpolation kernel satisfies the identity
\begin{gather*}
\int
K^{(2n)}_{c\sqrt{-1}}\big({\pm}\sqrt{-\vec{z}};\vec{x};t;p,q\big)
\Delta^{(n)}_S\big(\vec{z};-pq^{1/2}v^{\pm 2}/c^2,t,qt;t^2;p^2,q^2\big)\\
\qquad {}= \prod_{1\le i\le 2n} \Gampq\big({-}\sqrt{-p}q^{1/4}v^{\pm 1}x_i^{\pm 1}\big)\\
\qquad\quad{}\times \int
K^{(2n)}_{q^{1/4}c}\big(q^{\pm 1/4}\vec{z};\vec{x};t;p,q\big)
\Delta^{(n)}_S\big(\vec{z};\sqrt{-p}q^{1/4}v^{\pm 1}/c,\pm\sqrt{t};t;p,q^{1/2}\big).
\end{gather*}
\end{cor}

\begin{cor}
The interpolation kernel satisfies the identity
\begin{gather*}
\int K^{(2n)}_{c\sqrt{-1}}\big({\pm}\sqrt{-\vec{z}};\vec{x};t;p,q\big)
\Delta^{(n)}_S\big(\vec{z};-pqt^{1/2}v^{\pm 2}/c^2,1,t;t^2;p^2,q^2\big)\\
\qquad {}= \prod_{1\le i\le 2n} \Gampq\big({-}\sqrt{-pq}t^{1/4}v^{\pm 1}x_i^{\pm 1}\big)\\
\qquad\quad{}\times \int
K^{(2n)}_{t^{-1/4}c}(\vec{z};\vec{x};t;p,q)
\Delta^{(2n)}_S\big(\vec{z};\sqrt{-pq}v^{\pm 1}/c,\pm t^{1/4};t^{1/2};p,q\big).
\end{gather*}
\end{cor}

If we count parameters in the above identities, we see that specializing $\vec{x}$ to a geometric progression of step $t$ gives in each case an
identity of elliptic Selberg integrals (both of which have evaluations). Thus these are only interesting in terms of the kernel (or interpolation
functions). One hope is that there may be analogues of Theorem~\ref{thm:van_litt} corresponding to the above identities, which might give
elliptic analogues (and nonzero values) for some of the remaining Macdonald polynomial results of~\cite{vanish}.

\appendix
\section{Appendix: Uniqueness of formal solutions}\label{sec:diff_uniq}

A key step in using integral or difference equations to determine various formal series was the fact that the limiting systems have no nonconstant polynomial solutions. As this requires some properties of interpolation polynomials \cite{OkounkovA:1998,bcpoly}, we address those statements in this appendix.

For the difference equations, we have the following.

\begin{lem}\label{lem:diff_uniq}
Let $q$, $t$ be generic, and let $u$ be such that $\prod\limits_{1\le i\le n}\big(1-t^{n-i}u\big)\ne 0$. Then for any nonconstant $BC_n$-symmetric
polynomial $f$,
\begin{gather*}
D^{(n)}_q(v,u/v;t)f
\end{gather*}
is a nonconstant function of $v$.
\end{lem}

\begin{proof} For any fixed $s$, we have an expansion of the form
\begin{gather*}
f(\vec{z}) = \sum_\mu c_\mu \bar{P}^{*(n)}_\mu(\vec{z};q,t,s),
\end{gather*}
where $\bar{P}^{*(n)}_\mu$ is Okounkov's interpolation polynomial~\cite{OkounkovA:1998} (in the notation of~\cite{bcpoly}). The limit $p\to 0$ of
Corollary~\ref{cor:int_eq_interp_ii} gives an expansion of the form
\begin{gather*}
D^{(n)}_q(v,u/v;t) \bar{P}^{*(n)}_\lambda(;q,t,s) \\
\qquad {}= \sum_{\mu\subset\lambda} s^{-|\lambda/\mu|}
C^0_{\lambda/\mu}\big(q^{1/2}t^{n-1}sv,q^{1/2}t^{n-1}su/v;q,t\big)
C^0_{(1^n+\mu)/\lambda}\big(t^{n-1}u;q,t\big) \\
\qquad\quad {} \times d_{\lambda\mu}(q,t) \bar{P}^{*(n)}_\mu\big(;q,t,q^{1/2}s\big),
\end{gather*}
where $d_{\lambda\mu}(q,t)$ is independent of $s$, $u$, $v$, and is nonzero only when $\lambda/\mu$ is a vertical strip. We thus see that the right-hand side is a~Laurent polynomial (symmetric under $v\mapsto u/v$) of degree $\le \ell(\lambda)$ in~$v$. Moreover, the only term that can
contribute to order $v^{\ell(\lambda)}$ is the one with $\mu=\lambda-1^{\ell(\lambda)}$, and the hypothesis ensures that this coefficient is nonzero. Now, among those $\lambda$ such that $c_\lambda\ne 0$, choose one with $\ell(\lambda)$ maximal. Then this term gives a~nonzero contribution to the coefficient of
\begin{gather*}
v^{\ell(\lambda)} \bar{P}^{*(n)}_{\lambda-1^{\ell(\lambda)}}\big(;q,t,q^{1/2}s\big)
\end{gather*}
in the output of the difference operator, while no other term can contribute to this coefficient. It follows that this coefficient is nonzero, and the result follows.
\end{proof}

\begin{rem}
The constraint on $u$ is equivalent to saying that $D_q(v,u/v;t)1\ne 0$.
\end{rem}

We also need the following, apparently weaker, system of equations.

\begin{cor}\label{cor:diff_uniq} Let $q$ and $t$ be generic, and $u$ such that $\prod\limits_{1\le i\le n}\big(1-t^{n-i}u^2\big)\ne 0$. Let $f(\vec{z})$ be a~$BC_n$-symmetric Laurent polynomial such that
\begin{gather*}
D^{(n)}_q\big(u\big(t^{1/2}v\big)^{\pm 1};t;p\big)f
\end{gather*}
is invariant under $v\mapsto 1/v$ as a polynomial in $v$. Then $f$ is constant.
\end{cor}

\begin{proof}
Let $D_v$ be the given operator. Since $D_v=D_{1/tv}$ as operators, we
conclude that $D_v f = D_{tv}f$ for all $v$, and thus $D_v f = D_{t^k v}$.
By Zariski density, we conclude that $D_v f$ is independent of $v$, and the
result follows.
\end{proof}

For the integral equations, essentially the same argument applies; the only
difference is that ``vertical strip'' is replaced by ``horizontal strip'',
and we must choose $\lambda$ to maximize $\lambda_1$ rather than
$\ell(\lambda)$. We obtain the following, where $I^{(n)}_t$ is the
integral operator defined in \cite{diffintrep_koorn}, which in present
terms may be viewed as the limit as $p\to 0$ of the integral operator with
kernel $K^{(n)}_{t^{1/2}}$. In particular, its action on interpolation
 polynomials again follows as the appropriate limit of
 Corollary \ref{cor:int_eq_interp_ii}.

\begin{lem}\label{lem:int_uniq}
Let $q$, $t$ be generic, and let $u$ be such that $(t^n u;q)\ne 0$. Then
for any nonconstant $BC_n$-symmetric Laurent polynomial $f(\vec{z})$,
\begin{gather*}
I^{(n)}_t(v,u/v;q)f
\end{gather*}
is a nonconstant function of~$v$.
\end{lem}

\begin{rem}Note that the excluded values of $u$ are precisely those for which the integral operator becomes singular.
\end{rem}

\subsection*{Acknowledgements} The author would particularly like to thank
P.~Etingof for an initial suggestion that taking $p$ to be a formal
variable might allow one to extend the $W(E_7)$ symmetry of the order 1
elliptic Selberg to $W(E_8)$; this turned out not to work (some symmetries
are, indeed, gained, but at the expense of others), but led the author to a
more general study of the formal limit. In addition, the author would like
to thank D.~Betea, M.~Wheeler, and P.~Zinn-Justin for discussions relating
to Izergin--Korepin determinants and their elliptic analogues, and
especially for discussions relating to Conjecture~1 of~\cite{BeteaD/WheelerM/Zinn-JustinP:2015} (which led the author to consider
the general case of the Littlewood kernel below). The author would also
like to thank O.~Warnaar for additional discussions related to the
Macdonald polynomial limit. The author would finally like to thank
H.~Rosengren for providing extra motivation to finish writing the present
work, as well as some helpful pointers to the vertex model literature. The
author was partially supported by the National Science Foundation (grant
number DMS-1001645).

\LastPageEnding

\end{document}